\documentclass[12]{amsart}
\usepackage[T1]{fontenc}
\usepackage[francais, english]{babel}
\usepackage{amsmath,amssymb,hyperref}
\usepackage{color,xcolor}
\usepackage[all]{xy}
\usepackage{geometry} 
\usepackage{array,multirow,makecell}
\usepackage{pst-all}
\usepackage{graphicx}

\newtheorem{thmA}{Theorem}
\newtheorem{thm}{Theorem}

\newtheorem{prop}[thm]{Proposition}
\newtheorem{lem}[thm]{Lemma}
\newtheorem{definition}[thm]{Definition}
\newtheorem{remark}[thm]{Remark}

\newcommand{\OM}[1]{\Omega^1_{#1}}

\DeclareMathOperator{\tr}{tr}

\def\C{\mathbb C}

\def\be{\boldsymbol{e}}

\def\OM0{\mathbf{\Omega}_0}

\newcommand{\cO}{{\mathcal O}}
\newcommand{\bH}{{\mathbf H}}

\newcommand{\cF}{{\mathcal F}}
\newcommand{\cE}{{\mathcal E}}

\newcommand\lra{\longrightarrow}

\newcommand\Spec{\mathop{\rm Spec}\nolimits}

\title[Canonical coordinates]{Canonical coordinates for moduli spaces of rank two irregular connections on curves.}
\author[A. Komyo, F. Loray, M.-H.\ Saito and Sz. Szab\'o]{Arata Komyo \and Frank Loray \and Masa-Hiko Saito \and Szil\'{a}rd Szab\'{o}}
\address{
Department of Material Science, Graduate School of Science, University of Hyogo, 2167 Shosha,
Himeji, Hyogo 671-2280, Japan}
\email{akomyo@sci.u-hyogo.ac.jp}
\address{Univ Rennes, CNRS, IRMAR - UMR 6625, F-35000 Rennes, France}
\email{frank.loray@univ-rennes.fr}
\address{Department of Data Science, Faculty of Business Administration, Kobe Gakuin University, Minatojima, Chuou-ku, Kobe,   650-8586, Japan}
\email{mhsaito@ba.kobegakuin.ac.jp}
\address{Institute of Mathematics, Faculty of Science, E\"otv\"os Lor\'and University, P\'azm\'any P\'eter s\'et\'any 1/C, Budapest, Hungary, H-1117; HUN-REN Alfr\'ed R\'enyi Institute of Mathematics, Re\'altanoda utca 13-15., Budapest 1053, Hungary}
\email{szilard.szabo@ttk.elte.hu, szabo.szilard@renyi.hu}
\subjclass[2020]{14H60 (primary), 14D20, 32G34, 53D30, 34M03 (secondary)}
\keywords{Moduli space of connections, cyclic vector, canonical coordinates}
\thanks{The first author is supported by JSPS KAKENHI: 
Grant Numbers 17H06127, 19K14506, and 24K06674.
The second author is supported by CNRS,  Centre Henri Lebesgue, program ANR-11-LABX-0020-01 and IsoMoDyn ANR-25-CE40-1360.
The third author is supported by JSPS KAKENHI:Grant Number 17H06127, 22H00094 and 22K18669.  
The fourth author was supported by the grants KKP 144148 and K 146401 of NKFIH}

\numberwithin{equation}{section}

\begin{document}

\begin{abstract}
In this paper, we study a geometric counterpart of the cyclic vector which 
allows us to put a rank 2 meromorphic connection on a curve into a ``companion'' normal form.
This allows us to naturally identify an open set of the moduli space of $\mathrm{GL}_2$-connections
(with fixed generic spectral data, i.e. unramified, non-resonant) with some Hilbert scheme of points on the twisted cotangent bundle
of the curve. We prove that this map is symplectic, therefore providing Darboux (or canonical) coordinates 
on the moduli space, i.e. separation of variables. On the other hand, for $\mathrm{SL}_2$-connections,
we give an explicit formula for the symplectic structure for a birational model given by Matsumoto.
We finally detail the case of an elliptic curve with a divisor of degree $2$. 
\end{abstract}

\maketitle

\section{Introduction}

In this paper, we introduce coordinates on the moduli spaces of 
rank 2 meromorphic connections on a Riemann surface, 
and we describe the symplectic structures on the moduli spaces 
by the introduced coordinates. 
Finally, we will have canonical coordinates on the moduli spaces.
Our motivation is to give explicit descriptions of 
the isomonodromic deformations of meromorphic connections over a general Riemann surface.
It is well-known that the isomonodromic deformations have
non-autonomous Hamiltonian descriptions 
(in detail, see \cite{Boalch07}, \cite{Krich02}, \cite{Hurt08}, \cite{Fedorov}, for example).
If we find explicit formulae for the isomonodromic Hamiltonians,
then we have explicit descriptions of isomonodromic deformations.
To find explicit formulae of Hamiltonians, 
it is necessary to
introduce canonical coordinates (which are also called Darboux coordinates)
on the moduli space of meromorphic connections.
The present paper is a first step to give explicit descriptions of 
isomonodromic deformations.

For the isomonodromic deformations of rank 2 projective connections
with regular singular points,
there are some results of explicit descriptions.
For example, 
Okamoto
considered non-autonomous Hamiltonian descriptions 
of isomonodromic deformations  
on elliptic curves
in \cite{Oka87} and \cite{Oka95}.
Iwasaki generalized
for general Riemann surfaces in \cite{Iwa91} and \cite{Iwa92}.
Here the independent variables of the isomonodromic deformations are 
the position of regular singular points on the Riemann surfaces.
That is, they are isomonodromic deformations of fixed Riemann surfaces.
On the other hand, Kawai \cite{Kawa03} gave explicit descriptions
of the isomonodromic Hamiltonians varying the elliptic curve.
Okamoto, Iwasaki, and Kawai in these papers
introduced canonical coordinates
on (a generic part of) the moduli space of rank 2 meromorphic projective connections
by using apparent singularities.
For our purpose, we take this strategy. 
That is, we will also introduce canonical coordinates
on the moduli space of meromorphic connections
by using apparent singularities.
On the other hand, in this paper, we are interested in 
the isomonodromic deformations of $\mathrm{GL}_2$-connections 
and 
of $\mathrm{SL}_2$-connections.
The coordinates using apparent singularities 
are an analog of 
the separation of variables in the Hitchin system,
which is a birational map from the moduli space of stable Higgs bundles
to the Hilbert scheme of points on the cotangent bundle over the underlying curve
of the Higgs bundles (see \cite{Hurt} and \cite{GNR}). 
This approach has been generalized to Higgs bundles with unramified irregular singularities~\cite[Section~8.3]{Kon-Soi},~\cite{Szabo_BNR},~\cite[Section~4.2]{Kon-Ode}. 
The moduli spaces of Higgs bundles corresponding to the Painlev\'e cases were analyzed from this perspective in~\cite{ISS1},~\cite{ISS2},~\cite{ISS3}. 
Here this map is a symplectomorphism of the open dense subsets of the moduli space.
The definition of the apparent singularities for general rank meromorphic connections
will be given in \cite{SS}.

\subsection{Our setting}
Let $\nu$ be a positive integer.
We set $I:= \{ 1,2,\ldots,\nu\}$.
Let $C$ be a compact Riemann surface of genus $g$ ($g\geq 0$),
and $D= \sum_{i \in I} m_i[t_i]$ be an effective divisor on $C$.
Let $E$ be a vector bundle over $C$
and $\nabla \colon E \rightarrow E \otimes \Omega^1_{C}(D)$
be a meromorphic connection acting on $E$.
We assume that the leading term of 
the expansion of a connection matrix of $\nabla$ at $t_i$
has distinct eigenvalues.
If $m_i=1$, then we assume that 
the difference of eigenvalues of the residue matrix at $t_i$ 
is not an integer.
That is, $t_i$ is a generic unramified irregular singular point of $\nabla$
or a non-resonant regular singular point of $\nabla$.

When $C$ is the projective line and $E$ is the trivial bundle,
the moduli space of meromorphic connections
has been studied by Boalch \cite{Boalch01}
and Hiroe--Yamakawa \cite{HY14}.
This moduli space has the natural symplectic structure 
coming from the symplectic structure on the (extended) coadjoint orbits. 
For a general $C$, 
the moduli space of meromorphic connections (with quasi-parabolic structures)
has been studied by
Inaba--Iwasaki--Saito \cite{IIS, IIS2},
Inaba \cite{Ina},
and
Inaba--Saito \cite{IS} from the algebraic geometric point of view.
From the differential geometric point of view,
Boalch \cite[Section 4]{Boalch01} and
Biquard--Boalch \cite{BB} have studied 
the moduli space of meromorphic connections
for a generic $C$; in particular, 
the latter paper endows the moduli space with a hyperk\"ahler structure.
For a general $C$, 
the moduli space has also the natural symplectic structure. 
In the papers \cite{IIS, Ina, IS},
the symplectic form is described by
a pairing of the hypercohomologies of a complex.
This description of the symplectic structure is an analog of 
the symplectic structure of the moduli spaces of stable Higgs bundles
due to Bottacin \cite{Bott}.
Note that it is not immediately clear that this agrees with the symplectic structure defined in \cite{Boalch01}; however, this was shown to be the case in \cite[Proposition~3.10]{Wong}.
For the case where $\nabla$ has only regular singular points,
Inaba showed that 
this symplectic structure coincides with the pull-back of the 
Atiyah--Bott--Goldman symplectic structure on the character variety
via the Riemann--Hilbert map
in \cite[the proof of Proposition 7.3]{Ina}.

Our purpose in this paper is to introduce
canonical coordinates on the moduli spaces of meromorphic connections. 
For this purpose, there are some strategies.
The first one is to consider canonical coordinates
on the products of coadjoint orbits.
This direction was studied by 
Jimbo--Miwa--Mori--Sato \cite{JMMS},
Harnad \cite{Harnad94}, and
Woodhouse \cite{Woodhouse07}.
Sakai--Kawakami--Nakamura \cite{KNS18}
and 
Gaiur--Mazzocco--Rubtsov \cite{GMR23}
gave some explicit formulae for the isomonodromic Hamiltonians 
by the coordinates of this direction.
The second one is to consider the apparent singularities.
As mentioned above, 
we take this strategy.

In this paper, 
we consider only the case where the rank of $E$ is two.
Let $X$ be an irregular curve,
which is described in Section \ref{SubSect:SpectralData}.
That is, $X$ is a tuple of (i) a compact Riemann surface $C$, 
(ii) an effective divisor $D$ on $C$, 
(iii) local coordinates around 
the support with $D$, and 
(iv) spectral data of meromorphic connections at 
the support with $D$ (with data of residue parts).
Here, the spectral data is described in Section \ref{SubSect:SpectralData}.
We fix an irregular curve $X$.
That is, we fix spectral data 
of rank 2 meromorphic connections at each point of the support with $D$.
By applying elementary transformations (which are also called Hecke modifications),
we may change the degree of the underlying vector bundle of a meromorphic connection 
freely. 
So we assume that $\deg(E)=2g-1$.
By this condition, the Euler characteristic of the vector bundle $E$ is $1$
by the Riemann--Roch theorem.
In this situation, for generic meromorphic connections $(E,\nabla)$, 
we have $\dim_{\mathbb{C}} H^0(C,E)=1$.
So the global section of $E$ is uniquely determined up to a constant.
This is convenient for the definition of the apparent singularities. 
In this paper, 
we consider only meromorphic connections with $\dim_{\mathbb{C}} H^0(C,E)=1$.
Moreover we assume that 
meromorphic connections $(E,\nabla)$ are irreducible. 
By this condition, the definition of apparent singularities becomes simple.

\subsection{$\mathrm{GL}_2$-connections}
In the first part of this paper, we discuss $\mathrm{GL}_2$-connections.
That is, we consider rank 2 meromorphic connections.
We do not fix the determinant bundles of the underlying vector bundles 
and the traces of connections.
Our purpose is to introduce canonical coordinates 
on the moduli space of rank 2 meromorphic connections
by using apparent singularities. 
When $C$ is the projective line, 
many people have introduced canonical coordinates on the moduli space
by using the apparent singularities 
(\cite{Oka86}, \cite{Obl05}, \cite{DM07}, \cite{Szabo13}, \cite{KS19},
\cite{DiarraLoray}, and \cite{Kom}).
In this paper, we consider apparent singularities for general Riemann surfaces.

Let $X$ be the fixed irregular curve.
If $(E,\nabla)$ is a rank 2 meromorphic connection 
such that $\deg(E) =2g-1$, $\dim_{\mathbb{C}} H^0(C,E)=1$, 
and $(E,\nabla)$ is irreducible,
then 
we can define apparent singularities for $(E,\nabla)$.
(In detail, see Definition \ref{2023_8_22_13_40} below).
The apparent singularities are the set of points 
$\{q_1,\ldots,q_N\}$ on the underlying curve $C$.
Here we set $N := 4g-3+\deg(D)$.
Let $M_X$ be the following moduli space 
$$
M_X := 
\left. \left\{ 
(E,\nabla )
\ \middle| \ 
\begin{array}{l}
\text{(i) $E$ is a rank 2 vector bundle on $C$ with $\deg(E)=2g-1$}\\
\text{(ii) $\nabla \colon E \rightarrow E \otimes \Omega^1_C(D)$ is a connection} \\
\text{(iii) $(E,\nabla)$ is irreducible, and}\\
\text{(iv) $\nabla$ has the fixed spectral data in $X$}
\end{array}
\right\} \right/  \cong.
$$
This moduli space $M_X$ has 
a natural symplectic structure 
due to Inaba--Iwasaki--Saito \cite{IIS}, Inaba \cite{Ina}, and Inaba--Saito \cite{IS}.
We consider a Zariski open subset $M_X^0$ of $M_X$ as follows:
$$
M^0_X
:= \left. \left\{ 
(E,\nabla ) \in M_X
\ \middle| \ 
\begin{array}{l}
\text{(i) $\dim_{\mathbb{C}} H^0(C,E) =1$}, \\
\text{(ii) $q_1+\cdots+q_{N}$ is reduced, and} \\
\text{(iii) $q_1+\cdots+q_{N}$ has disjoint support with $D$}
\end{array}
\right\} \right/  \cong
$$
(in detail, see Section \ref{subsect_moduli}).
The dimension of the moduli space $M^0_X$ is $2N$ (Proposition~\ref{prop:dimension}).
By taking apparent singularities, we have a map
$$
\begin{aligned}\label{eq:App}
\operatorname{App}\colon M^0_X &\longrightarrow  \mathrm{Sym}^N(C) \\
(E,\nabla) &\longmapsto \{ q_1,q_2 , \ldots ,q_N \} .
\end{aligned}
$$
Remark that 
the dimension of $\mathrm{Sym}^N(C)$
is half of the dimension of $M^0_X$.
To introduce coordinates on $M^0_X$,
it is necessary to find further invariants of connections, 
that are customarily called \emph{accessory parameters}.

To find these parameters, we introduce a companion normal form
for $(E,\nabla)$ as follows.
By the theory of apparent singularities discussed in Section \ref{SubSect:AppGL2},
we have a canonical inclusion morphism
$$
\mathcal{O}_C \oplus (\Omega^1_C(D))^{-1}
\longrightarrow E.
$$
By this morphism, 
we have the connection $\nabla_0$ on $\mathcal{O}_C \oplus (\Omega^1_C(D))^{-1}$
induced by a connection $\nabla$ on $E$.
Notice that $\nabla_0$ has simple poles at the apparent singularities.
By applying automorphisms on $\mathcal{O}_C \oplus (\Omega^1_C(D))^{-1}$,
we may normalize $\nabla_0$ as
$$
\nabla_0= 
\begin{pmatrix}
\operatorname{d} & \beta \\ 
1 & \delta
\end{pmatrix},
$$
which is called a companion normal form 
(in detail, see Section \ref{SubSect:CNFGL2} below).
Here $\operatorname{d}$ is the exterior derivative on $C$,
$\beta \in H^0(C, (\Omega^1_C)^{\otimes 2} (2D + q_1+\cdots+q_N))$,
and $\delta$ is a connection on $(\Omega^1_C(D))^{-1}$,
which has poles at the support of $D$ and
simple poles at the apparent singularities $q_1,\ldots,q_N$.
Here the residues of $\delta$ at each $q_1,\ldots,q_N$
are $+1$.
Let $z_j$ be a local coordinate chart of $C$ around $q_j$.
For an apparent singularity $q_j$, we define the accessory parameter by 
$$
\mathrm{res}_{q_j} (\beta) 
+ \left(  \delta - \frac{\operatorname{d}\! z_j}{z_j -z_j(q_j) } \right) |_{q_j},
$$
which is an element of $\Omega^1_C(D)|_{q_j}$.
However,
to define this accessory parameter,
we need to take a trivialization of $ (\Omega^1_C(D))^{-1}$
around $q_j$ and a defining function $z_j -z_j(q_j)$ of $q_j$.
So the accessory parameter depends on the choices of 
a trivialization of $(\Omega^1_C(D))^{-1}$ 
and a defining function of $q_j$.
Hence, we consider a twist of $\Omega^1_C(D)$
(In detail, Section \ref{subsect:Canonical_Coor} below).

We denote by $\Omega_C^1(L_{\delta}, D)$ the twist of $\Omega^1_C(D)$
defined in \eqref{2025_10_22_11_35} below.
Let 
$$
\pi_{C} \colon \boldsymbol{\Omega}(L_{\delta}, D) \longrightarrow C
$$
be the total space of $\Omega^1_C(L_{\delta}, D)$.
Let $\omega_{L_\delta, D}$ be the rational 2-form on 
$\boldsymbol{\Omega}(L_{\delta}, D)$
induced by the Liouville symplectic form.
This rational 2-form $\omega_{L_\delta, D}$ 
induces a symplectic structure on 
$\boldsymbol{\Omega}(L_{\delta}, D) \setminus \pi_{C}^{-1}(D)$.
We consider the symmetric product 
$\mathrm{Sym}^N(\boldsymbol{\Omega}(L_{\delta}, D))$.
Let $\sum_{j=1}^N \mathrm{pr}_j^* (\omega_{L_\delta, D})$ be the rational 2-form 
on the product $\boldsymbol{\Omega}(L_{\delta}, D)^N$.
Here $\mathrm{pr}_j \colon \boldsymbol{\Omega}(L_{\delta}, D)^N 
\rightarrow \boldsymbol{\Omega}(L_{\delta}, D)$ is the $j$-th projection.
This rational 2-form $\sum_{j=1}^N \mathrm{pr}_j^*( \omega_{L_\delta, D})$
induces a symplectic structure on a generic part of 
$\mathrm{Sym}^N(\boldsymbol{\Omega}(L_{\delta}, D))$.
We will define a map from $M^0_X$ 
to $\mathrm{Sym}^N(\boldsymbol{\Omega}(L_\delta, D))$ by
\begin{equation}\label{2023_8_23_8_9}
\begin{aligned}
f_{\mathrm{App}} \colon M^0_X &\longrightarrow 
\mathrm{Sym}^N(\boldsymbol{\Omega}(L_{\delta}, D)) \\
(E,\nabla) &\longmapsto \{ 	
(q_j, \mathrm{res}_{q_j} (\beta) 
+ (  \delta - \frac{\operatorname{d}\! z_j}{z_j -z_j(q_j) } ) |_{q_j} )\}_{
1\leq j \leq N}.
\end{aligned}
\end{equation}
Here, notice that $\mathrm{res}_{q_j} (\beta) \in \Omega^1_C(D)|_{q_j}$
and $(  \delta - \frac{\operatorname{d}\! z_j}{z_j -z_j(q_j) } ) |_{q_j}$ 
is justified by considering 
the twisted cotangent bundle
(after Remark \ref{2025_10_22_11_42} below, we will discuss it).
Remark that 
the dimension of $\mathrm{Sym}^N(\boldsymbol{\Omega}(L_{\delta}, D))$
is equal to the dimension of $M^0_X$.
The first main theorem is the following:

\begin{thmA}[Theorem \ref{2023_8_22_12_09} below]\label{TheoremA}
The pull-back of the symplectic form on a generic part of 
$\mathrm{Sym}^N(\boldsymbol{\Omega}(L_{\delta}, D))$
under the map \eqref{2023_8_23_8_9}
coincides with the symplectic form on $M^0_X$.
\end{thmA}

If we take canonical coordinates on $\boldsymbol{\Omega}(L_{\delta}, D)$,
then we have canonical coordinates on $\mathrm{Sym}^N(\boldsymbol{\Omega}(L_{\delta}, D))$,
since the symplectic structure on 
$\mathrm{Sym}^N(\boldsymbol{\Omega}(L_{\delta}, D))$
is induced by the 2-form $\sum_{j=1}^N \mathrm{pr}_j^*( \omega_{L_{\delta}, D})$.
Then we have canonical coordinates on $M^0_X$
by Theorem \ref{TheoremA}. 
Indeed, we define $p_j \in \mathbb{C}$ by 
$$
\mathrm{res}_{q_j} (\beta) 
+ (  \delta - \frac{\operatorname{d}\! z_j}{z_j -z_j(q_j) } ) |_{q_j} = p_j \operatorname{d}\!z_j,
$$
and we denote $z_j(q_j)$ by the same symbol $q_j$.
From $(E,\nabla)$,
we have a tuple $( q_j ,p_j)_{j=1,2,\ldots,N} \in \mathbb{C}^{2N}$.
We will show that $f_{\mathrm{App}}$ is birational in Proposition \ref{prop:birational} below.
So we can consider the tuple $( q_j ,p_j)_{j=1,2,\ldots,N}$ as local coordinates 
on $M^0_X$.
The symplectic form on $M^0_X$ becomes 
$$
dp_1 \wedge dq_1 + \cdots  + dp_N \wedge dq_N.
$$

In Section \ref{Sect:CompForElliptic}, 
we consider an example of this argument. 
We will calculate the canonical coordinates 
for an elliptic curve and a divisor $D$ of length $2$.
The moduli space of rank 2 meromorphic connection with fixed trace connection
on an elliptic curve with two simple poles
was studied in \cite{LR20} and \cite{FL}.
In this paper, we will discuss the $\mathrm{GL}_2$-connection case.

\subsection{$\mathrm{SL}_2$-connections}
In the second part of this paper, we discuss $\mathrm{SL}_2$-connections.
That is, we consider rank 2 meromorphic connections 
with fixed trace connection $(L_0,\nabla_{L_0})$.
Here $L_0$ is a fixed line bundle on $C$ of degree $2g-1$
and $\nabla_{L_0} \colon L_0 \rightarrow L_0 \otimes \Omega^1_C(D)$
is a fixed connection.
More precisely, we consider 
rank 2 quasi-parabolic connections $(E,\nabla , \{ l^{(i)}\} )$,
defined in \cite[Definition 2.1]{IS},
with fixed trace connection $(L_0,\nabla_{L_0})$.
Here the spectral data of $\nabla_{L_0}$ is determined by the fixed irregular curve $X$.
The quasi-parabolic structure $l^{(i)}$ at $t_i$
induces a one dimensional subspace $l^{(i)}_{\mathrm{red}}$ of $E|_{t_i}$, that is the restriction
of $l^{(i)}$ to $t_i$ (without multiplicity).
Our moduli space is as follows:
$$
M_X(L_0,\nabla_{L_0})_0
:= \left. \left\{ 
(E,\nabla , \{ l^{(i)}\} )
\ \middle| \ 
\begin{array}{l}
\text{(i) $\nabla$ has the fixed spectral data in $X$,}\\
\text{(ii) $E$ is an extension of $L_0$ by $\mathcal{O}_C$,}\\
\text{(iii) $\dim_{\mathbb{C}} H^0(C,E) =1$, and} \\
\text{(iv) $l_{\text{red}}^{(i)}  \not\in \mathcal{O}_{C}|_{t_i}
\subset \mathbb{P}(E)$ for any $i$} 
\end{array}
\right\} \right/  \cong,
$$
which is described in Section \ref{SubSect:ModuliSL2}.
Here $(E,\nabla , \{ l^{(i)}\} )$ are rank 2 quasi-parabolic connections on $(C,D)$ 
with fixed trace connection $(L_0,\nabla_{L_0})$.
When $g=0$, we impose one more condition (in detail, see
the paragraph after the proof of 
Lemma \ref{2023_7_10_12_15} below).
This moduli space also has 
a natural symplectic structure.
The dimension of the moduli space $M_X(L_0,\nabla_{L_0})_0$ is $2N_0$,
where $N_0 := 3g-3+\deg(D)$.
For $(E,\nabla , \{ l^{(i)}\} ) \in M_X(L_0,\nabla_{L_0})_0$,
we can also define apparent singularities (Section \ref{SubSect:MapsSL2} below).
The apparent singularities give an element of 
$\mathbb{P}H^0(C, L_0\otimes \Omega_C^1(D))$.
So we have a map 
$$\pi_{\mathrm{App}} \colon 
M_X(L_0,\nabla_{L_0})_0 \longrightarrow 
\mathbb{P}H^0(C, L_0\otimes \Omega_C^1(D)).
$$ 
For $(E,\nabla , \{ l^{(i)}\} ) \in M_X(L_0,\nabla_{L_0})_0$,
we forget the connection $\nabla$.
So we have a quasi-parabolic bundle $(E, \{ l^{(i)}\} )$.
By taking the extension class for the quasi-parabolic bundle $(E, \{ l^{(i)}\} )$,
we have a map
$$
\pi_{\mathrm{Bun}} \colon 
M_X(L_0,\nabla_{L_0})_0 \longrightarrow 
\mathbb{P}H^1(C, L_0^{-1}(-D)).
$$
Here the extension class is described in 
Section \ref{SubSect:ModuliParaBunSL2} below.
We consider the product 
$$
\pi_{\mathrm{App}} \times \pi_{\mathrm{Bun}} \colon 
M_X(L_0,\nabla_{L_0})_0 \longrightarrow 
\mathbb{P}H^0(C, L_0\otimes \Omega_C^1(D))
\times \mathbb{P}H^1(C, L_0^{-1}(-D)).
$$ 
This map has been studied by
Loray--Saito--Simpson \cite{LSS},
Loray--Saito \cite{LS},
Fassarella--Loray \cite{FL},
Fassarella--Loray--Muniz \cite{FLM}, and
Matsumoto \cite{Matsu}.

Notice that 
$H^1(C, L_0^{-1}(-D))$ is isomorphic to 
the dual of $H^0(C, L_0\otimes \Omega_C^1(D))$.
Remark that 
$$
\dim_{\mathbb{C}} \mathbb{P} H^0(C, {L_0} \otimes \Omega^1_C(D))
= \dim_{\mathbb{C}} \mathbb{P} H^1(C, L_0^{-1}(-D))
=N_0.
$$
Let us introduce the
homogeneous coordinates
$\boldsymbol{a} = (a_0 :\cdots : a_{N_0})$
on 
$\mathbb{P} H^0(C, {L_0} \otimes \Omega^1_C(D))\cong \mathbb{P}_{\boldsymbol{a}}^{N_0}$
and the dual coordinates 
$\boldsymbol{b} = (b_0 :\cdots : b_{N_0})$
on 
$$
\mathbb{P} H^1(C, L_0^{-1}(-D))\cong 
\mathbb{P} H^0(C, {L_0}\otimes \Omega^1_C(D))^{\vee} \cong \mathbb{P}_{\boldsymbol{b}}^{N_0}.
$$
We may define a $1$-form $\eta$ on 
$\mathbb{P}_{\boldsymbol{a}}^{N_0} \times \mathbb{P}_{\boldsymbol{b}}^{N_0}$
by 
$$
\eta = \left( \text{constant} \right) \cdot
\frac{a_0 \, d b_0 + a_1 \, d  b_1 
+ \cdots + a_{N_0} \, d b_{N_0}
}{a_0b_0 + a_1b_1 + \cdots + a_{N_0}b_{N_0}}.
$$
(In detail, see Section \ref{SubSect:SyplecticSL2}).
The $2$-form $d \eta$ gives a symplectic structure
on $\mathbb{P}_{\boldsymbol{a}}^{N_0} \times \mathbb{P}_{\boldsymbol{b}}^{N_0} 
\setminus \Sigma$. 
Here we set 
$$
\Sigma \colon  (a_0b_0 + a_1b_1 + \cdots + a_{N_0}b_{N_0}=0) 
\subset \mathbb{P}_{\boldsymbol{a}}^{N_0} \times \mathbb{P}_{\boldsymbol{b}}^{N_0}.
$$
The image of $M(L_0,\nabla_{L_0})_0$
is contained in $\mathbb{P}_{\boldsymbol{a}}^{N_0} \times \mathbb{P}_{\boldsymbol{b}}^{N_0} 
\setminus \Sigma$.
(In detail, see Section \ref{SubSect:MapsSL2}).
The second main theorem is the following:
\begin{thmA}[Theorem \ref{2023_8_22_16_22} below]
We assume that the fixed spectral data satisfies
the generic condition \eqref{2023_8_24_8_47} below.
The pull-back of the symplectic form $d\eta$ on 
$\mathbb{P}_{\boldsymbol{a}}^{N_0} \times \mathbb{P}_{\boldsymbol{b}}^{N_0}\setminus \Sigma$
under the map 
$\pi_{\mathrm{App}} \times \pi_{\mathrm{Bun}}$
coincides with the symplectic form on the moduli space $M_X(L_0,\nabla_{L_0})_0$.
\end{thmA}

\subsection{The organization of this paper}
In Section \ref{sect:CompanionNF},
the apparent singularities for a generic rank 2 meromorphic connection are defined.
After the definition of the apparent singularities,
we will discuss on the companion normal form of a
generic rank 2 meromorphic connection.
We will use this companion normal form when we will introduce 
canonical coordinates.
In Section \ref{Sect:SymplecticGL2andCano},
first, we will describe our moduli space of 
rank 2 meromorphic connections.
Second, we will discuss on tangent spaces of 
the moduli space of 
rank 2 meromorphic connections.
We will recall that the tangent spaces at a meromorphic connection 
are isomorphic to a hypercohomology 
of the complex defined by the meromorphic connection.
After that, we will describe a natural 
symplectic structure on the moduli space of 
rank 2 meromorphic connections.
Section \ref{2023_7_4_13_59}
and Section \ref{2023_7_12_16_39}
are preliminaries of the proof of the first main theorem.
In Section \ref{subsect:Canonical_Coor},
we will give the
map from a generic part of the moduli space to
$\mathrm{Sym}^N(\boldsymbol{\Omega}(L_{\delta}, D))$
and will show the first main theorem.

In Section \ref{sect:sympl_FixedDet},
we will consider rank 2 meromorphic connections 
with fixed trace connection.
First, to describe the bundle map $\pi_{\mathrm{Bun}}$,
we recall the moduli space of stable quasi-parabolic bundles
with fixed determinant.
Second, 
we will describe our moduli space of 
rank 2 meromorphic connections 
with fixed trace connection.
Third, we will describe the map $\pi_{\mathrm{App}}$
defined by considering the apparent singularities.
In Section \ref{SubSect:SyplecticSL2},
we will recall a natural 
symplectic structure on the moduli space of 
rank 2 meromorphic connections
with fixed trace connection,
and will show the second main theorem.

In Section \ref{Sect:CompForElliptic}, 
we will apply the argument in 
Section \ref{sect:CompanionNF}
and 
Section \ref{Sect:SymplecticGL2andCano}
to the case of an elliptic curve with a divisor $D$ of length $2$. 
When $D$ is reduced, this amounts to two logarithmic singularities, otherwise to an irregular singularity. 
It is remarkable that using our approach these two cases can be studied completely similarly.

In Section \ref{Sec:Higgs}, we will provide a method 
for obtaining canonical coordinates 
$\tilde{p}_j  \in \boldsymbol{\Omega}(L_{\delta}, D)_{|q_i}$ for generic $(E, \nabla) 
\in M^0_X$  by introducing a section $s \in H^0(C, \det(E))$ 
and $\gamma \in H^0(C, \Omega^1_C(D))$. We will utilize an open set 
$U_{0} = C \setminus \{ s=0, \gamma=0 \}$ and the trivialization of 
$E_{|U_0}$ to define $\tilde{p}_j \in \Omega^1_C(D)_{|q_j}$.
This method can be also used for  constructing a meromorphic connection 
$\nabla_1\colon  E \longrightarrow E \otimes \Omega^1_C(D(s))$ 
for a given $s \in H^0(C, \det(E))$, where $D(s)$ denotes the zero divisor of $s$. 
In Theorem \ref{thm:birational},  we will provide an alternative proof of the birationality of $f_{\mathrm{App}}$ (cf.  Proposition \ref{prop:birational}) by utilizing the Higgs fields 
$\nabla - \nabla_1$ and the BNR correspondence \cite{BNR}. 
This approach may shed new light on the relationship 
between the canonical coordinates of the moduli spaces of connections 
and the moduli spaces of Higgs bundles. (cf.  \cite{SS}).

\subsection*{Acknowledgments}
The authors would like to warmly thank Philip Boalch, Michi-aki Inaba
and Takafumi Matsumoto
for useful discussions. 
The first, third, and fourth authors would like to thank 
Frank Loray for his hospitality at IRMAR, Univ. Rennes.
The authors would also like to warmly thank the anonymous referee
for many helpful comments, suggestions and corrections.

\section{Companion normal form}\label{sect:CompanionNF}

Let $C$ be a compact Riemann surface of genus $g$ ($g\geq 0$),
and $D$ be an effective divisor on $C$.
We assume $4g-3+n>0$ where $n=\deg(D)$.
We consider a rank 2 meromorphic connection
\begin{equation}\label{2023_7_13_12_11}
\nabla:E\longrightarrow E\otimes\Omega^1_C(D)
\end{equation}
on $C$ where $\deg(E)=2g-1$.

When 
$g=0$,
Diarra--Loray
have given companion normal forms
of the rank 2 meromorphic connections in
\cite{DiarraLoray}.
By the companion normal forms,
we may construct a universal family of the rank 2 meromorphic connections
on some generic part of the moduli space of rank 2 meromorphic connections.
This universal family is useful to describe the isomonodromic deformations \cite{Kom}.
The purpose of this section is 
to give companion normal forms
of rank 2 meromorphic connections when $g\geq 0$.
For this purpose, first, we will introduce the apparent singularities for 
(generic) rank 2 meromorphic connections.

\subsection{Apparent singularities}\label{SubSect:AppGL2}
First we assume that $\dim_{\mathbb{C}} H^0(C,E)=1$ for the rank 2 meromorphic connection
\eqref{2023_7_13_12_11}.
This assumption holds
for a generic vector bundle of the rank 2 meromorphic connection
with $\deg(E)=2g-1$.
For an element of $H^0(C,E)$, 
we define 
the sequence of $\mathbb C$-linear maps 
\begin{equation}\label{2023_7_11_12_46}
\varphi_\nabla \colon \mathcal O_C \longrightarrow
 E \xrightarrow{\ \nabla \ } E\otimes\Omega^1_C(D) 
 \longrightarrow E/ \mathcal O_C \otimes\Omega^1_C(D) .
 \end{equation}
This composition $\varphi_\nabla$ is an $\mathcal O_C$-linear map.
From now on we assume that $\varphi_\nabla \neq 0$. 
This assumption holds for every $(E,\nabla)$, 
provided that the eigenvalues of the residues are chosen generically 
(see Remark \ref{2023_7_13_16_34} below). 
We call the global section in $H^0(C,E)$ in \eqref{2023_7_11_12_46} 
the {\it cyclic vector}.

Let us now define $E_0\subset E$ as the rank $2$ locally free subsheaf 
spanned by $\mathcal O_C$ and 
$$
\operatorname{Im}\left\{\nabla|_{\mathcal O_C} \otimes 
\operatorname{Id}_{(\Omega^1_C(D))^{-1}} 
\colon (\Omega^1_C(D))^{-1} \to E \right\}. 
$$ 
This construction gives rise to a short exact sequence of coherent sheaves 
$$
    0 \longrightarrow
     \mathcal O_C \longrightarrow
      E_0 \longrightarrow
       (\Omega^1_C(D))^{-1} \longrightarrow 0.
$$
We claim that this sequence splits, i.e. 
\begin{equation}\label{eq:decomposition}
 E_0\cong \mathcal O_C\oplus (\Omega^1_C(D))^{-1}. 
\end{equation}
Indeed, equivalence classes of extensions of $(\Omega^1_C(D))^{-1}$ by 
$\mathcal O_C$ are classified by the group 
$$
    \operatorname{Ext}^{1} ((\Omega^1_C(D))^{-1} , \mathcal O_C) = 
    \operatorname{Ext}^{1} (\mathcal O_C (-D) , \Omega^1_C) \cong 
    H^0 (C, \mathcal O_C (-D) )^{\vee} = 0,  
$$
where we have used Grothendieck--Serre duality. 
We denote by 
\begin{equation}\label{eq:phi_nabla}
 \phi_\nabla \colon E_0  \longrightarrow E. 
\end{equation}
the canonical inclusion morphism, and define the meromorphic connection 
\begin{equation}\label{2023_7_13_12_39}
 \nabla_0 = \phi_\nabla^* (\nabla) 
\end{equation}
on $E_0$. 
We note that the polar divisor of $\nabla_0$ is $D+B$ where 
\begin{equation}\label{eq:B}
 B=\mathrm{div}(\varphi_\nabla ). 
\end{equation}
We note that
\begin{equation}\label{eq:deg(B)}
 \deg(B)=4g-3+n.
\end{equation}
From now on, moreover, we assume that $B$ is reduced, with support disjoint from $D$. 
In different terms, in view of~\eqref{eq:deg(B)}, we have 
$$
B=q_1+\cdots+q_{4g-3+n},
$$
where $q_i \neq q_j$ once $i\neq j$ and $q_i \notin D$ for all $i$.

\begin{definition}\label{2023_8_22_13_40}
Assume that $\varphi_\nabla \neq 0$ and 
$\mathrm{div}(\varphi_\nabla )$ is reduced, with support disjoint from $D$.
We call
the points of the support $\{ q_1,\ldots,q_{4g-3+n}\}$ of $\mathrm{div}(\varphi_\nabla )$
the {\rm apparent singularities of $(E,\nabla)$}.
\end{definition}

\subsection{Companion normal form}\label{SubSect:CNFGL2}
The desired companion normal form is a normal form of $\nabla_0$ in \eqref{2023_7_13_12_39}.
So the companion normal form is given by normalization of $\nabla_0$
by applying automorphisms on $\mathcal O_C\oplus (\Omega^1_C(D))^{-1}$.
To give the companion normal form,
first, we describe 
a decomposition of $\nabla_0$ relative to~\eqref{eq:decomposition}: 
$$\nabla_0=\begin{pmatrix}\alpha & \beta\\
\gamma & \delta \end{pmatrix}$$
where 
$$\left\{\begin{matrix}
\alpha&:& \mathcal O_C\longrightarrow \Omega^1_C(D+B) \hfill (\text{connection})\\
\beta&:& (\Omega^1_C(D))^{-1}\longrightarrow \Omega^1_C(D+B)  \hfill (\mathcal O_C\text{-linear})\\
\gamma&:& \mathcal O_C\longrightarrow \mathcal O_C(B)  \hfill (\mathcal O_C\text{-linear})\\
\delta&:& (\Omega^1_C(D))^{-1}\longrightarrow (\Omega^1_C(D))^{-1}\otimes \Omega^1_C(D+B)  \hskip1cm (\text{connection})\rlap{.}
\end{matrix}\right.$$
This form is unique only up to pre-composition by an element 
of the automorphism group $\mathrm{Aut}(E_0)$ of $E_0$.
Elements of $\mathrm{Aut}(E_0)$ are described as follows:
$$
\begin{pmatrix}
\lambda_1 & F \\
0& \lambda_2
\end{pmatrix},
$$
where $\lambda_1, \lambda_2 \in \mathbb C^*$
and $F \in H^0(C, \Omega^1_C(D))$.
It follows by construction that $\nabla_0$ admits no pole in restriction to 
$\mathcal O_C$ over the divisor $B$, so that actually we have 
$$\left\{\begin{matrix}
\alpha&:& \mathcal O_C\longrightarrow \Omega^1_C(D) \hskip1cm (\text{connection})\\
\gamma&:& \mathcal O_C\longrightarrow \mathcal O_C  \hfill (=\text{identity})\rlap{.}
\end{matrix}\right.$$
The action of an automorphism of the form 
\begin{equation}\label{eq:automorphism}
 \begin{pmatrix} 1 & F\\
    0 & 1 \end{pmatrix},\ \ \ F\in H^0(C, \Omega^1_C(D))
\end{equation}
transforms $\alpha$ into $\alpha-F\gamma$ (without affecting $\gamma$).
Therefore, there exists a unique choice $F$ such that $\alpha=\operatorname{d}$ is the 
trivial connection on $\mathcal O_C$. 
We thus get the unique companion normal form 
\begin{equation}\label{eq:normal_form}
 \nabla_0=\begin{pmatrix}\operatorname{d} & \beta\\
 1 & \delta \end{pmatrix}.
\end{equation}

Notice that the same companion normal form is obtained simply by taking 
the generator $\varphi_\nabla (1)$ for the second factor 
of~\eqref{eq:decomposition}, and the action of the 
automorphism~\eqref{eq:automorphism} in the above argument simply amounts 
to switching to this particular generator.

\subsection{Spectral data}\label{SubSect:SpectralData}

Now we consider the polar part of the meromorphic connection \eqref{2023_7_13_12_11}
at each point of the support of $D$.
We impose some conditions on the polar parts.
To describe the conditions, we introduce the notion of irregular curves with residues.
Let $\nu$ be a positive integer.
We set $I := \{ 1,2,\ldots,\nu\}$.
Let 
$\mathfrak{h}$
be the Cartan subalgebra
$$
\mathfrak{h}= 
\left\{ 
\begin{pmatrix}
h_1 & 0 \\
0 & h_2 
\end{pmatrix} \ \middle| \ 
h_1,h_2  \in \mathbb{C}
\right\}
$$
of the Lie algebra $\mathfrak{gl}_2(\mathbb{C})$.
Let $\mathfrak{h}_0$ be the
regular locus of $\mathfrak{h}$.

\begin{definition}\label{2023_7_14_23_01}
We say $X=(C,D,\{ z_i \}_{i \in I} , \{\boldsymbol{\theta}_{i} \}_{i \in I},
\boldsymbol{\theta}_{\mathrm{res}})$ is an 
{\rm irregular curve with residues} if 
\begin{itemize}
\item[(i)]  $C$ is a compact Riemann surface of genus $g$, 

\item[(ii)] $D = \sum_{i\in I} m_i [t_i]$ is an effective divisor on $C$,

\item[(iii)] $z_i$ is a generator of the maximal ideal of $\mathcal{O}_{C,t_i}$,

\item[(iv)] 
$\boldsymbol{\theta}_{i}
= (\theta_{i,-m_i} ,( \theta_{i,-m_i+1} ,\ldots, \theta_{i,-2}) )
\in \mathfrak{h}_0 \times \mathfrak{h}^{m_i -2}$, and 

\item[(v)] $\boldsymbol{\theta}_{\mathrm{res}}
= (\theta_{1,-1}, \theta_{2,-1},\ldots, \theta_{\nu,-1})$, where 
$\theta_{i,-1} \in 
\mathfrak{h}$, such that 
$\sum_{i=1}^{\nu} \mathrm{tr}(\theta_{i,-1} ) =-(2g-1)$.
\end{itemize}
We set  
$$
 \theta_{i,-1} = 
\begin{pmatrix}
\theta_{i,-1}^{-} & 0 \\
0 & \theta_{i,-1}^+
\end{pmatrix} \quad 
\text{for each $i \in I$}.
$$
We assume that 
$\sum_{i=1}^{\nu} \theta^{\pm}_{i,-1} \not\in \mathbb{Z}$
whatever are the signs $\pm$,
and,
if $m_i =1$, then $\theta^+_{i,-1} - \theta^-_{i,-1} \not\in \mathbb{Z}$.
\end{definition}

For an
irregular curve with residues $X$,
we set
\begin{equation}\label{2023_4_10_19_32}
\omega_i(X ) := 
\theta_{i,-m_i} \frac{\operatorname{d}\!z_i}{z_i^{m_i}} +
 \theta_{i,-m_i+1} \frac{\operatorname{d}\!z_i}{z_i^{m_i-1}} +
\cdots+ \theta_{i,-2}\frac{\operatorname{d}\!z_i}{z_i^{2}} +
\theta_{i,-1}\frac{\operatorname{d}\!z_i}{z_i}
\end{equation}
and $\mathcal{O}_{m_i[t_i]} := \mathcal{O}_{C,t_i}/(z_i^{m_i})$.
For an
irregular curve with residues $X$
and a meromorphic connection $(E,\nabla)$ in \eqref{2023_7_13_12_11},
we set $E|_{m_i [t_i]} := E \otimes \mathcal{O}_{m_i[t_i]}$.
Let 
$$
\nabla|_{m_i[t_i]} 
\colon 
E|_{m_i [t_i]} \longrightarrow E|_{m_i [t_i]} \otimes \Omega^1_C(D)
$$
be the morphism induced by $\nabla$.

\begin{definition}\label{2023_7_13_16_05}
We call $(E,\nabla  )$ a {\rm rank $2$ meromorphic connection over an
irregular curve with residues $X$} if 
\begin{itemize}
\item[(i)] $E$ is a rank $2$ vector bundle of degree $2g-1$ on $C$,

\item[(ii)] $\nabla \colon E \rightarrow E \otimes \Omega^1_{C}(D)$ is a connection, and

\item[(iii)] there exists an isomorphism 
$ \varphi_{m_i[t_i]} \colon E|_{m_i [t_i]} \rightarrow  \mathcal{O}^{\oplus 2}_{m_i[t_i]}$
for each $i \in I$
such that 
 $$
( \varphi_{m_i[t_i]}\otimes 1) \circ \nabla|_{m_i[t_i]} \circ \varphi^{-1}_{m_i[t_i]} 
 = \operatorname{d} +\, \omega_i(X ).
 $$
 Here $\omega_i(X )$ is defined in \eqref{2023_4_10_19_32}.

\end{itemize}
We call $\omega_i(X )$ the {\rm spectral data} of $(E,\nabla  )$
and call the submodule $\varphi_{m_i[t_i]}^{-1}(\mathcal{O}_{m_i[t_i]} \oplus 0)$
of $E|_{m_i [t_i]}$ the {\rm quasi-parabolic structure of $(E,\nabla  )$ at $t_i$}.

\end{definition}

From now on, by a connection we will mean a rank $2$ meromorphic connection over a fixed irregular curve with residues $X$. 
So we impose the condition (iii) of Definition \ref{2023_7_13_16_05} on 
the polar parts of the meromorphic connection $\nabla$ in \eqref{2023_7_13_12_11}
at the points of the support of $D$.
This condition means that 
the polar parts of $\nabla$ at $t_i$ 
are diagonalizable with eigenvalues equal to the diagonal entries of $\omega_i(X )$ for $i=1,2,\ldots,n$.

\begin{remark}\label{2023_7_13_16_34}
In Definition \ref{2023_7_13_16_05}, 
we impose the condition that $\sum_{i=1}^{\nu} \theta^{\pm}_{i,-1} \not\in \mathbb{Z}$
whatever are the signs $\pm$.
By this assumption and the argument as in \cite[Proposition 6]{KLS}, 
we have that $(E,\nabla)$ is irreducible. 
Then some arguments become simple.
For example, 
$\varphi_\nabla = 0$ if and only if the free subsheaf $\mathcal O_C$ of 
$E$ is a proper $\nabla$-invariant subbundle.
So we have that $\varphi_{\nabla} \neq 0$.
Moreover, $(E,\nabla)$ is automatically stable (described in Section \ref{subsect_moduli} below).
\end{remark}

\subsection{The polar parts of $\delta$}\label{subsec:delta}

We fix an
irregular curve with residues $X$.
Let $(E,\nabla  )$ 
be a rank $2$ meromorphic connection over $X$
and $\nabla_0$ be the companion normal form for $(E,\nabla  )$.
We consider the $(2,2)$-entry $\delta$ of this 
companion normal form $\nabla_0$. 

It immediately follows from~\eqref{eq:normal_form} that the connection 
$\delta$ coincides with the trace connection $\tr(\nabla_0)$ on 
$\det(E_0)=(\Omega^1_C(D))^{-1}$.
It is further related to the trace connection $\tr(\nabla)$ by
$$\delta = \tr(\nabla_0)=\tr(\nabla)+\frac{\operatorname{d}\!\varphi_\nabla}{\varphi_\nabla}.$$
\begin{lem}\label{lem:delta}
\begin{enumerate}
\item The polar part of $\delta$ over $D$ is determined by the spectral 
data;
\item The polar part of $\delta$ over $B$ is logarithmic with residue $+1$;
\item $\delta$ is determined by the irregular curve with residues $X$ up to adding a holomorphic $1$-form of $C$. 
\end{enumerate}
\end{lem}
\begin{proof}
 The polar part of $\delta$ at $t_i$ is equal to $\operatorname{tr} (\omega_i )$, showing the first assertion. 
 In view of our assumption $q_{j_1}\neq q_{j_2}$ for ${j_1}\neq {j_2}$, the second assertion is classical. 
 Let now $\delta, \delta'$ be the $(2,2)$-entries of companion normal forms $\nabla_0, \nabla_0'$ of connections $\nabla, \nabla'$ satisfying the conditions of  Definition~\ref{2023_7_13_16_05}. 
 By the first part, $\delta - \delta'$ is then a global holomorphic $1$-form on $C$.
\end{proof}

As a consequence of the lemma and by $\operatorname{dim}_{\mathbb{C}} H^0(C, \Omega_C^1) = g$, the possible values for $\delta$ represent $g$ free parameters for a meromorphic connection over $X$.

\subsection{The polar parts of $\beta$}\label{subsec:beta}

Next we consider the $(1,2)$-entry $\beta$ of the companion normal form $\nabla_0$ of a meromorphic connection $\nabla$ over the irregular curve with residues $X$. 
By the condition $\gamma = 1$ in~\eqref{eq:normal_form}, $\beta$ accounts for the  determinant of the characteristic polynomial of the residues. 
By Definition~\ref{2023_7_13_16_05}, the eigenvalues of the connection matrix of 
$\nabla$ are differentials (of the first kind) with a pole of order at most $m_i$ at $t_i$. 
The same condition then holds for $\nabla_0$ too, because it only differs from $\nabla$ by elementary modifications at points $q_j \neq t_i$. 
As the determinant of a $2\times 2$ matrix is a quadratic expression of the eigenvalues, we see that $\beta$ must be a quadratic differential with poles of order at most $2m_i$ at $t_i$. 
Over $B$, a similar argument shows that $\beta$ has poles of order at most $2$. 

Let us fix local coordinate charts $z_i$ centered at the pole $t_i$. 
One may then expand $\beta$ into Laurent series: 
$$
  \beta = \left( \beta_{i, -2m_i} z_i^{-2m_i} + \cdots + \beta_{i, -2} z_i^{-2} + O(z_i^{-1}) \right) (\operatorname{d}\! z_i)^{\otimes 2}. 
$$
Notice that for given $\beta$ the coefficient $\beta_{i, -2}$ is independent of the chosen coordinate chart $z_i$, however the other coefficients depend on $z_i$. 
We also fix local coordinate charts $z_j$ centered at the apparent singularity $q_j$, 
and have a similar expansion 
$$
  \beta = \left( \beta_{j, -2} z_j^{-2} + \beta_{j, -1} z_j^{-1} + O(z_j^0) \right) (\operatorname{d}\! z_j)^{\otimes 2}. 
$$
Analogously to Lemma~\ref{lem:delta}, we therefore find 
\begin{lem}\label{2023_7_11_12_53}
\begin{enumerate}
\item The coefficients $\beta_{i, -2m_i}, \ldots , \beta_{i, -2}$ 
are uniquely determined by the irregular curve with residues $X$ (and the holomorphic coordinate $z_i$); 
\item We have $\beta_{j, -2} =0$. 
\item $\beta$ is determined by the irregular curve with residues $X$ up to adding a section of $(\Omega^1_C)^{\otimes2}(D)$. 
\end{enumerate}
\end{lem}
\begin{proof}
The coefficients $\beta_{i, -2m_i}, \ldots , \beta_{i, -2}$ all admit homogeneous quadratic expressions in terms of the eigenvalues of $\boldsymbol{\theta}_{i}, \boldsymbol{\theta}_{\mathrm{res}}$, therefore they are determined by them. 
Conversely, the coefficients $\beta_{i, -2m_i}, \ldots , \beta_{i, -2}$ determine the polar part of the eigenvalues. 
It is classical that for an apparent singularity of $\nabla_0$, one of the two eigenvalues of the residue must vanish. 
This implies that for every $q\in B$ the product of the eigenvalues of $\operatorname{res}_q(\nabla_0 )$ vanishes. 
As this latter product gives the leading (second) order term $\beta_{j, -2}$, we get the second assertion. 
The last part follows from the first two as in Lemma~\ref{lem:delta}. 
\end{proof}

As a consequence of the lemma and by $\operatorname{dim}_{\mathbb{C}} H^0(C, (\Omega^1_C)^{\otimes2}(D)) = 3g-3+n$, the possible values for $\beta$ represent  $3g-3+n$ free parameters for a connection $\nabla$ on $X$ having apparent singularities at a fixed reduced divisor $B$ of length $N$. 

From now on, we set $\beta_{j, -1} = \zeta_j$, so that we have the expansion 
\begin{equation}\label{eq:expansion_beta}
  \beta = \zeta_j \frac{(\operatorname{d}\! z_j)^{\otimes 2}}{z_j} + \beta^{(j)}  
\end{equation}
for some local holomorphic quadratic differential $\beta^{(j)}$. 
Notice that $\zeta_j$ depends on the coordinate $z_j$, however the element $\zeta_j \operatorname{d}\! z_j \in \Omega_C^1\vert_{q_j}$ of the fiber of the holomorphic cotangent (or canonical) bundle over $q_j$ does not depend on it. 
As a matter of fact, since $\beta$ belongs to an affine space modelled over $H^0(C, (\Omega^1_C)^{\otimes2}(D))$ (and in order to be consistent with the  decomposition~\eqref{eq:decomposition}), it is even more rigourous to consider 
$\zeta_j \operatorname{d}\! z_j$ as elements of the fiber $\Omega^1_C (D)\vert_{q_j}$, 
using the inclusion $\Omega^1_C \subset \Omega^1_C(D)$. 
In the sequel we will consider them to be such elements. 
It will turn out that these quantities $\zeta_j \operatorname{d}\! z_j$ are closely related to accessory parameters.

\subsection{Determination of $\beta$ and $\delta$ in terms of $\zeta$}

Fix a reduced divisor $B$ of length $N$ on $C$ with support disjoint from $D$. 
In Section \ref{subsec:delta} and Section \ref{subsec:beta}, 
we have found that (normal forms of) meromorphic connections with residue on $X$ that have apparent singularities at $B$ can be described by an affine space of complex dimension $g + 3g-3+n = N$ ($g$ coming from the choice of $\delta$ and $3g-3+n$ from the choice of $\beta$). 
In this section, we provide a description of such connections in terms of analogs of separated variables. 
Namely, it will turn out that generically the data of $\delta, \beta$ is 
equivalent to the $N$-tuple 
$(\zeta_1 \operatorname{d}\! z_1, \ldots, \zeta_N \operatorname{d}\! z_N)$. 

The fact that singular points are apparent over $B$ imposes further constraints on 
$\beta$ and $\delta$. 
This constraint gives $1$ linear condition for each point $q_j$ 
and we can expect that these  constraints fix $\beta$ and $\delta$ uniquely 
in terms of the data $(q_j,\zeta_j \operatorname{d}\! z_j)_{j=1}^N$. 
In fact, this is true for the genus $g=0$ case (see \cite{DiarraLoray}) and we 
will show in Lemma~\ref{lem:independence} that this is also true for 
\emph{generic} choices of $(q_j,\zeta_j \operatorname{d}\! z_j)_{j=1}^N$ if $g>0$.

In fact, the data of $\zeta_j \operatorname{d}\! z_j$ can be interpreted as a certain 
quasi-parabolic structure over $B$.
Indeed, at a point $q_j$ and with respect to the decomposition~\eqref{eq:decomposition}, 
the residue of $\nabla_0$ reads as 
$$
\mathrm{res}_{q_j}\nabla_0=\begin{pmatrix} 0 & \zeta_j \operatorname{d}\! z_j\\
0 & 1 \end{pmatrix} .
$$
So, the vector $\begin{pmatrix}  \zeta_j \operatorname{d}\! z_j\\ 1 \end{pmatrix}$ is an eigenvector with respect to eigenvalue $1$ and the map $\phi_\nabla$ (see~\eqref{eq:phi_nabla}) is just the positive elementary transformation with respect to these parabolic directions at all points $q_j$. 
In summary, the data of all values $\zeta_j \operatorname{d}\! z_j$ is 
equivalent to the data of a quasi-parabolic structure of $E_0$ over $B$ (i.e., a line in the fiber of $E_0$ over each $q_j$) distinct from the destabilizing subbundle 
$\mathcal O_C\subset E_0$ for every $j$.

Let us denote by $\boldsymbol{\Omega}(D)$ the total space  of the line bundle 
$\Omega^1_C(D)$.

\begin{lem}\label{lem:independence} For generic data 
$(q_j,\zeta_j\operatorname{d}\! z_j)_j\in \mathrm{Sym}^{4g-3+n}(\boldsymbol{\Omega}(D))$
there exist unique $\beta$ and $\delta$ as above such that the corresponding $\nabla_0$ 
has apparent singular points at all the points $q_j$ ($1\leq j \leq N:= 4g-3+n$), and such that the Laurent  expansion~\eqref{eq:expansion_beta} is fulfilled. 
\end{lem}

\begin{proof} 
Let us consider $(q_j,\zeta_j\operatorname{d}\! z_j)_j$ 
such that $q_j$'s are pair-wise distinct, and 
do not intersect the support of $D$. Given one point $(q_j,\zeta_j\operatorname{d}\! z_j)$, 
we can diagonalize the 
residue $\mathrm{res}_{q_i}\nabla_0$ by conjugating by a triangular matrix
\begin{equation}\label{diag_by_conj_by_triangular}
\begin{pmatrix}  1 & \zeta_j \operatorname{d}\! z_j\\ 0& 1 \end{pmatrix}^{-1}
\begin{pmatrix}  0 & \beta\\ 1 & \delta \end{pmatrix}
\begin{pmatrix}  1 & \zeta_j\operatorname{d}\! z_j\\ 0 & 1 \end{pmatrix}
+\begin{pmatrix}  1 & \zeta_j\operatorname{d}\! z_j\\ 0& 1 \end{pmatrix}^{-1}
\operatorname{d}\! 
\begin{pmatrix}  1 & \zeta_j\operatorname{d}\! z_j\\ 0 & 1 \end{pmatrix}
\end{equation}
$$
=\begin{pmatrix}   - \zeta_j \operatorname{d}\! z_j 
& \beta-\zeta_j  \delta \otimes \operatorname{d}\! z_j
-\zeta_j^2 \operatorname{d}\! z_j^{\otimes 2} \\ 1 
& \delta+\zeta_j\operatorname{d}\! z_j \end{pmatrix}
=\begin{pmatrix}  0 & 0\\ 0 & 
\frac{\operatorname{d}\! z_j}{z_j} \end{pmatrix}
+\text{holomorphic}$$
where $z_j$ stands for a local coordinate at $q_j$. Then the elementary transformation $\phi_\nabla$
is locally equivalent to the conjugacy by $\begin{pmatrix}  1 & 0\\ 0 & z_j^{-1} \end{pmatrix}$ yielding
\begin{equation}\label{conn_matrix_at_qj}
\begin{pmatrix}   - \zeta_j \operatorname{d}\! z_j
& \frac{\beta- \zeta_j\delta\otimes \operatorname{d}\! z_j-\zeta_j^2 \operatorname{d}\! z_j^{\otimes 2}}{z_j} \\ 
z_j  & \delta+\zeta_j\operatorname{d}\! z_j
- \frac{\operatorname{d}\! z_j}{z_j} \end{pmatrix}.
\end{equation}
The apparent point condition is therefore equivalent to saying that 
$\beta-\zeta_j\delta\otimes \operatorname{d}\! z_j 
-\zeta_j^2\operatorname{d}\! z_j^{\otimes 2}$ is (holomorphic and) {\bf vanishing} at $q_j$.
This condition is linear on $\beta$ and $\delta$ and rewrites
\begin{equation}\label{eq:apparent}
\underbrace{\beta-\zeta_j\delta \otimes \operatorname{d}\! z_j}_\text{holomorphic}\vert_{q_j}
= \zeta_j^2\operatorname{d}\! z_j^{\otimes 2}|_{q_j},
\end{equation}
where the right hand side does not involve $\beta$ and $\delta$. 
If we assume that $(q_1, \ldots, q_N)$ lies in the image of the map $\operatorname{App}$ (see~\eqref{eq:App}), then the normal form of any $(E,\nabla )$ in the preimage produces a solution $(\delta_0, \beta_0)$. 
Fixing such solutions, by Lemmas~\ref{lem:delta},~\ref{2023_7_11_12_53} we may rewrite
$$\left\{\begin{matrix}
\beta&=& \beta_0+b_1\nu_1+\cdots+b_{N-g}\nu_{N-g} \hfill \\
\delta&=& \delta_0+d_1\omega_1+\cdots+d_g\omega_g
\end{matrix}\right.$$
where $(\omega_l)_{l=1}^g$, $(\nu_k)_{k=1}^{N-g}$ are respective bases of  $H^0(C,\Omega^1_C)$ and $H^0(C,(\Omega^1_C)^{\otimes2}(D))$. 
Using these expressions, the constraint that $q_j$ is an apparent singularity can be  rewritten as a linear system consisting of $N$ equations in the $N$ variables $b_k$, $d_l$.
The condition to uniquely determine $\beta$ and $\delta$ in terms of the data $(q_j,\zeta_j \operatorname{d}\! z_j)$ 
is that the following determinant does not vanish
\begin{equation}\label{eq:determinant}
\det\begin{pmatrix}  
\nu_1(q_1)& \cdots& \nu_{N-g}(q_1) & \zeta_1\operatorname{d}\! z_1 \omega_1(q_1)&\cdots& \zeta_1 \operatorname{d}\! z_1 \omega_g(q_1) \\
\vdots & \ddots & \vdots & \vdots & \ddots & \vdots \\
\nu_1(q_{N})& \cdots& \nu_{N-g}(q_{N}) & \zeta_{N}\operatorname{d}\! z_N \omega_1(q_{N})&\cdots& \zeta_{N}\operatorname{d}\! z_N \omega_g(q_{N}) 
\end{pmatrix}.
\end{equation}
Of course, it is sufficient for our purpose to check that 
we can find some  $(q_j,\zeta_j \operatorname{d}\! z_j)$'s such that 
this determinant does not vanish, so that it will be generically non vanishing. 
If we set $\zeta_1=\cdots=\zeta_{N-g}=0$, then the matrix has a zero block of dimension $(N-g)\times g$ in the top right corner, and the determinant factors as 
$$\zeta_{N-g+1}\operatorname{d}\! z_{N-g+1} \cdots \zeta_{N}\operatorname{d}\! z_N\cdot 
\det\begin{pmatrix}  
\nu_1(q_1)& \cdots& \nu_{N-g}(q_1) \\
\vdots & \ddots & \vdots  \\
\nu_1(q_{N-g})& \cdots& \nu_{N-g}(q_{N-g})
\end{pmatrix}
\cdot   
\det\begin{pmatrix}  
\omega_1(\tilde q_{1})&\cdots& \omega_g(\tilde q_1) \\
 \vdots & \ddots & \vdots \\
\omega_1(\tilde q_g)&\cdots& \omega_g(\tilde q_g) 
\end{pmatrix}$$
where $\tilde q_j=q_{j+N-g}$. After setting $\zeta_{N-g+1}=\cdots=\zeta_{N}=1$,
it is enough to find $q_j$'s such that the two smaller determinants are non zero.
To conclude the proof, let us denote by $L$ any of the two lines bundles $\Omega^1_C$ or
$(\Omega^1_C)^{\otimes2}(D)$, and by $\mu_1,\ldots,\mu_{N'}$ a corresponding basis of
$H^0(C,L)$. Then we want to prove that the image of the curve by the evaluation map
$$C\stackrel{\text{ev}}{\longrightarrow} \mathbb P^{N'-1}\ ;\ 
q\mapsto (\mu_1(q):\ldots:\mu_{N'}(q))$$
is not contained in some hyperplane, i.e. that we can find $q_1,\ldots,q_{N'}\in C$ such that
the image is not contained in some hyperplane. But this is true, otherwise, we would have 
a linear relation between $\mu_1,\ldots,\mu_{N'}$ contradicting that they form a basis.
\end{proof}

\begin{remark} In the previous proof, the locus of $q_j$'s for which 
$\det(\omega_i(\tilde q_j))_{i,j}$
vanishes corresponds to the Brill--Noether locus for divisor $\tilde q_1+\cdots+\tilde q_g$.
In genus $0$, see \cite[Lemma 4.4]{DiarraLoray}, and in the case $g = 1, n=2$, it will be analyzed in Lemma \ref{2023_7_21_17_07(2)} below.
The analogous locus has been worked out for $r=3,g=0$  in~\cite{Ivanics}.
The conclusion~\cite[Theorem~1.7]{Ivanics} is that there exists a hypersurface $V\subset \mathrm{Sym}^{3n-8}(\boldsymbol{\Omega}(D)\setminus \Delta)$ along which the canonical coordinates are undefined.
(Here, we use $\Delta$ to denote the big diagonal determined by $q_j\neq q_{j'}$ for all $j\neq j'$.)
Moreover, for $r=3, g=0, n=3$ there exists another hypersurface $\hat{V}$ transversely intersecting $V$ and an open subset of $V \cap \hat{V}$ whose blow-up in $\boldsymbol{\Omega}(D)\setminus \Delta$ is contained in the moduli space~\cite[Section~4]{Ivanics}.
Conjecturally, this phenomenon holds more generally too~\cite[Conjecture~1.9]{Ivanics}.
However, in general these hypersurfaces are quite complicated to analyze; already in the case $r=3, g=0, n=3$, the hypersurface $V$ is determined by an explicit polynomial that is linear in the variables $\zeta_j$ and of degree $10$  in the variables $q_j$, but $\hat{V}$ is not made explicit because its formula would be too large and not very informative.
\end{remark}

\begin{lem} When $g=0$, any data 
$(q_j,\zeta_j\operatorname{d}\! z_j)_j\in \mathrm{Sym}^{n-3}(\boldsymbol{\Omega}(D))$
gives rise to unique $\beta$ and $\delta$ such that the corresponding $\nabla_0$ 
has apparent singular points at all $q_j$'s. However, for $g>0$, there always exist 
data $(q_j,\zeta_j\operatorname{d}\! z_j)_j$ such that the determinant (\ref{eq:determinant}) vanishes.
\end{lem}

\begin{proof} When $g=0$, this directly follows from \cite{DiarraLoray} 
(a consequence of Lagrange interpolation).
When $g>0$, fix generic $q_j$'s and let $\omega\in H^0(C,\Omega^1_C(D))$. 
If we set $\zeta_j:=\omega(q_j)$, 
then the last column of (\ref{eq:determinant}) is just the evaluation of
the section $\omega\otimes\omega_g\subset H^0(C,(\Omega^1_C)^{\otimes2}(D))$ 
at $q_1,\cdots,q_{4g-3+n}$
and is therefore a linear combination of the $3g-3+n$ first columns.
\end{proof}

\section{Symplectic structure and canonical coordinates}\label{Sect:SymplecticGL2andCano}

We fix an irregular curve with residues $X=(C,D,\{ z_i \}_{i \in I} , \{\boldsymbol{\theta}_{i} \}_{i \in I}, \boldsymbol{\theta}_{\mathrm{res}})$. 
As usual, we use the notation $N:= 4g+n-3$, where $g$ is the genus of $C$ and $n = \deg (D)$.
We will consider the moduli space $M_{X}$ of rank 2 meromorphic connections over $X$.
This moduli space is 
constructed in \cite[Theorem 2.1]{IS} and carries a natural symplectic structure described in \cite[Proposition 4.1]{IS}. 
The purpose of this section is to give canonical coordinates on an open subset of $M_{X}$ with respect to this symplectic structure. 
First we describe the moduli space $M_{X}$.

\subsection{Moduli spaces}\label{subsect_moduli}

Let $(E,\nabla)$ be a rank 2 meromorphic connection over $X$.
Then, the subsheaf 
$$
l^{(i)} := \varphi_{m_i[t_i]}^{-1}(\mathcal{O}_{m_i[t_i]} \oplus 0)
\subset E_{m_i[t_i]}
$$
equips $(E,\nabla)$ with a canonical quasi-parabolic structure at each $t_i$.  
So we may consider $(E,\nabla)$ as a {\it quasi-parabolic connection}
$(E,\nabla ,  \{l^{(i)}\} )$
defined in \cite[Definition 1.1]{Ina} and \cite[Definition 2.1]{IS}.
A stability condition for 
quasi-parabolic connections
is introduced in \cite[Definition 2.1]{Ina} and \cite[Definition 2.2]{IS}.
The moduli space of 
stable quasi-parabolic connections
is constructed in \cite[Theorem 2.1]{Ina} and \cite[Theorem 2.1]{IS}.
In our situation, any rank $2$ meromorphic connections over $X$ 
are irreducible (see Remark \ref{2023_7_13_16_34}).
So our objects are automatically stable objects.
We omit the stability condition of the quasi-parabolic connections.

Let $M_{X}$ be the moduli space of rank 2 meromorphic connections over the irregular curve with residues $X$.
If $(E,\nabla  ) \in M_{X}$ satisfies $\dim_{\mathbb{C}} H^0(C,E)=1$, then we have a unique $\mathcal{O}_C$-morphism $\varphi_\nabla$ in~\eqref{2023_7_11_12_46}. 
The $\mathcal{O}_C$-morphism $\varphi_\nabla$ is nonzero, since $(E,\nabla )$ is irreducible.
So we may define the divisor $\mathrm{div}(\varphi_\nabla )$ in~\eqref{eq:B} for $(E,\nabla)$.
We set 
$$
M_{X}^0
:= \left\{ (E,\nabla ) \in M_X
\ \middle| \ 
\begin{array}{l}
\text{$\dim_{\mathbb{C}} H^0(C,E) =1$,}\\
\text{$\mathrm{div}(\varphi_\nabla )$ is reduced, and} \\
\text{$\mathrm{div}(\varphi_\nabla )$ has disjoint support with $D$}
\end{array}
\right\}.
$$
Next we recall the natural symplectic structure on $M_X$.

\subsection{Symplectic structure}\label{subsect:symplecticGL2}

We will describe the natural symplectic structure on $M_X$ 
via \v{C}ech cohomology.
This is defined in \cite[Proposition 7.2]{Ina} and \cite[Proposition 4.1]{IS}.
This is an analog of the symplectic form on the moduli space of stable Higgs bundles in \cite{Bott}.
This description of the symplectic structure is useful 
to compare this symplectic structure with
the Atiyah--Bott--Goldman symplectic structure on the character variety via 
the Riemann--Hilbert map 
(for example, see \cite[the proof of Proposition 7.3]{Ina} and \cite[Theorem 3.2]{Bis}).
Moreover,
this description of the symplectic structure is useful 
to describe the isomonodromic deformations 
(for example, see \cite[Proposition 4.3]{BHH}, 
\cite[Proposition 4.4]{BHH1}, and \cite[Proposition 3.8]{Kom0}).

First we recall the description of the tangent space of
$M_{X}$ 
at $(E,\nabla  ) \in M_{X}$
in terms of the hypercohomology of a certain complex
(\cite[the proof of Theorem 2.1]{Ina} and \cite[the proof of Proposition 4.1]{IS}).
We consider $(E,\nabla  )$ as a 
quasi-parabolic connection $(E,\nabla ,  \{l^{(i)}\} )$.
We define a complex $\cF^{\bullet}$ for $(E,\nabla ,  \{l^{(i)}\} )$ by 
\begin{equation}\label{2023_7_4_14_18}
\begin{aligned}
&\cF^0 := \left\{  s \in \cE nd (E)  
\ \middle| \  
s |_{m_i t_i} (l_j^{(i)}) 
\subset l_j^{(i)}  \text{ for $j=0,1$ and any $i$}
\right\} \\
&\cF^1 :=  \left\{  s \in  \cE nd (E)\otimes \Omega^1_{C}(D)  
\ \middle| \    
s|_{m_i t_i} (l_j^{(i)}) 
\subset l_{j+1}^{(i)} \otimes \Omega^1_C \text{ for $j=0,1$ and any $i$} 
\right\} \\
&\nabla_{\cF^{\bullet}} \colon \cF^0 \lra \cF^1; 
\quad \nabla_{\cF^{\bullet}} (s) = \nabla \circ s - s \circ \nabla.
\end{aligned}
\end{equation}
Here we put $l_0^{(i)} := E_{m_i[t_i]} $, 
$l_1^{(i)} : =l^{(i)}$, and 
$l_2^{(i)} : =0$.
Then we have an isomorphism between 
the tangent space 
$T_{(E,\nabla ,  \{l^{(i)}\} )} M_{X}$
and $\bH^1(\mathcal{F}^{\bullet})$.

Now we recall this isomorphism.
We take an analytic (or affine) open 
covering $C = \bigcup_{\alpha} U_{\alpha}$ such that 
$E|_{U_{\alpha}} \cong \mathcal{O}^{\oplus 2}_{U_{\alpha}}$ for any $\alpha$,
$\sharp\{ i \mid t_i  \cap U_{\alpha} \neq \emptyset \} \le 1$ for any $\alpha$ and 
$\sharp\{ \alpha \mid t_i  \cap U_{\alpha} \neq \emptyset \} \le 1$ for any $i$.
Take a tangent vector 
$v \in T_{(E,\nabla ,  \{l^{(i)}\} )} M_{X}$.
The field $v$ corresponds to an infinitesimal deformation 
$(E_{\epsilon},\nabla_{\epsilon}, \{ l_{\epsilon}^{(i)} \})$
of $(E,\nabla ,  \{l^{(i)}\} )$
over $C \times \mathrm{Spec}\,\mathbb{C}[\epsilon]$
such that $(E_{\epsilon},\nabla_{\epsilon}, 
\{ l_{\epsilon}^{(i)} \}) \otimes \mathbb{C}[\epsilon]/(\epsilon) 
\cong ( E, \nabla, \{ l^{(i)} \})$,
where $\mathbb{C}[\epsilon]= \mathbb{C}[t]/(t^2)$.
There is an isomorphism
\begin{equation*}\label{equation isom verphi 1}
\varphi_{\alpha} \colon  E_{\epsilon}|_{U_{\alpha}\times \mathrm{Spec}\, \mathbb{C}[\epsilon] }
\xrightarrow{\sim} \mathcal{O}^{\oplus 2}_{U_{\alpha}\times \mathrm{Spec}\,  \mathbb{C}[\epsilon]} 
\xrightarrow{\sim} E|_{U_{\alpha}} \otimes \mathbb{C}[\epsilon]
\end{equation*}
such that $\varphi_{\alpha}\otimes \mathbb{C}[\epsilon]/(\epsilon ) 
\colon E_{\epsilon}\otimes \mathbb{C}[\epsilon]/(\epsilon)|_{U_{\alpha}} 
\xrightarrow{\sim} E|_{U_{\alpha}}\otimes \mathbb{C}[\epsilon]/(\epsilon)
=E|_{U_{\alpha}}$ 
is the given isomorphism and 
that $\varphi_{\alpha}|_{t_i\times \Spec \mathbb{C}[\epsilon]} (l_{\epsilon}^{(i)}) 
= l^{(i)}|_{U_{\alpha} \times \Spec \mathbb{C}[\epsilon]}$ 
if $t_i \cap U_{\alpha} \neq \emptyset$.
We put 
\begin{equation*}
\begin{aligned}
u_{\alpha\beta} :=&\ \varphi_{\alpha} \circ \varphi_{\beta}^{-1} 
- \mathrm{id}_{E|_{U_{\alpha\beta}\times \Spec \mathbb{C}[\epsilon]}},\\
v_{\alpha} :=&\ (\varphi_{\alpha}\otimes \mathrm{id}) \circ 
\nabla_{\epsilon} |_{U_{\alpha}\times \Spec \mathbb{C}[\epsilon]} \circ \varphi^{-1}_{\alpha} 
- \nabla|_{U_{\alpha}\times \Spec \mathbb{C}[\epsilon]}.
\end{aligned}
\end{equation*}
Then $\{ u_{\alpha\beta} \} \in C^1 ((\epsilon) \otimes \cF^0)$, 
$\{ v_{\alpha} \} \in C^0((\epsilon) \otimes \cF^1)$ and 
we have the cocycle conditions
\begin{equation*}
u_{\beta \gamma}-u_{\alpha \gamma} +u_{\alpha\beta}  = 0
\quad \text{and}
\quad \nabla\circ u_{\alpha\beta}-u_{\alpha\beta} \circ \nabla =  v_{\beta}-v_{\alpha}.
\end{equation*}
So $[(\{ u_{\alpha\beta} \},\{ v_{\alpha} \} )]$ determines an element  
of $\bH^1(\mathcal{F}^{\bullet})$.
This correspondence gives an isomorphism
between 
the tangent space 
$T_{(E,\nabla ,  \{l^{(i)}\} )} M_{X}$
and $\bH^1(\mathcal{F}^{\bullet})$.

We define a pairing
\begin{equation}\label{2020.11.7.15.48}
\begin{aligned}
\bH^1( \cF^{\bullet}) \otimes \bH^1(\cF^{\bullet}) 
&\lra \bH^2(\mathcal{O}_C \xrightarrow{d}\Omega_{C}^{1}) \cong \mathbb{C}\\
[(\{ u_{\alpha\beta} \}, \{ v_{\alpha} \})]\otimes[(\{ u_{\alpha\beta}' \}, \{ v_{\alpha}' \} )] &\longmapsto 
[ (\{  \mathrm{tr}( u_{\alpha\beta}  u_{\beta\gamma}') \}, -\{
 \mathrm{tr} (u_{\alpha\beta}  v_{\beta}') 
- \mathrm{tr} (v_{\alpha}  u'_{\alpha\beta}) \} )],
\end{aligned}
\end{equation}
considered in \v{C}ech cohomology with respect to an open covering $\{ U_{\alpha} \}$ 
of $C$, $\{ u_{\alpha\beta} \} \in C^1(\cF^0)$, $\{ v_{\alpha} \} \in C^0(\cF^1)$
and so on.
This pairing gives a nondegenerate 2-form on the moduli space $M_{X}$.
This fact follows from the Serre duality and the five lemma:
\begin{equation}\label{diagram:hypercohomology_LES}
\xymatrix{
H^0(\mathcal{F}^0) \ar[r] \ar[d]^-{\sim} & H^0(\mathcal{F}^1) \ar[r] \ar[d]^-{\sim} 
& \bH^1(\mathcal{F}^{\bullet}) \ar[r] \ar[d]^-{\sim} & H^1(\mathcal{F}^0) \ar[r] \ar[d]^-{\sim} 
& H^1(\mathcal{F}^1) \ar[d]^-{\sim} \\
H^1(\mathcal{F}^1)^{\vee} \ar[r] & H^1(\mathcal{F}^0)^{\vee} \ar[r] 
& \bH^1(\mathcal{F}^{\bullet})^{\vee} \ar[r] & H^0(\mathcal{F}^1)^{\vee} \ar[r] 
& H^0(\mathcal{F}^0)^{\vee}\rlap{.}
}
\end{equation}
We denote by $\omega$
the nondegenerate 2-form on $M_{X}$.
This 2-form $\omega$ is a symplectic structure.
That is, we have $d\omega=0$ (see \cite[Proposition 7.3]{Ina}
and \cite[Proposition 4.2]{IS}).

We get as a consequence: 

\begin{prop}\label{prop:dimension}
 The dimension of $M_{X}^0$ is equal to $2N$, where $N = 4g - 3 + n$. 
\end{prop}

\begin{proof}
By irreducibility of $(E,\nabla)$ and Schur's lemma we have 
$$
  \bH^0(\mathcal{F}^{\bullet})\cong \mathbb{C}. 
$$
On a Zariski open subset of $M_{X}$, the underlying 
quasi-parabolic vector bundle $(E, \{ l^{(i)} \}_i ) $ is irreducible, so we also have
$$
  H^0(\mathcal{F}^{0})\cong \mathbb{C}. 
$$
Clearly, we have $\operatorname{deg} (\mathcal{F}^0) = -\operatorname{length} (D)$. 
From Riemann--Roch we find 
\begin{align*}
  \operatorname{dim}_{\mathbb{C}} H^1 (\mathcal{F}^0) & = \operatorname{dim}_{\mathbb{C}} H^0 (\mathcal{F}^0) + 4 (g-1) - \operatorname{deg} (\mathcal{F}^0) \\
  & = 4g - 3 + n = N. 
\end{align*}
By Serre duality and Euler characteristic count applied to the hypercohomology long exact sequence~\eqref{diagram:hypercohomology_LES}, we get the statement. 
\end{proof}

\subsection{Trivializations of $E$}\label{2023_7_4_13_59}

Our purpose is to give canonical coordinates of 
$M_{X}^0$
with respect to the symplectic form~\eqref{2020.11.7.15.48}.
To do it, we will calculate the \v{C}ech cohomology 
by taking trivializations of $E$.
To simplify the calculation, we take 
trivializations of $E$ by using 
$$
\phi_{\nabla} \colon E_0 \xrightarrow{\ \subset \ } E,
$$
whose cokernel defines the apparent singularities.
In this section, we will discuss the construction of
the trivializations of $E$ by using $\phi_{\nabla}$.

We take $(E,\nabla, \{ l^{(i)} \}) \in M_{X}^0$.
Let $\{(q_j, \zeta_j\operatorname{d}\! z_j)\}_{j=1,2,\ldots,N}$ 
be the point on $\mathrm{Sym}^{N}(\boldsymbol{\Omega}(D))$ 
corresponding to $(E,\nabla, \{ l^{(i)} \})$.
We assume that the point $\{(q_j, \zeta_j\operatorname{d}\! z_j)\}_{j=1,2,\ldots,N}$ is generic in the sense of Lemma~\ref{lem:independence}.
Let $U^{\mathrm{an}}_{q_j}$ be an analytic open subset of $C$ such that 
$q_j \in U^{\mathrm{an}}_{q_j}$
and 
$U^{\mathrm{an}}_{t_i}$ be 
an analytic open subset of $C$ such that 
$t_i \in U^{\mathrm{an}}_{t_i}$.
We assume that $U^{\mathrm{an}}_{q_j}$ and $U^{\mathrm{an}}_{t_i}$ are small enough.
We take an analytic coordinate $z_j$ on $U^{\mathrm{an}}_{q_j}$
such that it is independent of the moduli space 
$M_{X}^0$.
We denote also by $q_j$ the complex number so that 
the point $q_j$ on $C$ is defined by $z_j -q_j =0$.

\begin{definition}\label{2023_7_8_20_43}
Let $\{ U_{\alpha} \}_{\alpha}$ be an analytic open covering of $C$: $C = \bigcup_{\alpha} U_{\alpha}$ such that 
\begin{itemize}
\item[(i)]
$\sharp\{ i \mid t_i  \cap U_{\alpha} \neq \emptyset \} \le 1$ for any $\alpha$, and 
$\sharp\{ \alpha \mid t_i  \cap U_{\alpha} \neq \emptyset \} \le 1$ for any $i$,

\item[(ii)]  
$\sharp\{ j \mid q_j  \cap U_{\alpha} \neq \emptyset \} \le 1$ for any $\alpha$, and 
$\sharp\{ \alpha \mid q_j  \cap U_{\alpha} \neq \emptyset \} \le 1$ for any $j$,

\item[(iii)] 
$\Omega^1_{C}(D) $ is free on $U_{\alpha}$ for any $\alpha$,
that is, $\Omega^1_{C}(D)|_{U_{\alpha}} \cong \mathcal{O}_{U_{\alpha}} $, 

\item[(iv)] 
$U_{\alpha_{t_i}} = U^{\mathrm{an}}_{t_i}$ and
$U_{\alpha_{q_j}} = U^{\mathrm{an}}_{q_j}$.
\end{itemize}
Here we denote by $\alpha_{t_i}$ 
the index $\alpha $ such that $t_i \in U_{\alpha}$, and
by $\alpha_{q_j}$ the index
$\alpha $ such that $q_j \in U_{\alpha}$.

\end{definition}

We fix trivializations
$\omega_{\alpha} \colon 
\mathcal{O}_{U_{\alpha}}\xrightarrow{\ \sim \ }   \Omega^1_{C}(D)|_{U_{\alpha}}$
of $\Omega^1_{C}(D)$.
We assume that $\omega_{\alpha}$ is independent of the moduli space 
$M_{X}^0$.
By using $\omega_{\alpha}$, we have
$\omega^{-1}_{\alpha} \colon 
\mathcal{O}_{U_{\alpha}} \xrightarrow{\ \sim \ }  (\Omega^1_{C}(D))^{-1}|_{U_{\alpha}}$.
By the trivializations, we have trivializations
$\varphi_{\alpha}^{\mathrm{norm}} \colon  
\mathcal{O}^{\oplus 2}_{U_{\alpha}}
\xrightarrow{\ \sim \ }
E_0 |_{U_{\alpha}} $
of $E_0$.
Assume that the connection matrices $A_{\alpha}^{\mathrm{norm}}$
of $\nabla_0$ associated to $\varphi_{\alpha}^{\mathrm{norm}}$
are 
\begin{equation}\label{2023_4_5_11_56}
A_{\alpha}^{\mathrm{norm}}=
\begin{pmatrix}
0 & \beta_{\alpha} \\
\gamma_{\alpha} & \delta_{\alpha}
\end{pmatrix},
\end{equation}
where $\beta_{\alpha},\delta_{\alpha} \in \Omega_C^1(D+B)|_{U_{\alpha}}$
are determined by
$\{(q_j, \mathrm{res}_{q_j} (\beta))\}_{j=1,2,\ldots,N}$
(see Lemma \ref{lem:independence}).
The 1-form 
$\gamma_{\alpha} \in \Omega_C^1(D)|_{U_{\alpha}}$ 
is the image of $1$ under the composition
$$
\mathcal{O}_{U_{\alpha}} \xrightarrow{ \ \sim \ }
(\Omega^1_{C}(D))^{-1} \otimes \Omega^1_{C}(D)  |_{U_{\alpha}} 
\xrightarrow{\omega_{\alpha}\otimes 1 }
\mathcal{O}_{U_{\alpha}}\otimes \Omega^1_{C}(D).
$$
In particular, 
$\gamma_{\alpha}$
is independent of the moduli space $M_{X}^0$ 
for any $\alpha$.
The polar part of $A_{\alpha_{t_i}}^{\mathrm{norm}}$
at $t_i$ is
independent of the moduli space $M_{X}^0$ 
for any $i$.
We set
\begin{equation}\label{2023_7_12_16_04}
\zeta_j := \frac{\mathrm{res}_{q_j} (\beta)}{ \gamma_{\alpha_{q_j}}|_{q_j} } \in  \mathbb{C}
\quad 
\text{ for $j=1,2,\ldots,N$.}
\end{equation}
Here $\beta \in H^0(C,(\Omega^1_C)^{\otimes 2}(2D+B))$ 
is the $(1,2)$-entry of \eqref{eq:normal_form}.
Notice that $\beta|_{U_{\alpha}}  = \beta_{\alpha}\gamma_{\alpha}$,
where $\beta_{\alpha}$ and $\gamma_{\alpha}$ are in \eqref{2023_4_5_11_56}.
So we have
$$
\mathrm{res}_{q_j} ( A_{\alpha_{q_j}}^{\mathrm{norm}} ) =
\begin{pmatrix}
0 & \zeta_j \\
0 &1
\end{pmatrix}\quad 
\text{ for $j=1,2,\ldots,N$.}
$$

\begin{definition}\label{2023_7_8_23_29}
We define other trivializations 
$\varphi_{\alpha}^{\mathrm{App},0}
\colon \mathcal{O}^{\oplus 2}_{U_{\alpha}}
\xrightarrow{\ \sim \ }
E_0 |_{U_{\alpha}}$
of $E_0$ for each $\alpha$ as follows:
\begin{itemize}

\item[(i)] When $\alpha = \alpha_{q_j}$, 
we take a trivialization $\varphi_{\alpha}^{\mathrm{App},0}$
as 
\begin{equation*}\label{2023_3_16_18_38}
\varphi_{\alpha}^{\mathrm{App},0} 
= 
\varphi_{\alpha}^{\mathrm{norm}} \circ
\begin{pmatrix}
1 &\zeta_j \\
 0 & 1
\end{pmatrix}.
\end{equation*}
Note that this triangular matrix appeared in \eqref{diag_by_conj_by_triangular}.

\item[(ii)] Otherwise,
we take a trivialization $\varphi_{\alpha}^{\mathrm{App},0}$
as $\varphi_{\alpha}^{\mathrm{App},0} 
= \varphi_{\alpha}^{\mathrm{norm}}$.

\end{itemize}
\end{definition}

Let $A_{\alpha}^{\mathrm{App},0}$ be
the connection matrix
of $\nabla_0$ associated to $\varphi_{\alpha}^{\mathrm{App},0}$,
that is,
$$
(\varphi_{\alpha}^{\mathrm{App},0} )^{-1}
\circ ( \phi_{\nabla}^* \nabla ) \circ
\varphi_{\alpha}^{\mathrm{App},0}  
= \operatorname{d} + A_{\alpha}^{\mathrm{App},0}.
$$
We have that
$$
A_{\alpha}^{\mathrm{App},0}=
\begin{cases}
\begin{pmatrix}   - \zeta_j \gamma_{\alpha}  
& \beta_{\alpha}-\zeta_j\delta_{\alpha}-\zeta_j^2\gamma_{\alpha}\\ 
\gamma_{\alpha} & \delta_{\alpha} + \zeta_j\gamma_{\alpha} \end{pmatrix} 
& \text{when $\alpha = \alpha_{q_j}$} \\
\begin{pmatrix}
0 & \beta_{\alpha} \\
\gamma_{\alpha} & \delta_{\alpha}
\end{pmatrix}  & \text{otherwise}.
\end{cases}
$$
We have 
$$
\mathrm{res}_{q_j} ( A_{\alpha_{q_j}}^{\mathrm{App},0} ) =
\begin{pmatrix}
0 & 0 \\
0 &1
\end{pmatrix}\quad 
\text{ for $j=1,2,\ldots,N$.}
$$

Now we define trivializations of $E$
by using $\phi_{\nabla} \colon E_0 \rightarrow E$ in \eqref{eq:phi_nabla}
and the trivialization of $E_0$ in Definition \ref{2023_7_8_23_29}.

\begin{definition}\label{2023_7_2_11_52}
Now we define trivialization
$\varphi^{\mathrm{App}}_{\alpha} \colon 
\mathcal{O}^{\oplus 2}_{U_{\alpha}}
 \xrightarrow{\ \sim \ } 
E|_{U_{\alpha}}$
of $E$ for the open covering $\{ U_{\alpha} \}_{\alpha}$ 
in Definition \ref{2023_7_8_20_43} as follows.
\begin{itemize}
\item[(i)] 
When $\alpha = \alpha_{q_j}$, 
we take a trivialization $\varphi^{\mathrm{App}}_{\alpha}$
so that 
$$
(\varphi^{\mathrm{App}}_{\alpha} )^{-1}\circ \phi_{\nabla}|_{U_{\alpha}} \circ 
\varphi^{\mathrm{App},0}_{\alpha}
=\begin{pmatrix}
1 & 0 \\
 0 & z_j -q_j
\end{pmatrix}.
$$

\item[(ii)] When $\alpha = \alpha_{t_i}$, 
we take $g^{t_i}_{\alpha} \in \mathrm{Aut}(\mathcal{O}_{U_{\alpha}}^{\oplus 2})$
so that 
the polar part of $(g^{t_i}_{\alpha})^{-1} A_{\alpha}^{\mathrm{norm}} g^{t_i}_{\alpha}$
is diagonal at $m_i[t_i]$.
We take a trivialization $\varphi^{\mathrm{App}}_{\alpha}$
as 
$$
\varphi^{\mathrm{App}}_{\alpha} = 
\phi_{\nabla}|_{U_{\alpha}} \circ 
\varphi_{\alpha}^{\mathrm{norm}} \circ g^{t_i}_{\alpha}.
$$
Here remark that $\phi_{\nabla}|_{U_{\alpha}}$ is invertible.
Since 
the polar part of $A_{\alpha_{t_i}}^{\mathrm{norm}}$
at $t_i$ is
independent of the moduli space $M_{X}^0$,
we may assume that $(g^{t_i}_{\alpha})_{< m_i}$
is
independent of the moduli space $M_{X}^0$.
Here we define $(g^{t_i}_{\alpha})_{< m_i}$ so that
$g^{t_i}_{\alpha} = (g^{t_i}_{\alpha})_{< m_i} + O(z_i^{m_i})$.

\item[(iii)] Otherwise, 
we take a trivialization $\varphi^{\mathrm{App}}_{\alpha}$
so that 
$$
(\varphi^{\mathrm{App}}_{\alpha} )^{-1}
\circ \phi_{\nabla}|_{U_{\alpha}} \circ 
\varphi^{\mathrm{norm}}_{\alpha}
=\begin{pmatrix}
1 &0 \\
 0 & 1
\end{pmatrix}.
$$
Since $\phi_{\nabla}|_{U_{\alpha}}$ is invertible in this case, 
$\varphi^{\mathrm{App}}_{\alpha}  
= \phi_{\nabla}|_{U_{\alpha}} \circ 
\varphi^{\mathrm{norm}}_{\alpha}$.

\end{itemize}
\end{definition}

Let $A_{\alpha}$ be 
the connection matrix
of $\nabla$ associated to $\varphi^{\text{App}}_{\alpha}$, 
that is
$$
(\varphi^{\text{App}}_{\alpha}  )^{-1}
\circ \nabla \circ 
\varphi^{\text{App}}_{\alpha}   
= \operatorname{d} + A_{\alpha}.
$$
We have that
\begin{equation}\label{2023_3_6_18_19}
A_{\alpha}=
\begin{cases}
\begin{pmatrix}   - \zeta_j \gamma_{\alpha} &
 \frac{\beta_{\alpha}-\zeta_j\delta_{\alpha}-\zeta_j^2\gamma_{\alpha}}{z_j -q_j}\\ 
(z_j -q_j)\gamma_{\alpha} & 
\delta_{\alpha} + \zeta_j\gamma_{\alpha}
- \frac{\operatorname{d}\! z_j}{z_j-q_j} \end{pmatrix}
& \text{when $\alpha = \alpha_{q_j}$} \\
\omega_i(X )  + [\text{holo.\ part}]
& \text{when $\alpha = \alpha_{t_i}$} \\
\begin{pmatrix}
0 & \beta_{\alpha} \\
\gamma_{\alpha} & \delta_{\alpha}
\end{pmatrix}  & \text{otherwise}.
\end{cases}
\end{equation}
Here $\omega_i(X )$ 
is the 1-form defined in \eqref{2023_4_10_19_32}.
The connection matrix $A_{\alpha_{q_j}}$ on $U_{\alpha_{q_j}}$ appeared in
\eqref{conn_matrix_at_qj}.
The connection matrix $A_{\alpha_{q_j}}$ has no pole at $q_j$ for any $j=1,2,\ldots,N$,
since  $\beta_{\alpha},\delta_{\alpha}$
are determined by Lemma \ref{lem:independence}.
We have considered diagonalization of the polar part of 
the connection $(E,\nabla)$ at each $t_i$. 
The reason why we consider diagonalization of the polar parts
is that we use the connection matrix \eqref{2023_3_6_18_19}
to calculate an infinitesimal deformation of $(E,\nabla)$.
So we will calculate variations of
the transition functions with respect to the trivializations in Definition \ref{2023_7_2_11_52}
and 
variations of the connection matrices \eqref{2023_3_6_18_19}.
These are elements of $\mathcal{F}^0$ and $\mathcal{F}^1$
of \eqref{2023_7_4_14_18}, respectively.
To be elements of $\mathcal{F}^0$ and $\mathcal{F}^1$, 
we need the compatibility with the quasi-parabolic structure.
However, this compatibility follows directly from diagonalization of the polar parts.

\subsection{Descriptions of the cocycles of an infinitesimal deformation}\label{2023_7_12_16_39}

Let
$\boldsymbol{\Omega} (D) \rightarrow C$ 
be the total space of $\Omega^1_C(D)$.
By the argument as in Lemma \ref{2023_7_11_12_53},
we may define a map 
\begin{equation}\label{2023_7_12_16_02}
\begin{aligned}
f_{\mathrm{App},0} \colon M_{X}^0
&\longrightarrow \mathrm{Sym}^{N}(\boldsymbol{\Omega} (D)) \\
(E,\nabla  ) & \longmapsto \left\{
(q_j, \mathrm{res}_{q_j} (\beta))
\right\}_{j=1,2,\ldots,N}.
\end{aligned}
\end{equation}
Here $\beta \in H^0(C, (\Omega^1_C)^{\otimes 2} (2D+B))$ 
is the $(1,2)$-entry of \eqref{eq:normal_form} 
and $\mathrm{res}_{q_j} (\beta) \in \Omega_C^1(D)|_{q_j}$.
We take an analytic open subset $V$ of 
$M_{X}^0$.
For the analytic open subset $V$,
we assume that we may define a composition
$$
\begin{aligned}
V \longrightarrow f_{\mathrm{App},0} (V) 
& \longrightarrow \mathrm{Sym}^{N}(\mathbb{C}_{(q,\zeta)}^2)  \\
(E,\nabla  ) \longmapsto
\left\{ (q_j, \mathrm{res}_{q_j} (\beta)) 
\right\}_{j=1,2,\ldots,N}
& \longmapsto \left\{
(q_j, \zeta_j)
\right\}_{j=1,2,\ldots,N},
\end{aligned}
$$
where $\zeta_j$ is defined in \eqref{2023_7_12_16_04},
and the image of $V$ under the composition is isomorphic to 
some analytic open subset of $\mathbb{C}_{(q,\zeta)}^{2N}$.
Let $U_{(q,\zeta)}$ be such an analytic open subset of $\mathbb{C}_{(q,\zeta)}^{2N}$.
So we have a map 
\begin{equation}\label{2023_7_12_16_26}
\begin{aligned}
M_{X}^0 \supset V
 &\longrightarrow U_{(q,\zeta)} \subset
\mathbb{C}_{(q,\zeta)}^{2N}  \\
(E,\nabla ) &\longmapsto 
(q_1,\ldots,q_N, \zeta_1,\ldots,\zeta_N),
\end{aligned}
\end{equation}
which are coordinates that we will use in this subsection.
We consider the family of $(E,\nabla,\{ l^{(i)} \})$
parametrized by $U_{(q,\zeta)}$
such that this family induces the inverse map of the map $V \rightarrow U_{(q,\zeta)}$.
Here this family is constructed by Lemma \ref{lem:independence}.
By using the trivializations
$\{\varphi^{\text{App}}_{\alpha}\}_{\alpha}$ of $E$ in Definition \ref{2023_7_2_11_52}, 
we have transition functions and connection matrices 
of the family of $(E,\nabla,\{ l^{(i)} \})$
parametrized by $U_{(q,\zeta)}$.
Indeed, the transition function is
\begin{equation}\label{2023_7_2_11_55}
B_{\alpha\beta} :=( \varphi^{\text{App}}_{\alpha}|_{U_{\alpha\beta}})^{-1}
\circ \varphi^{\text{App}}_{\beta}|_{U_{\alpha\beta}}
 \colon \mathcal{O}^{\oplus 2}_{U_{\alpha \beta}}
\longrightarrow 
\mathcal{O}^{\oplus 2}_{U_{\alpha \beta}},
\end{equation}
and the connection matrix is as in \eqref{2023_3_6_18_19}.

Let $(q_j,\zeta_j)_j$ be a point on $U_{(q,\zeta)}$.
The purpose of this subsection is to describe the tangent map 
\begin{equation}\label{2023_7_12_16_59}
\begin{aligned}
T_{(q_j,\zeta_j)_j} \mathbb{C}_{(q,\zeta)}^{2N}
&\longrightarrow
T_{(E,\nabla ,  \{l^{(i)}\} )} M_{X}^0 
\cong {\bf H}^1 (\mathcal{F}^{\bullet}) \\
v &\longmapsto [(\{u_{\alpha\beta} (v)\}, 
\{ v_{\alpha} (v)\} )]
\end{aligned}
\end{equation}
induced by the inverse map of \eqref{2023_7_12_16_26}.
For this purpose, 
we will calculate the variations of the transition functions and the connection matrices 
parametrized by $U_{(q,\zeta)}$ with respect to 
the tangent vector $v$ in $U_{(q,\zeta)} \subset \mathbb{C}_{(q,\zeta)}^{2N}$.
By using these variations,
we will calculate the cocycles $(\{u_{\alpha\beta} (v)\}, 
\{ v_{\alpha} (v)\} )$ of the infinitesimal deformation 
of $(E,\nabla ,  \{l^{(i)}\} )$ with respect to $v$.

First, we calculate $u_{\alpha\beta} (v) \in \mathcal{F}^1(U_{\alpha\beta})$.
We consider the variation of $B_{\alpha\beta}$ in \eqref{2023_7_2_11_55} by $v$: 
$$
B_{\alpha\beta} ( \mathrm{id} + \epsilon B_{\alpha\beta}^{-1} v(B_{\alpha\beta})) 
\colon \mathcal{O}^{\oplus 2}_{U_{\alpha \beta}}
\longrightarrow 
\mathcal{O}^{\oplus 2}_{U_{\alpha \beta}} \otimes \mathbb{C}[\epsilon].
$$
Then $u_{\alpha\beta}(v)$ has the following description:
\begin{equation}\label{2023_7_4_14_16}
u_{\alpha\beta} (v)= 
\varphi^{\text{App}}_{\beta}|_{U_{\alpha\beta}}
\circ
\left( B_{\alpha\beta}^{-1} v(B_{\alpha\beta}) \right) \circ 
(\varphi^{\text{App}}_{\beta}|_{U_{\alpha\beta}})^{-1}.
\end{equation}

\begin{lem}\label{2023_3_20_20_28}
Let $I_{\mathrm{cov}}$ be the set of the indices of the open covering $\{ U_{\alpha} \}$
in Definition \ref{2023_7_8_20_43}.
We set
$I^t_{\mathrm{cov}} = \{\alpha_{t_1}, \ldots,\alpha_{t_\nu}\}$
and 
$I^q_{\mathrm{cov}} = \{\alpha_{q_1}, \ldots,\alpha_{q_{N}}\}$,
which are subsets of $I_{\mathrm{cov}}$.
For 
$v\in T_{(q_j,\zeta_j)_j} \mathbb{C}_{(q,\zeta)}^{2N}$,
we have the equality
\begin{equation}\label{2023_3_12_23_00}
u_{\alpha\beta} (v)
= 
\begin{cases}
0 & \alpha,\beta \in I_{\mathrm{cov}} \setminus (I^t_{\mathrm{cov}} \cup I^q_{\mathrm{cov}})  \\
\varphi^{\mathrm{App}}_{\alpha_{q_j}}|_{U_{\alpha\alpha_{q_j}}}
\circ
\begin{pmatrix}
0 & \frac{v(\zeta_j)}{z_j -q_j} \\
 0 &  \frac{v(q_j)}{z_j -q_j}
\end{pmatrix}\circ (\varphi^{\mathrm{App}}_{\alpha_{q_j}}|_{U_{\alpha\alpha_{q_j}}})^{-1} & 
\alpha \in I_{\mathrm{cov}} \setminus (I^t_{\mathrm{cov}} \cup I^q_{\mathrm{cov}}) ,
 \beta = \alpha_{q_j} \in I^q_{\mathrm{cov}} \\
\varphi^{\mathrm{App}}_{\alpha_{t_i}}|_{U_{\alpha\alpha_{t_i}}}
\circ
\left( (g^{t_i}_{\alpha_{t_i}})^{-1} v(g^{t_i}_{\alpha_{t_i}})  \right)
\circ (\varphi^{\mathrm{App}}_{\alpha_{t_i}}|_{U_{\alpha\alpha_{t_i}}})^{-1}
&\alpha \in I_{\mathrm{cov}}\setminus (I^t_{\mathrm{cov}} \cup I^q_{\mathrm{cov}})  ,
 \beta = \alpha_{t_i} \in I^t_{\mathrm{cov}},
\end{cases}
\end{equation}
and we have that 
\begin{equation}\label{2023_3_13_17_13}
(g^{t_i}_{\alpha_{t_i}})^{-1} v(g^{t_i}_{\alpha_{t_i}}) =O(z_i^{m_i} ).
\end{equation}
\end{lem}

\begin{proof}
Let $\alpha \in I_{\mathrm{cov}}\setminus (I^t_{\mathrm{cov}} \cup I^q_{\mathrm{cov}})$.
If $\beta \in I_{\mathrm{cov}}\setminus (I^t_{\mathrm{cov}} \cup I^q_{\mathrm{cov}})$,
then we have the following equalities:
$$
\begin{aligned}
B_{\alpha\beta} &=( \varphi^{\text{App}}_{\alpha}|_{U_{\alpha\beta}})^{-1}
\circ \varphi^{\text{App}}_{\beta}|_{U_{\alpha\beta}} \\
&=
(\varphi^{\mathrm{norm}}_{\alpha}|_{U_{\alpha\beta}})^{-1}
\circ (\phi_{\nabla}|_{U_{\alpha\beta}})^{-1}
\circ \phi_{\nabla}|_{U_{\alpha\beta}} \circ 
\varphi^{\mathrm{norm}}_{\beta}|_{U_{\alpha\beta}}\\
&= (\varphi^{\mathrm{norm}}_{\alpha}|_{U_{\alpha\beta}})^{-1}
\circ \varphi^{\mathrm{norm}}_{\beta} |_{U_{\alpha\beta}}= 
\begin{pmatrix}
1 & 0 \\
0 & ( (\omega_{\alpha}^{-1})^{-1} \circ  \omega_{\alpha_{q_j}}^{-1} )
\end{pmatrix}.
\end{aligned}
$$
Here $\omega^{-1}_{\alpha}$ is a trivialization 
$\mathcal{O}_{U_{\alpha}} \xrightarrow{\cong}   (\Omega^1_{C}(D))^{-1}|_{U_{\alpha}}$ 
for any $\alpha$.
Since 
$( (\omega_{\alpha}^{-1})^{-1} \circ  \omega_{\alpha_{q_j}}^{-1} )$
is independent of the moduli space 
$M_X^0$,
we have $v(B_{\alpha\beta})=0$.
So $u_{\alpha\beta}(v)=0$.

If $\beta = \alpha_{q_j}$,
then we have the following equalities:
\begin{equation}\label{2023_3_12_23_33}
\begin{aligned}
B_{\alpha\alpha_{q_j}} &=( \varphi^{\text{App}}_{\alpha}|_{U_{\alpha\alpha_{q_j}}})^{-1}
\circ \varphi^{\text{App}}_{\alpha_{q_j}}|_{U_{\alpha\alpha_{q_j}}} \\
&=
(\varphi^{\mathrm{App},0}_{\alpha}|_{U_{\alpha\alpha_{q_j}}})^{-1}
\circ (\phi_{\nabla}|_{U_{\alpha\alpha_{q_j}}})^{-1}
\circ \phi_{\nabla}|_{U_{\alpha\alpha_{q_j}}} \circ 
\varphi^{\mathrm{App},0}_{\alpha_{q_j}} |_{U_{\alpha\alpha_{q_j}}} \circ 
\begin{pmatrix}
1 & 0 \\
0& \frac{1}{z_j-q_j}
\end{pmatrix}\\
&= (\varphi^{\mathrm{App},0}_{\alpha}|_{U_{\alpha\alpha_{q_j}}})^{-1}
\circ \varphi^{\mathrm{App},0}_{\alpha_{q_j}}|_{U_{\alpha\alpha_{q_j}}}
\circ \begin{pmatrix}
1 & 0 \\
0& \frac{1}{z_j-q_j}
\end{pmatrix}\\
&=(\varphi^{\mathrm{norm}}_{\alpha}|_{U_{\alpha\alpha_{q_j}}})^{-1}
\circ \varphi^{\mathrm{norm}}_{\alpha_{q_j}}|_{U_{\alpha\alpha_{q_j}}}
\circ \begin{pmatrix}
1 & \zeta_j \\
0& 1
\end{pmatrix}
 \begin{pmatrix}
1 & 0 \\
0& \frac{1}{z_j-q_j}
\end{pmatrix}\\
&= 
\begin{pmatrix}
1 & 0 \\
0 & ( (\omega_{\alpha}^{-1})^{-1} \circ  \omega_{\alpha_{q_j}}^{-1} )
\end{pmatrix}
\begin{pmatrix}
1 & \frac{\zeta_j}{z_j-q_j} \\
0&   \frac{1}{z_j-q_j}
\end{pmatrix}.
\end{aligned}
\end{equation}
So we have 
$$
B_{\alpha\alpha_{q_j}}^{-1} v(B_{\alpha\alpha_{q_j}}) 
=\begin{pmatrix}
1 & -\zeta_j \\
0& z_j-q_j
\end{pmatrix}
\begin{pmatrix}
0 & \frac{v(\zeta_j) (z_j-q_j) + \zeta_j v(q_j) }{(z_j-q_j)^2} \\
0& - \frac{-v(q_j)}{(z_j-q_j)^2}
\end{pmatrix} =\begin{pmatrix}
0 & \frac{v(\zeta_j)}{z_j -q_j} \\
 0 &  \frac{v(q_j)}{z_j -q_j}
\end{pmatrix}.
$$

If $\beta = \alpha_{t_i}$,
then we have the following equalities:
$$
\begin{aligned}
B_{\alpha\alpha_{t_i}} &=( \varphi^{\text{App}}_{\alpha}|_{U_{\alpha\alpha_{t_i}}})^{-1}
\circ \varphi^{\text{App}}_{\alpha_{t_i}} |_{U_{\alpha\alpha_{t_i}}}\\
&=
(\varphi^{\mathrm{norm}}_{\alpha}|_{U_{\alpha\alpha_{t_i}}})^{-1}
\circ (\phi_{\nabla}|_{U_{\alpha\alpha_{t_i}}})^{-1}
\circ \phi_{\nabla}|_{U_{\alpha\alpha_{t_i}}} \circ 
\varphi^{\mathrm{norm}}_{\alpha_{t_i}} |_{U_{\alpha\alpha_{t_i}}} \circ 
g^{t_i}_{\alpha_{t_i}}\\
&= \begin{pmatrix}
1 & 0 \\
0 & ( (\omega_{\alpha}^{-1})^{-1} \circ  \omega_{\alpha_{q_j}}^{-1} )
\end{pmatrix}\circ 
g^{t_i}_{\alpha_{t_i}}.
\end{aligned}
$$
So we have 
$B_{\alpha\alpha_{t_i}}^{-1} v(B_{\alpha\alpha_{t_i}}) 
=(g^{t_i}_{\alpha_{t_i}})^{-1} v(g^{t_i}_{\alpha_{t_i}})$.
Since
$(g^{t_i}_{\alpha})_{< m_i}$
is independent of the moduli space $M^0_X$,
we have that $v(g^{t_i}_{\alpha_{t_i}}) = O(z_i^{m_i})$.
Finally, we have the statement of the lemma.
\end{proof}

Next we calculate $v_{\alpha}(v) \in \mathcal{F}^1(U_{\alpha})$
for 
$v\in T_{(q_j,\zeta_j)_j} \mathbb{C}_{(q,\zeta)}^{2N}$.
This is given by calculating the variation 
of the connection matrix $A_{\alpha}$
in \eqref{2023_3_6_18_19}
with respect to $v$.
So we have
\begin{equation}\label{2023_3_12_22_59}
v_{\alpha}(v)=
\begin{cases}
 \varphi^{\text{App}}_{\alpha}\circ
\begin{pmatrix}   - v( \zeta_j) \gamma_{\alpha}
 & v\left( \frac{\beta_{\alpha}-\zeta_j\delta_{\alpha}-\zeta_j^2\gamma_{\alpha}}{z_j -q_j} \right)\\ 
-v(q_j) \gamma_{\alpha} & 
 v(\delta_{\alpha} -\frac{\operatorname{d}\!z_j}{z_j-q_j} )
 +  v(\zeta_j)\gamma_{\alpha}  \end{pmatrix}
\circ ( \varphi^{\text{App}}_{\alpha})^{-1}
& \text{when $\alpha = \alpha_{q_j}$} \\
\varphi^{\text{App}}_{\alpha} \circ
\begin{pmatrix}
0 & v(\beta_{\alpha}) \\
0 & v(\delta_{\alpha})
\end{pmatrix} 
\circ( \varphi^{\text{App}}_{\alpha})^{-1}&
\text{when $\alpha \in I_{\mathrm{cov}}\setminus (I^t_{\mathrm{cov}} \cup I^q_{\mathrm{cov}})$}
\end{cases}.
\end{equation}
Here remark that 
$\gamma_{\alpha}$
is independent of the moduli space $M_{X}^0$ 
for any $\alpha$.
When $\alpha = \alpha_{t_i}$, 
we have that $v_{\alpha}(v)$ is holomorphic at $t_i$.

\subsection{Canonical coordinates}\label{subsect:Canonical_Coor}

Now we introduce canonical coordinates on 
$M_{X}^0$
with respect to the symplectic form \eqref{2020.11.7.15.48}.
We recall that we have set $N:= 4g+n-3$. 

Let 
$ \pi \colon  \boldsymbol{\Omega}(D) \rightarrow C$ 
and
$ \pi_0 \colon  \boldsymbol{\Omega} \rightarrow C$ 
be the total spaces of $\Omega^1_C(D)$ and $\Omega^1_C$,
respectively.
The total space $\boldsymbol{\Omega}$ has 
the Liouville symplectic form $\omega_{\text{Liouv}}$.
Since we have an isomorphism
$$
\pi_0^{-1} (C\setminus \mathrm{Supp}(D)) 
\xrightarrow{ \ \sim \ }
\pi^{-1} (C\setminus \mathrm{Supp}(D)) ,
$$
the Liouville symplectic form induces a 
symplectic form
$\pi^{-1} (C\setminus \mathrm{Supp}(D))$.
Let 
$\pi_N \colon 
\mathrm{Sym}^{N}(\boldsymbol{\Omega}(D))
\rightarrow 
\mathrm{Sym}^{N}(C)$
be the map induced by the map $ \pi \colon  \boldsymbol{\Omega}(D) \rightarrow C$.
We set 
$$
\mathrm{Sym}^{N}(\boldsymbol{\Omega}(D))_0 := 
\left\{ \{ q_1,\ldots,q_N \} \in 
\pi_N^{-1} (\mathrm{Sym}^{N}(C \setminus \mathrm{Supp}(D) ))
\ \middle| \  
q_{j_1} \neq q_{j_2} \ (j_1\neq j_2)
\right\}.
$$
Then $\mathrm{Sym}^{N}(\boldsymbol{\Omega}(D))_0$ has the induced 
symplectic form from the Liouville symplectic form.

\begin{remark}\label{2025_10_22_11_42}
We have a map 
$f_{\mathrm{App},0} \colon  M_{X}^0 
\rightarrow \mathrm{Sym}^{N}(\boldsymbol{\Omega}(D))_0$,
which is described in \eqref{2023_7_12_16_02}.
Notice that $M_{X}^0$ 
and $\mathrm{Sym}^{N}(\boldsymbol{\Omega}(D))_0 $ have 
symplectic forms. 
But by the explicit calculation as below, we realize that
this map $f_{\mathrm{App},0}$ does {\rm not} preserve these symplectic structures.
So $f_{\mathrm{App},0}$ does {\rm not} give canonical coordinates directly.
To give canonical coordinates,
we have to modify the map $f_{\mathrm{App},0}$ as follows.
\end{remark}

In the definition of
$f_{\mathrm{App},0}$ in \eqref{2023_7_12_16_02},
we used $\mathrm{res}_{q_j} (\beta) $.
To modify the map $f_{\mathrm{App},0}$,
we will use
\begin{equation}\label{2026_8_13_11_58}
\mathrm{res}_{q_j} (\beta) 
+ \left( \delta_{\alpha_{q_j}} -\frac{\operatorname{d}\!z_j}{z_j-q_j} \right)|_{q_j} ,
\end{equation}
which is the evaluation of the $(2,2)$-entry of the connection matrix
\eqref{2023_3_6_18_19} at $q_j$.
Indeed, this is an element of $\Omega_{C}^1(D)|_{q_j}$.
Remark that $ \delta_{\alpha_{q_j}}$ is a connection matrix of a connection $\delta$
on $(\Omega_C^1(D))^{-1}$.
Since $\frac{\operatorname{d}\!z_j}{z_j-q_j}$ is the principal part of 
$ \delta_{\alpha_{q_j}}$ at $q_j$, 
the difference 
$ \delta_{\alpha_{q_j}} -\frac{\operatorname{d}\!z_j}{z_j-q_j} $
is the holomorphic part of the connection matrix $\delta_{\alpha_{q_j}}$
at $q_j$,
which depends on the choice of local defining functions of $q_j$.
So \eqref{2026_8_13_11_58}
depends on the choice of 
local trivializations of $(\Omega_C^1(D))^{-1}$
and the choice of local defining functions of $q_j$.
For this reason,
it is inappropriate that 
\begin{equation}\label{2025_10_21_17_28}
(q_j, \mathrm{res}_{q_j} (\beta) 
+ \left( \delta_{\alpha_{q_j}} -\frac{\operatorname{d}\!z_j}{z_j-q_j} \right)|_{q_j} )
\end{equation}
is considered as a point of $\boldsymbol{\Omega}(D)$.
We may avoid this problem by considering a twisted cotangent bundle over $C$.
The issue with the construction of the twisted cotangent bundle
is that local defining functions of $q_j$ depend on the moduli space 
$M_{X}^0$, since $q_j$ are apparent singularities.

Now we will construct the appropriate twisted cotangent bundle over $C$. 
Let ${\rm pr}_i \colon C \times C \rightarrow C$ be the $i$-th projection for $i=1,2$.
Let $\Delta_C \subset C \times C$ be the diagonal. 
Since the dimension of $C$ is $1$, 
the diagonal $\Delta_C$ is a divisor on $C \times C$.
We set 
$$
L_{\delta} := {\rm pr}_1^* ((\Omega^1_C(D))^{-1} ) 
\otimes \mathcal{O}_{C\times C} (\Delta_C),
$$
which is a line bundle on $C\times C$.
Let $\tilde U \subset C \times C$ be an affine open subset.
We set
$$
\mathcal D_{L_{\delta},{\rm pr}_2}(\tilde U) :=  \mathrm{End}_{{\rm pr}_2^{-1}
\mathcal{O}_{C}(\tilde U)}(L_{\delta}(\tilde U),L_{\delta}(\tilde U)),
$$
which is an $\mathcal{O}_{C\times C}(\tilde U)$-bimodule.
We define $\mathcal D_{L_{\delta},{\rm pr}_2}^{i}(\tilde U)$ (for $i \ge -1$)
by induction
as follows:
$$
\begin{aligned}
\mathcal D_{L_{\delta},{\rm pr}_2}^{-1} (\tilde U)&:=0 \\
\mathcal D_{L_{\delta},{\rm pr}_2}^{i}(\tilde U)&:=
\left\{ s \in \mathcal D_{L_{\delta},{\rm pr}_2}(\tilde U)
\ \middle|\ 
\begin{array}{l}
r \cdot s - s \cdot r \in \mathcal D_{L_{\delta},{\rm pr}_2}^{i-1}(\tilde U)\\
\text{for any $r \in \mathcal{O}_{C\times C}(\tilde U)$} 
\end{array}
\right\}.
\end{aligned}
$$

\begin{definition}
We say ${\rm pr}_2^{-1}\mathcal{O}_C$-morphism 
$s \colon L_{\delta} \rightarrow L_{\delta}$
is a {\rm relative differential operator of degree $i$ for ${\rm pr}_2$} if 
for any affine subset $\tilde U \subset C\times C$ 
the corresponding morphism $s_U \colon L_{\delta}(\tilde U) \rightarrow L_{\delta}(\tilde U)$
lies in $\mathcal D_{L_{\delta},{\rm pr}_2}^{i}(\tilde U)$
and is not contained in 
$\mathcal D_{L_{\delta},{\rm pr}_2}^{i-1}(\tilde U)$.
\end{definition}

Let $\mathcal D_{L_{\delta},{\rm pr}_2}^{i}$ be 
the sheaf of relative differential operator of degree $\le i$ for ${\rm pr}_2$ 
acting on $L_{\delta}$.
We have $\mathcal D_{L_{\delta},{\rm pr}_2}^{i}/\mathcal D_{L_{\delta},{\rm pr}_2}^{i-1}
= S^i(T_{C\times C/C})$, 
where $S^i(T_{C\times C/ C})$ is the $i$-th symmetric product of 
$T_{C\times C/C}$,
which is a relative tangent sheaf for the second projection 
${\rm pr}_2 \colon C\times C \rightarrow C$.
We denote by $\mathrm{symb}_1$ 
the composition 
\[
\mathcal D_{L_{\delta},{\rm pr}_2}^{1}
\longrightarrow \mathcal D_{L_{\delta},{\rm pr}_2}^{1}/\mathcal D_{L_{\delta},{\rm pr}_2}^{0}
\xrightarrow{\ =\ }
T_{C\times C/C}.
\]
For the differential operator $v$ of degree $1$,
the image $\mathrm{symb}_1(v)$ is called the {\it symbol} of $v$.
Since $\mathcal D_{L_{\delta},{\rm pr}_2}^{0} = \mathcal{O}_{C\times C}$,
we have the following short exact sequence: 
\begin{equation*}\label{2023_3_24_19_28}
0\longrightarrow  \mathcal{O}_{C\times C}
\longrightarrow  \mathcal D_{L_{\delta},{\rm pr}_2}^{1}
\xrightarrow{\ \mathrm{symb}_1 \  } T_{C\times C/C}
\longrightarrow 
0.
\end{equation*}
If we restrict $ \mathcal D_{L_{\delta},{\rm pr}_2}^{1}$ on a fiber of ${\rm pr}_2$,
then the restriction of $\mathcal D_{L_{\delta},{\rm pr}_2}^{1}$ 
is naturally a Lie-algebroid, 
called the Atiyah algebroid of the $\mathbb{G}_m$-principal bundle 
$\operatorname{Tot} (T_C) \setminus 0$, where $0$ stands 
for the $0$-section; for details, see~\cite[Section~3.1.2]{LM}. 
We consider the subsheaf $T_{C\times C/C} (-{\rm pr}_1^*D) \subset T_{C\times C/C}$.
We set 
$$
\mathcal D_{L_{\delta},{\rm pr}_2}^{1}(D) 
:=  \mathrm{symb}_1^{-1} T_{C\times C/C} (-{\rm pr}_1^*D),
$$ 
which is an extension 
$$
0\longrightarrow  \mathcal{O}_{C\times C}
\longrightarrow  \mathcal D_{L_{\delta},{\rm pr}_2}^{1}(D)
\xrightarrow{\ \mathrm{symb}_1 \  } T_{C\times C/C} (-{\rm pr}_1^*D)
\longrightarrow 
0.
$$
We set
$$
\Omega_{{\rm pr}_2}^{1}(L_{\delta},D) := 
\left\{ \phi \in \mathcal D_{L_{\delta},{\rm pr}_2}^{1}(D)^{\vee} \ \middle|\ 
\langle \phi, 1_{\mathcal D_{L_{\delta},{\rm pr}_2}^{1}(D)} \rangle =1 \right\}.
$$

Let $\Omega^1_{C\times C/C} $ be a relative cotangent bundle 
over $C\times C$
for the second projection ${\rm pr}_2 \colon C\times C \rightarrow C$,
and we set $\Omega^1_{C\times C/C} ({\rm pr}_1^* D):= \Omega^1_{C\times C/C}
\otimes \mathcal{O}_{C\times C} ({\rm pr}_1^*D)$.
We have a natural map 
$\pi_{C\times C} \colon \Omega_{{\rm pr}_2}^{1}(L_{\delta},D) 
\rightarrow C \times C$.
This is an affine bundle on $C \times C$ 
modeled over $\Omega^1_{C\times C/C} ({\rm pr}_1^* D)$.
Now we will discuss it.
Let $\{ f_{\alpha\beta} \}$ and $\{ g_{\alpha\beta} \}$
be transition functions for 
${\rm pr}_1^* ((\Omega^1_C(D))^{-1} )$
and 
$\mathcal{O}_{C\times C} (\Delta_C)$,
respectively. 
So $\{ f_{\alpha\beta} \otimes g_{\alpha\beta} \}$
are transition functions for $L_{\delta}$.
By the construction of $\Omega_{{\rm pr}_2}^{1}(L_{\delta},D)$,
we may take an affine open covering $\{ \tilde U_{\alpha} \}$ of $C\times C$ 
such that there exist isomorphisms
$$
\xi_{\alpha} \colon \Omega_{{\rm pr}_2}^{1}(L_{\delta},D)_{\tilde U_\alpha } 
\xrightarrow{\ \cong  \ }
\Omega^1_{C\times C/C} ({\rm pr}_1^* D)_{\tilde U_{\alpha}}
$$
which satisfy $\xi_{\alpha} = \xi_{\alpha\beta} \circ \xi_{\beta}$.
Here $\xi_{\alpha\beta}$ is defined by
$$
\begin{aligned}
\xi_{\alpha\beta}\colon \Omega^1_{C\times C/C} ({\rm pr}_1^* D)_{\tilde U_{\alpha\beta}} 
&\longrightarrow 
\Omega^1_{C\times C/C} ({\rm pr}_1^* D)_{\tilde U_{\alpha\beta}} \\
\tilde \omega &\longmapsto 
\tilde \omega + \frac{\operatorname{d}_{{\rm pr}_2} f_{\alpha\beta}}{f_{\alpha\beta}}
+ \frac{\operatorname{d}_{{\rm pr}_2} g_{\alpha\beta}}{g_{\alpha\beta}}.
\end{aligned}
$$
Here $\operatorname{d}_{{\rm pr}_2} \colon \mathcal{O}_{C\times C}
\rightarrow {\rm pr}_1^* \Omega^1_{C} $
is the relative exterior derivative for the second projection 
${\rm pr}_2 \colon C\times C \rightarrow C$.

We denote the diagonal map $C \rightarrow C \times C$ by also $\Delta_C$.
We consider the map $\Delta_C \colon C \rightarrow C \times C$
as a section of the first projection ${\rm pr}_1 \colon C \times C \rightarrow C$.
We define $ \pi_C \colon \Omega_{C}^{1}(L_{\delta},D)
\rightarrow C$ by the pull-back of 
$\pi_{C\times C} \colon \Omega_{{\rm pr}_2}^{1}(L_{\delta},D) 
\rightarrow C \times C$ 
under the diagonal map $\Delta_C \colon C \rightarrow C \times C$:
\begin{equation}\label{2025_10_22_11_35}
\xymatrix{
\Omega_{C}^{1}(L_{\delta},D)  \ar[d]_-{\pi_{C}} \ar[r]
& \Omega_{{\rm pr}_2}^{1}(L_{\delta},D) \ar[d]^-{\pi_{C\times C}} \\ 
C \ar[r]^-{\Delta_C} & C \times C\rlap{.}
}
\end{equation}
By the construction of $\Omega_{C}^{1}(L_{\delta},D)$,
we may take an affine open covering $\{  U_{\alpha} \}$ of $C$ 
such that there exist isomorphisms
\begin{equation}\label{2025_10_21_17_33}
\xi^{\Delta}_{\alpha} \colon \Omega_{C}^{1}(L_{\delta},D)_{U_\alpha } 
\xrightarrow{\ \cong  \ }
\Omega^1_{C} (D)_{U_{\alpha}}
\end{equation}
which satisfy $\xi^{\Delta}_{\alpha} = \xi^{\Delta}_{\alpha\beta} \circ \xi^{\Delta}_{\beta}$.
Here $\xi_{\alpha\beta}$ is defined by
\begin{equation}\label{2025_10_21_17_35}
\begin{aligned}
\xi^{\Delta}_{\alpha\beta}\colon \Omega^1_{C} (D)_{U_{\alpha\beta}} 
&\longrightarrow 
\Omega^1_{C} (D)_{U_{\alpha\beta}} \\
\omega &\longmapsto 
\omega + \Delta_C^* \left(  \frac{\operatorname{d}_{{\rm pr}_2} f_{\alpha\beta}}{f_{\alpha\beta}} 
+ \frac{\operatorname{d}_{{\rm pr}_2} g_{\alpha\beta}}{g_{\alpha\beta}}\right) .
\end{aligned}
\end{equation}
Here $\{ f_{\alpha\beta} \}$ and $\{ g_{\alpha\beta} \}$
be transition functions for 
$p_1^* ((\Omega^1_C(D))^{-1} )$
and 
$\mathcal{O}_{C\times C} (\Delta_C)$
for $\{ U_{\alpha} \times U_{\alpha}\}_{\alpha}$. 
The affine bundle $\Omega_{C}^{1}(L_{\delta},D) \rightarrow C$
is the appropriate twisted cotangent bundle over $C$.
Indeed,
the choice of 
local trivializations of $(\Omega_C^1(D))^{-1}$
corresponds to the choice of $\{ f_{\alpha\beta}\}$.
The choice of local defining functions of $q_j$
corresponds to the choice of $\{ g_{\alpha\beta}\}$.
That is, the choice of the description of
\eqref{2025_10_21_17_28}
corresponds to the choice of the isomorphisms \eqref{2025_10_21_17_33}.
The transition when the choice changes is described in \eqref{2025_10_21_17_35}.
Let $\boldsymbol{\Omega} (L_{\delta},D)$ be
the total space of the twisted cotangent bundle $\Omega_{C}^{1}(L_{\delta},D)$.
It is appropriate that 
\eqref{2025_10_21_17_28}
is considered as a point of $\boldsymbol{\Omega} (L_{\delta},D)$.
We denote a generic element of this affine bundle by $(q, \tilde{p} )$.

\begin{definition}\label{2023_7_12_23_06}
We define the \emph{accessory parameter} $\tilde{p}_j$ 
associated to $(E,\nabla )$ at $q_j$ by 
$$
  \tilde{p}_j = 
  \mathrm{res}_{q_j} (\beta) 
+ \left( \delta_{\alpha_{q_j}} -\frac{\operatorname{d}\!z_j}{z_j-q_j} \right)|_{q_j}, 
$$
where $\beta \in H^0(C, (\Omega^1_C)^{\otimes 2} (2D+B))$ 
is the $(1,2)$-entry of \eqref{eq:normal_form}
and 
$\delta_{\alpha_{q_j}} \in \Omega^1_C(D+B)|_{U_{\alpha_{q_j}}}$
is the $(2,2)$-entry of the connection matrix \eqref{2023_4_5_11_56}
of the companion normal form on $U_{\alpha_{q_j}}$. 
The $N$-tuple $\left\{ \left( q_j , \tilde{p}_j \right) \right\}_{j=1,2,\ldots,N}$ 
will be called \emph{canonical coordinates} of $(E,\nabla )$. 
We let $f_{\mathrm{App}}$ be the map 
\begin{equation*}
\begin{aligned}
f_{\mathrm{App}} \colon M_{X}^0 &\longrightarrow
 \mathrm{Sym}^N (\boldsymbol{\Omega} (L_{\delta},D)) \\
(E,\nabla ,  \{l^{(i)}\} ) &\longmapsto \left\{ \left( q_j , 
\tilde{p}_j \right) \right\}_{j=1,2,\ldots,N}.
\end{aligned}
\end{equation*}
Remark that $\tilde{p}_j $ is a 1-form.
We define a corresponding quantity $p_j \in \mathbb{C}$ by
\begin{equation}\label{2023_7_12_16_43}
\tilde{p}_j  = p_j \operatorname{d}\!z_j.
\end{equation}
\end{definition}

Notice that the map $f_{\mathrm{App},0}$ in \eqref{2023_7_12_16_02} 
is defined by using only $\mathrm{res}_{q_j}(\beta)$.
The reason why we consider the twisted cotangent bundle 
$\boldsymbol{\Omega} (L_{\delta},D)$
is to justify 
$\left( \delta_{\alpha_{q_j}} -\frac{\operatorname{d}\!z_j}{z_j-q_j} \right)|_{q_j}$.
The next proposition shows that the quantities introduced 
in the definition may indeed be called coordinates.

\begin{prop}\label{prop:birational}
 The map $f_{\mathrm{App}}$ introduced in Definition~\ref{2023_7_12_23_06} is birational. 
\end{prop}

\begin{proof}
It follows from Proposition~\ref{prop:dimension} that the dimensions of the source and target of $f_{\mathrm{App}}$ agree. 
We therefore need to show two things: first, that $f_{\mathrm{App}}$ is rational, and second, that it admits an inverse over a Zariski open subset of 
$\mathrm{Sym}^N (\boldsymbol{\Omega}(L_{\delta},D))$. 

The first assertion is trivial, because the construction of the apparent singularities $q_j$ and their accessory parameters $\tilde{p}_j$ follow from algebraic arguments on certain Zariski open subsets. 

The key statement is existence of a generic inverse. 
This is now a variant of Lemma~\ref{lem:independence}.
Namely, fixing generic $\left\{ \left( q_j , \tilde{p}_j \right) \right\}_{j=1,2,\ldots,N}$, we must find a unique $(\delta, \beta)$. 
Since we have $\delta = \mathrm{tr}(\nabla_0)$, we get the expression 
$$
  \tilde{p}_j = \zeta_j \operatorname{d}\! z_j + \delta - \frac{\operatorname{d}\! z_j}{z_j}. 
$$
An algebraic manipulation shows that the constraint~\eqref{eq:apparent} expressing that the singularity at $q_j$ be apparent is equivalent to the holomorphicity and vanishing of the expression 
\begin{equation}\label{eq:quantum_characteristic}
 \beta + \delta \left( \tilde{p}_j + \frac{\operatorname{d}\! z_j}{z_j} \right) - \left( \tilde{p}_j + \frac{\operatorname{d}\! z_j}{z_j} \right)^2.
\end{equation}
We now study these conditions by taking the Laurent expansion of this expression with respect to $z_j$. 
We first observe that it clearly admits a pole of order at most $2$ at $q_j$, because $q_j\neq t_i$. 
Since $\delta$ has a simple pole with residue $1$, the term of degree $-2$ is $$
\left( \operatorname{d}\! z_j\right)^{\otimes 2}-\left( \operatorname{d}\! z_j\right)^{\otimes 2}=0. 
$$
So the pole is automatically at most simple. 

For the study of the residue, we need to introduce some notation: let us write 
\begin{align*}
  \delta_0 & = \frac{\operatorname{d}\! z_j}{z_j} + \delta_0^{(j)} \\
  \beta_0 & = \zeta_j \frac{\left( \operatorname{d}\! z_j\right)^{\otimes 2}}{z_j} + \beta_0^{(j)}
\end{align*}
for a holomorphic rank $1$ connection $\delta_0^{(j)}$ and a holomorphic quadratic differential $\beta_0^{(j)}$ on $U_{q_j}$. 
Then, the degree $-1$ part of~\eqref{eq:quantum_characteristic} is (up to a global factor $\operatorname{d}\! z_j$) 
$$
\zeta_j \operatorname{d}\! z_j + \tilde{p}_j + \left( \delta -  \frac{\operatorname{d}\! z_j}{z_j} \right) - 2  \tilde{p}_j = 0
$$
by the definition of $\tilde{p}_j$. 

Finally, to deal with the vanishing constraint, we make use of the same basis expansions for $\delta$ and $\beta$ as in Lemma~\ref{lem:independence}.  
Then, the conditions read as 
$$
  \sum_{k=1}^{N-g} b_k \nu_k (q_j) + \tilde{p}_j \sum_{l=1}^g d_l \omega_l (q_j) =  \left( \tilde{p}_j\right)^{\otimes 2} - \delta_0^{(j)} (q_j) \tilde{p}_j - \beta_0^{(j)} (q_j). 
$$
Now, the determinant of this linear system of $N$ equations (for $1\leq j \leq N$) in $N$ variables $b_1,\ldots, b_{N-g}, d_1, \ldots, d_g$ agrees with the determinant studied in Lemma~\ref{lem:independence}, up to replacing each occurrence of $\zeta_j \operatorname{d}\! z_j$ by $\tilde{p}_j$. 
The end of the proof then follows word by word the method of  Lemma~\ref{lem:independence}. 
\end{proof}

\begin{remark}
 The expression~\eqref{eq:quantum_characteristic} has variables $\tilde{p}_j$ in the twisted cotangent sheaf rather than the ordinary cotangent sheaf. 
 The quadratic polynomial of $\tilde{p}_j$ can be viewed as the characteristic polynomial of the connection matrix of $\nabla_0$. 
 Thus, in a sense the vanishing condition on~\eqref{eq:quantum_characteristic} may be interpreted as the requirement that $\tilde{p}_j$ lie on the \emph{quantum spectral curve} of $\nabla_0$, see e.g.~\cite{DumMul}. 
\end{remark}

\begin{remark}
This $\tilde p_j$ is just the evaluation of 
the $(2,2)$-entry of the connection matrix $A_{\alpha_{q_j}}$ in
\eqref{2023_3_6_18_19} at $q_j$.
Note that the $(2,1)$-entry of this connection matrix $A_{\alpha_{q_j}}$ at $q_j$
vanishes.
So $\tilde p_j$ is an ``eigenvalue'' of $\nabla$ at $q_j$.
(On the other hand, $\zeta_j$ is an ``eigenvector'' of $\nabla_0$ at $q_j$).
This fact means that 
the coordinates $(q_j , \tilde p_j)_j$ 
are an analog of the coordinates on the moduli space of (parabolic) Higgs bundles 
given as in \cite{GNR} and \cite{Hurt}.
The coordinates on the moduli space of (parabolic) Higgs bundles 
are by using the BNR correspondence \cite{BNR}.   (See Section \ref{Sec:Higgs}).  
\end{remark}

Let 
$\pi_{C}^N \colon 
\mathrm{Sym}^{N}(\boldsymbol{\Omega}(L_{\delta},D))
\rightarrow 
\mathrm{Sym}^{N}(C)$
be the map induced by the map 
$ \pi_{C} \colon  \boldsymbol{\Omega}(L_{\delta},D)\rightarrow C$.
We set 
$$
\mathrm{Sym}^{N}(\boldsymbol{\Omega}(L_{\delta},D))_0 :=
\left\{ \{ (q_j, \tilde{p}_j) \}_{j=1}^N \in 
(\pi_{C}^N)^{-1} (\mathrm{Sym}^{N}(C \setminus \mathrm{Supp}(D) ))
\ \middle| \  
q_{j_1} \neq q_{j_2} \ (j_1\neq j_2)
\right\}.
$$
Then $\mathrm{Sym}^{N}(\boldsymbol{\Omega}(L_{\delta},D))_0$ has the induced 
symplectic form from the Liouville symplectic form.
Notice that by construction the image of $M_{X}^0$ under the map $f_{\mathrm{App}}$ is contained in $\mathrm{Sym}^{N}(\boldsymbol{\Omega}(L_{\delta},D))_0$.

\begin{thm}\label{2023_8_22_12_09}
Let $\omega$ be the symplectic form on 
$M_{X}^0$
defined by \eqref{2020.11.7.15.48}.
The pull-back of the symplectic form on 
$ \mathrm{Sym}^N (\boldsymbol{\Omega}(L_{\delta},D))_0$ under the map
$$
f_{\mathrm{App}} \colon M_{X}^0
\longrightarrow \mathrm{Sym}^{N}(\boldsymbol{\Omega}(L_{\delta},D))_0
$$
in Definition \ref{2023_7_12_23_06}
coincides with $\omega$.
\end{thm}

\begin{proof}
Let $V$ be an analytic open subset of $M_{X}^0$
as in Section \ref{2023_7_12_16_39}.
Moreover, we assume that we may define a composition
$$
\begin{aligned}
V \longrightarrow f_{\mathrm{App}} (V) 
& \longrightarrow \mathrm{Sym}^{N}(\mathbb{C}_{(q,p)}^2)  \\
(E,\nabla  ) \longmapsto
f_{\mathrm{App}}(E,\nabla  ) & \longmapsto \left\{
(q_j, p_j)
\right\}_{j=1,2,\ldots,N},
\end{aligned}
$$
where $p_j$ is defined in \eqref{2023_7_12_16_43},
and the image of $V$ under the composition is isomorphic to 
some analytic open subset of $\mathbb{C}_{(q,p)}^{2N}$.
Let $U_{(q,p)}$ be such an analytic open subset of $\mathbb{C}_{(q,p)}^{2N}$.
We denote by $f_2$ the map 
\begin{equation*}
\begin{aligned}
M_{X}^0 \supset V
 &\longrightarrow U_{(q,p)} \subset
\mathbb{C}_{(q,\zeta)}^{2N}  \\
(E,\nabla) &\longmapsto 
(q_1,\ldots,q_N, p_1,\ldots, p_N).
\end{aligned}
\end{equation*}
We consider the following maps
$$
\xymatrix{
U_{(q,\zeta)} & \ar[l]_-{f_1}^-{\sim}  V \ar[r]^-{f_2} 
& U_{(q,p)} \rlap{.}
}
$$
Here $f_1 \colon V \xrightarrow{\sim} U_{(q,\zeta)}$ is the isomorphism
\eqref{2023_7_12_16_26}.
The symplectic structure on $U_{(q,p)}$ induced by 
the symplectic structure on $ \mathrm{Sym}^N (\boldsymbol{\Omega}(L_{\delta},D))$
is $\sum^N_{j=1} d p_j \wedge dq_j$.
We will show that 
$$
(f_1^{-1})^* (\omega|_V) =
(f_2 \circ f_1^{-1})^*\left( \sum^N_{j=1} d p_j \wedge dq_j \right).
$$
Let $v,v'$ be elements of 
$T_{(q_j,\zeta_j)_j } U_{(q,\zeta)} $
for $(q_j,\zeta_j)_j \in U_{(q,\zeta)}$.
We will use the description of the tangent map 
\eqref{2023_7_12_16_59} of $f^{-1}_1 \colon U_{(q,\zeta)}  \rightarrow V$.
That is, we calculate $(f_1^{-1})^* (\omega|_V)$
by applying the descriptions 
\eqref{2023_3_12_23_00}
and
\eqref{2023_3_12_22_59}
of $u_{\alpha\beta}(v)$ and $v_{\alpha}(v)$, respectively.

First we consider $\{ u_{\alpha\beta}(v)u_{\beta\gamma} (v')\}_{\alpha\beta\gamma}$.
Remark that $U_{\alpha_{q_{j_1}}} \cap U_{\alpha_{q_{j_2}}} = \emptyset$ 
for any $j_1$ and $j_2$,
$U_{\alpha_{t_{i_1}}} \cap U_{\alpha_{t_{i_2}}} = \emptyset$ 
for any $i_1$ and $i_2$,
and 
$U_{\alpha_{q_{j}}} \cap U_{\alpha_{t_{i}}} = \emptyset$ 
for any $j$ and $i$.
Then we have $u_{\alpha\beta}(v)u_{\beta\gamma}(v')=0$
by Lemma \ref{2023_3_20_20_28}.
So we may take a representative of the class in the pairing \eqref{2020.11.7.15.48} 
so that 
$$
[- \{ \mathrm{tr} (u_{\alpha\beta}(v)  v_{\beta}(v')) 
- \mathrm{tr} (v_{\alpha} (v)  u_{\alpha\beta}(v'))\}_{\alpha\beta} ]
\in H^1(C,\Omega^1_C) \cong \mathbb{C}.
$$

Now we calculate 
$\mathrm{tr} (u_{\alpha\beta}(v) v_{\beta}(v')) 
- \mathrm{tr} (v_{\alpha} (v)  u_{\alpha\beta}(v'))$.
If $\alpha \in I_{\mathrm{cov}}\setminus (I^t_{\mathrm{cov}} \cup I^q_{\mathrm{cov}})$
and $\beta = \alpha_{q_j}$,
then, by applying \eqref{2023_3_12_23_00} and
\eqref{2023_3_12_22_59},
we have the following equalities
\begin{equation}\label{2023_3_12_23_44}
\begin{aligned}
& \mathrm{tr} (u_{\alpha\alpha_{q_j}}(v) v_{\alpha_{q_j}}(v')) 
- \mathrm{tr} (v_{\alpha} (v)  u_{\alpha\alpha_{q_j}}(v')) \\
&=
\mathrm{tr} \left( 
\begin{pmatrix}
0 & \frac{v(\zeta_j)}{z_j -q_j} \\
 0 &  \frac{v(q_j)}{z_j -q_j}
\end{pmatrix} 
\begin{pmatrix} 
  * &  * \\ 
-v'(q_j) \gamma_{\alpha_{q_j}} & 
v'(\delta_{\alpha_{q_j}} -\frac{\operatorname{d}\!z_j}{z_j-q_j})
+ v'(\zeta_j)\gamma_{\alpha_{q_j}} 
 \end{pmatrix}
\right) \\
&\qquad - \mathrm{tr} \left(
\begin{pmatrix}   * &  * \\ 
0 & v(\delta_{\alpha})
\end{pmatrix}
\begin{pmatrix}
0 & 
v'(\zeta_j) + \frac{\zeta_jv'(q_j)}{z_j -q_j} \\
 0 &  \frac{v'(q_j)}{z_j -q_j}
\end{pmatrix}\right) \\
&= -\frac{v(\zeta_j)v'(q_j) \gamma_{\alpha_{q_j}}}{z_j -q_j} 
+ \frac{v(q_j) \left(v'(\delta_{\alpha_{q_j}} -\frac{\operatorname{d}\!z_j}{z_j-q_j})
+v'(\zeta_i)\gamma_{\alpha_{q_j}} \right)}{z_j -q_j}  
- \frac{v'(q_j) v(\delta_{\alpha})  }{z_j -q_j}\\
&= -\frac{ \left( v(\zeta_j) \gamma_{\alpha_{q_j}} +
v(\delta_{\alpha}) \right) v'(q_j) }{z_j -q_j} 
+ \frac{v(q_j) \left(v'(\delta_{\alpha_{q_j}} -\frac{\operatorname{d}\!z_j}{z_j-q_j})
+v'(\zeta_i)\gamma_{\alpha_{q_j}} \right)}{z_j -q_j}  .
\end{aligned}
\end{equation}
Since the trivializations of $(\Omega_C^1(D))^{-1}$
are independent of the moduli space $M_X^0$,
we have 
$$
v(\delta_{\alpha_{q_j}} -\frac{\operatorname{d}\!z_j}{z_j-q_j})
-v(\delta_{\alpha})
 = - \frac{v(q_j)\operatorname{d}\!z_j}{(z_j -q_j)^2} , $$
We have that 
\begin{equation}\label{2023_3_13_9_7}
\begin{aligned}
& \mathrm{tr} (u_{\alpha\alpha_{q_j}}(v) v_{\alpha_{q_j}}(v')) 
- \mathrm{tr} (v_{\alpha} (v) u_{\alpha\alpha_{q_j}}(v')) \\
&=
-\frac{ \left( v(\zeta_j) \gamma_{\alpha_{q_j}} +
v(\delta_{\alpha_{q_j}} -\frac{\operatorname{d}\!z_j}{z_j-q_j})  \right) v'(q_j) }{z_j -q_j} 
+ \frac{v(q_j) \left(v'(\delta_{\alpha_{q_j}} -\frac{\operatorname{d}\!z_j}{z_j-q_j})
+v'(\zeta_i)\gamma_{\alpha_{q_j}} \right)}{z_j -q_j}  
- \frac{ v(q_j )  v'(q_j) \operatorname{d}\! z_j}{(z_j -q_j)^3} .
\end{aligned}
\end{equation} 
So we may extend the 1-form 
$$
\mathrm{tr} (u_{\alpha\alpha_{q_j}}(v) v_{\alpha_{q_j}}(v')) 
- \mathrm{tr} (v_{\alpha} (v) u_{\alpha\alpha_{q_j}}(v'))
$$ 
from $U_{\alpha\alpha_{q_j}}$ to $U_{\alpha_{q_j}}$ by \eqref{2023_3_13_9_7}.
Then we have a meromorphic 1-form defined on $U_{\alpha_{q_j}}$,
which has a pole at $q_j$.
We denote by $\omega_{\alpha_{q_j}}(v,v') $
the meromorphic 1-form defined on $U_{\alpha_{q_j}}$.

Next we consider the case
where 
$\alpha \in I_{\mathrm{cov}}\setminus (I^t_{\mathrm{cov}} \cup I^q_{\mathrm{cov}})$
and $\beta = \alpha_{t_i}$.
We have the following equalities
$$
\begin{aligned}
& \mathrm{tr} (u_{\alpha\alpha_{t_i}}(v) v_{\alpha_{t_i}}(v')) 
- \mathrm{tr} (v_{\alpha} (v)  u_{\alpha\alpha_{t_i}}(v')) \\
&=
\mathrm{tr} \left( 
(g^{t_i}_{\alpha_{t_i}})^{-1} v(g^{t_i}_{\alpha_{t_i}})
v'(A_{\alpha_{t_i}})
\right)  - \mathrm{tr} \left(
\left( (g^{t_i}_{\alpha_{t_i}} )^{-1} v(A_{\alpha}) g^{t_i}_{\alpha_{t_i}}\right)
(g^{t_i}_{\alpha_{t_i}})^{-1} v'(g^{t_i}_{\alpha_{t_i}})
\right)
\end{aligned}
$$
We have the cocycle condition 
$$
\begin{aligned}
&v(A_{\alpha_{t_i}}) - (g^{t_i}_{\alpha_{t_i}})^{-1}v(A_{\alpha}) (g^{t_i}_{\alpha_{t_i}}) \\
&= (\operatorname{d} + A_{\alpha_{t_i}})
 \circ \left( (g^{t_i}_{\alpha_{t_i}})^{-1} v(g^{t_i}_{\alpha_{t_i}}) \right)
-\left( (g^{t_i}_{\alpha_{t_i}})^{-1} v(g^{t_i}_{\alpha_{t_i}}) \right) \circ 
(\operatorname{d} + A_{\alpha_{t_i}})\\
&= \operatorname{d}  \left( (g^{t_i}_{\alpha_{t_i}})^{-1} v(g^{t_i}_{\alpha_{t_i}}) \right)
+
\left[ \, A_{\alpha_{t_i}} , \, 
\left( (g^{t_i}_{\alpha_{t_i}})^{-1} v(g^{t_i}_{\alpha_{t_i}}) \right) 
\, \right].
\end{aligned}
$$
By this condition, we have 
\begin{equation}\label{2023_3_13_17_15}
\begin{aligned}
& \mathrm{tr} (u_{\alpha\alpha_{t_i}}(v) v_{\alpha_{t_i}}(v')) 
- \mathrm{tr} (v_{\alpha} (v)  u_{\alpha\alpha_{t_i}}(v')) \\
&=
\mathrm{tr} \left( 
(g^{t_i}_{\alpha_{t_i}})^{-1} v(g^{t_i}_{\alpha_{t_i}})
v'(A_{\alpha_{t_i}})
\right)  - \mathrm{tr} \left(
v(A_{\alpha_{t_i}})(g^{t_i}_{\alpha_{t_i}})^{-1}
 v'(g^{t_i}_{\alpha_{t_i}})
\right) \\
&\qquad +
\mathrm{tr} \left( \left(
 \operatorname{d}  \left( (g^{t_i}_{\alpha_{t_i}})^{-1} v(g^{t_i}_{\alpha_{t_i}}) \right)
+
\left[ \, A_{\alpha_{t_i}} , \, 
\left( (g^{t_i}_{\alpha_{t_i}})^{-1} v(g^{t_i}_{\alpha_{t_i}}) \right) \right] \right)
(g^{t_i}_{\alpha_{t_i}})^{-1} v'(g^{t_i}_{\alpha_{t_i}})
\right)
\end{aligned}
\end{equation}
So we may extend the 1-form 
$$
\mathrm{tr} (u_{\alpha\alpha_{t_i}}(v) v_{\alpha_{t_i}}(v')) 
- \mathrm{tr} (v_{\alpha} (v) u_{\alpha\alpha_{t_i}}(v'))
$$ 
from $U_{\alpha\alpha_{t_i}}$ to $U_{\alpha_{t_i}}$ by \eqref{2023_3_13_17_15}.
Since we have the vanishing of the lower terms \eqref{2023_3_13_17_13},
the extended 1-form defined on $U_{\alpha_{t_i}}$
is holomorphic.
We denote by $\omega_{\alpha_{t_i}}(v,v') $
the holomorphic 1-form defined on $U_{\alpha_{t_i}}$.

For $\alpha \in I_{\mathrm{cov}}\setminus (I^t_{\mathrm{cov}} \cup I^q_{\mathrm{cov}})$,
we set $\omega_{\alpha}(v,v') =0$.
By \eqref{2023_3_13_9_7} and \eqref{2023_3_13_17_15},
we have a meromorphic coboundary $\{ \omega_{\alpha}(v,v')\}_{\alpha}$ 
of
$$
\{ \mathrm{tr} (u_{\alpha\beta}(v) \circ v_{\beta}(v')) 
- \mathrm{tr} (v_{\alpha} (v) \circ u_{\alpha\beta}(v'))\}_{\alpha\beta}. 
$$
So we have
$$
\begin{aligned}
H^1(C,\Omega_C^1)
&\xrightarrow{ \ \cong  \ }\mathbb{C} \\
[ -\{ \mathrm{tr} (u_{\alpha\beta}(v) \circ v_{\beta}(v')) 
- \mathrm{tr} (v_{\alpha} (v) \circ u_{\alpha\beta}(v'))\}_{\alpha\beta} ]
&\longmapsto 
\sum_{x \in C }
- \mathrm{res}_{x} \left(\omega_{\alpha}(v,v') \right) .
\end{aligned}
$$
By taking the residues of the right hand sides of
\eqref{2023_3_13_9_7} and \eqref{2023_3_13_17_15},
we have that 
$$
\begin{aligned}
-\sum_{x \in C} \mathrm{res}_{x} \left(\omega_{\alpha}(v,v') \right)
&=
\sum_{j=1}^N
\mathrm{res}_{q_j} \left(
\frac{ \left( v(\zeta_j) \gamma_{\alpha_{q_j}} +
v(\delta_{\alpha_{q_j}} -\frac{\operatorname{d}\!z_j}{z_j-q_j})
 \right) v'(q_j) }{z_j -q_j} \right) \\
&\qquad 
- \sum_{j=1}^N\mathrm{res}_{q_j} \left( \frac{v(q_j)
 \left(
v'(\delta_{\alpha_{q_j}} -\frac{\operatorname{d}\!z_j}{z_j-q_j})
+v'(\zeta_i)\gamma_{\alpha_{q_j}} \right)}{z_j -q_j}  \right)  \\
&=
\sum_{j=1}^N \left( v( p_j) v'(q_j) 
- v(q_j) v'(p_j)  \right)
=
\left( \sum^N_{j=1} d p_j \wedge dq_j \right) (v,v').
\end{aligned}
$$
Here remark that 
$\gamma_{\alpha}$
is independent of the moduli space $M_{X}^0$ 
for any $\alpha$.
\end{proof}

\begin{remark}
We define $p_j^0 \in \mathbb{C}$ by
\begin{equation*}
\mathrm{res}_{q_j} \left( \beta \right) = p_j^0 dz_j.
\end{equation*}
The symplectic form $\sum^N_{j=1} d p^0_j \wedge d q_j$
is induced by the symplectic form on $\mathrm{Sym}^N (\boldsymbol{\Omega}(D))_0$.
If $g=0$, then 
$\left( \delta_{\alpha_{q_j}} 
-\frac{d z_j}{z_j-q_j} \right)|_{q_j} $
depends on only $q_j$.
So we have $\sum^N_{j=1} d p_j \wedge d q_j 
=\sum^N_{j=1} d p^0_j \wedge d q_j$.
Here $p_j$ is defined in \eqref{2023_7_12_16_43}.
\end{remark}

\begin{remark}
In general, $\sum^N_{j=1} d p_j \wedge d q_j 
\not= \sum^N_{j=1} d p^0_j \wedge d q_j$,
that is,
\begin{equation}\label{2023_3_18_22_44}
\sum_{j}
d ( p_j -p_j^0 )
\wedge d q_j
\end{equation}
does not vanish.
This is related to the determinant map 
$$
\begin{aligned}
M_{X}^0 &\longrightarrow 
M^{\mathrm{rk}=1}_X(\boldsymbol{\theta}_{\textrm{res}}) \\
(E,\nabla,\{l^{(i)}\}) &\longmapsto (\mathrm{det}(E), \mathrm{tr}(\nabla)).
\end{aligned}
$$
We consider an infinitesimal deformation of $\mathrm{det}(E) $.
We have that
$$
\mathrm{det} (B_{\alpha\alpha_{q_j}}) 
=  \mathrm{det}\left(
\begin{pmatrix}
1 & 0 \\
0 & ( (\omega_{\alpha}^{-1})^{-1} \circ  \omega_{\alpha_{q_j}}^{-1} )
\end{pmatrix}
\begin{pmatrix}
1 & \frac{\zeta_j}{z_j-q_j} \\
0&   \frac{1}{z_j-q_j}
\end{pmatrix} \right)
=  \frac{( (\omega_{\alpha}^{-1})^{-1} \circ  \omega_{\alpha_{q_j}}^{-1} )}{z_j-q_j}.
$$
Here $B_{\alpha\alpha_{q_j}}$ is calculated in \eqref{2023_3_12_23_33}.
Set 
$$
u^{\mathrm{det}}_{\alpha\alpha_{q_j}}(v)
:= \mathrm{det} (B_{\alpha\alpha_{q_j}})^{-1}  v(\mathrm{det} (B_{\alpha\alpha_{q_j}})) 
= \frac{v(q_j)}{z_j -q_j} . 
$$
Since we have $\mathrm{tr}(A_{\alpha_{q_j}})
=\delta_{\alpha_{q_j}} -\frac{\operatorname{d}\!z_j}{z_j-q_j}$ by \eqref{2023_3_6_18_19},
the 2-form \eqref{2023_3_18_22_44} comes from
$$
\left[\{ u^{\mathrm{det}}_{\alpha\beta}(v)u^{\mathrm{det}}_{\beta\gamma}(v')
 \} , -\{
 u^{\mathrm{det}}_{\alpha\beta}(v)v'(\mathrm{tr}(A_{\beta}))
- v(\mathrm{tr}(A_{\alpha})) u^{\mathrm{det}}_{\alpha\beta}(v') \}  \right] \in 
 \bH^2 (\mathcal{O}_C \rightarrow \Omega^1_C).
$$
This class gives rise to the $2$-form on $M_{X}^0$
which is just the pull-back of the natural symplectic form on 
$M^{\mathrm{rk}=1}_X(\boldsymbol{\theta}_{\textrm{res}})$ 
under the determinant map.
The determinant map is not degenerate in general.
So the class \eqref{2023_3_18_22_44} does not vanish in general.
\end{remark}

\section{Symplectic structure on the moduli space 
with fixed trace connection}\label{sect:sympl_FixedDet}

In this section, we consider 
the moduli spaces of rank 2 quasi-parabolic connections 
{\it with fixed trace connection}.
When the effective divisor $D$ is reduced, this moduli space is detailed in 
\cite{AL}, \cite{LS} (when $g=0$),
\cite{FL}, \cite{FLM} (when $g=1$),
and \cite{Matsu} (when $g\geq 1$).
The moduli spaces of rank 2 quasi-parabolic connections 
with fixed trace connection has 
a natural symplectic structure described as in Section \ref{subsect:symplecticGL2}.
The purpose of this section is to give coordinates on some generic part of the moduli space 
and to describe the natural symplectic structure by using the coordinates.
As in the case where the effective divisor $D$ is reduced 
(\cite{LS}, \cite{FL}, \cite{FLM}, \cite{Matsu}),  
we may define the map forgetting connections and 
the apparent map.
These maps are from a generic part of the moduli space
to projective spaces.
These maps will give our coordinates on the generic part of the moduli space.
First we describe these maps.

\subsection{Moduli space of quasi-parabolic bundles with fixed determinant}\label{SubSect:ModuliParaBunSL2}

To describe the map forgetting connections,
we recall the moduli space of quasi-parabolic bundles.
The moduli space of (quasi-)parabolic bundles 
was introduced in Mehta--Seshadri \cite{MS}.
Yokogawa generalized this notion to (quasi-)parabolic sheaves and studied their moduli \cite{Yoko1}.

Let $\nu$ be a positive integer. 
Set $I:=\{ 1,2,\ldots,\nu\}$.
Let $C$ be a compact Riemann surface of genus $g$, 
and $D = \sum_{i\in I} m_i [t_i]$ be an effective divisor on $C$.
We assume $3g-3+n>0$ where $n=\operatorname{length}(D)$.
Let
$z_i$ be a generator of the maximal ideal of $\mathcal{O}_{C,t_i}$.
We fix a line bundle $L_0$ with $\deg(L_0) =2g-1$.

\begin{definition}
We say $(E,  \{l^{(i)}\} )$ 
a {\rm rank $2$ quasi-parabolic bundle with determinant $L_0$ over $(C,D)$} if 
\begin{itemize}
\item[(i)] $E$ is a rank $2$ vector bundle of degree $2g-1$ on $C$ with 
$\det(E) \cong L_0$, and

\item[(ii)] $E|_{m_i[t_i]}  \supset l^{(i)} \supset 0$
is a filtration by free
$\mathcal{O}_{m_i[t_i]}$-modules such that 
$E|_{m_i[t_i]}/ l^{(i)}\cong \mathcal{O}_{m_i[t_i]}$ and
$l^{(i)}\cong \mathcal{O}_{m_i[t_i]}$
for any $i \in I$.
\end{itemize}
\end{definition}

We fix weights 
$\boldsymbol{w} 
= (w_1,\ldots , w_{\nu})$ such that $w_i \in [0,1]$ for any $i \in I$.
When $g=0$, we assume that $(w_i)_{i\in I}$ satisfies
\begin{equation}\label{2023_7_10_11_59(1)}
w_1=\cdots = w_{\nu}
\quad \text{and}\quad \frac{1}{\deg(D)} <w_i 
< \frac{1}{\deg(D)-2}.
\end{equation}
When $g\geq 1$, we assume that $(w_i)_{i\in I}$ satisfy 
\begin{equation}\label{2023_7_10_11_59(2)}
0 <w_i \ll 1.
\end{equation}

\begin{definition}
Let $(E,  \{l^{(i)}\} )$ be
a rank $2$ quasi-parabolic bundle with determinant $L_0$.
Let $L$ be a line subbundle of $E$.
We define the $\boldsymbol{w}$-stability index of $L$ to be the real number
$$
\mathrm{Stab}_{\boldsymbol{w} } (L) := 
\deg (E) - 2 \deg(L) 
+  \sum_{i\in I}   w_{i} \left(m_i - 2  \,
\mathrm{length}   
(l_{i}\cap L|_{m_i[t_i]})  \right).
$$
\end{definition}

\begin{definition}
A rank $2$ quasi-parabolic bundle $(E,  \{l^{(i)}\} )$ is {\rm $\boldsymbol{w}$-stable} if
for any subbundle $L \subset E$,
the inequality $\mathrm{Stab}_{\boldsymbol{w}}(L)>0$ holds.
\end{definition}

We say that a quasi-parabolic bundle
$(E, \{ l^{ (i) } \} )$ is decomposable if
there exists a decomposition $E =L_1 \oplus L_2$ such that
$l^{(i)} = l^{(i)}_{+}$ or $l^{(i)} = l^{(i)}_{-}$ for any $i\in I$,
where we set $l^{(i)}_{+} := l^{(i)}\cap (L_1|_{m_i[t_i]})$
and $l^{(i)}_{-} := l^{(i)}\cap (L_2|_{m_i[t_i]})$.
We say that $(E, \{ l^{ (i)} \} )$ is  undecomposable if
$(E, \{ l^{ (i) } \} )$ is not decomposable.
A free $\mathcal{O}_{m_i [t_i]}$-submodule $l^{ (i) }$ of $E|_{m_i [t_i]}$
induces a one dimensional subspace $l^{(i)}_{\mathrm{red}}$ of $E|_{t_i}$, that is the restriction
of $l^{(i)}$ to $t_i$ (without multiplicity).

\begin{lem}\label{2023_7_10_12_15}
Let $(E,  \{l^{(i)}\} )$ be
a rank $2$ quasi-parabolic bundle with determinant $L_0$.
If 
\begin{itemize}

\item[(i)]
$E$ is an extension of ${L_0}$ by $\mathcal{O}_C$
(when $g=0$, moreover we assume that $(E,  \{l^{(i)}\} )$ is undecomposable)

\item[(ii)]
$\dim_{\mathbb{C}} H^1(C,E) =0$

\item[(iii)]
$l_{\text{red}}^{(i)}  \not\in \mathcal{O}_{C}|_{t_i}
\subset \mathbb{P}(E)$ for any $i$,

\end{itemize}
then $(E,  \{l^{(i)}\} )$ is $\boldsymbol{w}$-stable.
\end{lem}

\begin{proof}
When $g=0$, we have this statement from \cite[Proposition 46]{KLS}
by the condition \eqref{2023_7_10_11_59(1)}.
When $g\geq 1$, we have that $E$ is stable, that is, $\deg(E) - 2\deg(L)$ is a positive integer 
for any line subbundle $L \subset E$.
This claim follows from the same argument as in 
\cite[Lemma 5.2]{Matsu}.
Since $0 <w_i \ll 1$ in \eqref{2023_7_10_11_59(2)}, 
we have that $\mathrm{Stab}_{\boldsymbol{w}}(L)>0$.
\end{proof}

Let $P_{(C,D)}^{\boldsymbol{w}}$ be 
a moduli space of $\boldsymbol{w}$-stable quasi-parabolic bundles
constructed in \cite{Yoko1}.
Let $P_{(C,D)}^{\boldsymbol{w}}(L_0)$ be the fiber of $L_0$ under the 
determinant map 
$$
P_{(C,D)}^{\boldsymbol{w}} \longrightarrow \mathrm{Pic}_C^{2g-1}; 
\quad (E,  \{l^{(i)}\} ) \longmapsto \det(E).
$$
We set 
$$
P_{(C,D)}(L_0)_0 := 
\left\{ (E,\{ l^{(i)}\} )  
\ \middle| \ 
\begin{array}{l}
\text{$(E,\{ l^{(i)}\} )$ is rank $2$ quasi-parabolic bundle over $(C,D)$ such that}\\
\text{(i) $\det (E) \cong L_0$,\quad 
(ii) $E$ is an extension of ${L_0}$ by $\mathcal{O}_C$,} \\
\text{(iii) $\dim_{\mathbb{C}} H^1(C,E) =0$,\quad 
(iv) $l_{\text{red}}^{(i)}  \not\in \mathcal{O}_{C}|_{t_i}
\subset \mathbb{P}(E)$ for any $i$,} \\
\text{(v) $(E,\{ l^{(i)}\} )$ is undecomposable (when $g=0$)}
\end{array}
\right\}.
$$
By Lemma \ref{2023_7_10_12_15}, we have an inclusion 
$$
P_{(C,D)}(L_0)_0 \subset P_{(C,D)}^{\boldsymbol{w}}(L_0).
$$
For $(E,\{ l^{(i)}\} ) \in P_{(C,D)}(L_0)_0$, we have an extension 
\begin{equation}\label{2023_7_10_13_56}
0 \longrightarrow 
\mathcal{O}_C \longrightarrow
E \longrightarrow
L_0 \longrightarrow 0.
\end{equation}
Since $\dim_{\mathbb{C}} H^1(C,E) =0$, we have that $\dim_{\mathbb{C}} H^0(C,E) =1$.
So the injection $\mathcal{O}_C \xrightarrow{\subset} E$ 
in \eqref{2023_7_10_13_56} is unique up to a constant.

\begin{definition}\label{2023_7_10_12_29}
Let $(E,\{ l^{(i)}\} ) \in P_{(C,D)}(L_0)_0$.
We take an affine open covering $\{ U_{\alpha} \}_{\alpha}$ of $C$, 
i.e. $C = \bigcup_{\alpha} U_{\alpha}$. 
Let $\{\varphi_{\alpha}^{\mathrm{Ext}}\}_{\alpha} $ be trivializations
$\varphi_{\alpha}^{\mathrm{Ext}} \colon \mathcal{O}_{U_{\alpha}}^{\oplus 2} 
\rightarrow E|_{U_{\alpha}}$ of 
the underlying vector bundle $E$
such that 
\begin{itemize}

\item[(i)]
the composition
$$
\begin{aligned}
 \mathcal{O}_{U_{\alpha}} &\longrightarrow \mathcal{O}_{U_{\alpha}}^{\oplus 2} 
\xrightarrow{ \varphi_{\alpha}^{\mathrm{Ext}} } E|_{U_{\alpha}}  \\
f &\longmapsto (f,0)
\end{aligned}
$$
is just the inclusion $\mathcal{O}_C \subset E$ of
the extension \eqref{2023_7_10_13_56} for any $\alpha$, and

\item[(ii)] the image of the composition
$$
\begin{aligned}
 \mathcal{O}_{U_{\alpha}} &\longrightarrow \mathcal{O}_{U_{\alpha}}^{\oplus 2} 
\xrightarrow{ \varphi_{\alpha}^{\mathrm{Ext}} } E|_{U_{\alpha}} 
\longrightarrow E|_{m_i[t_i]} \\
f &\longmapsto (0,f)
\end{aligned}
$$
generates the submodule $l^{(i)} \subset E|_{m_i[t_i]}$ 
for each $i$ and $\alpha$ where $t_i \in U_{\alpha}$.
\end{itemize}
\end{definition}

Notice that 
the claim that we may take $\varphi_{\alpha}^{\mathrm{Ext}}$ 
which satisfies the condition (ii) of Definition \ref{2023_7_10_12_29}
follows from the condition that 
$l_{\text{red}}^{(i)}  \not\in \mathcal{O}_{C}|_{t_i}
\subset \mathbb{P}(E)$ for any $i$.

Now we define a map
\begin{equation}\label{2023_7_10_12_44}
P_{(C,D)}(L_0)_0
\longrightarrow \mathbb{P} H^1(C, L_0^{-1}(-D))
\end{equation}
as follows.
Let $\{\varphi_{\alpha}^{\mathrm{Ext}}\}_{\alpha} $ be
the trivializations in Definition \ref{2023_7_10_12_29}.
We have the transition matrices
$$
B_{\alpha\beta} := (\varphi_{\alpha}^{\mathrm{Ext}} |_{U_{\alpha\beta}})^{-1} \circ
\varphi_{\beta}^{\mathrm{Ext}} |_{U_{\alpha\beta}}
\colon \mathcal{O}_{U_{\alpha\beta}}^{\oplus 2} 
\longrightarrow \mathcal{O}_{U_{\alpha\beta}}^{\oplus 2} .
$$
We represent $B_{\alpha\beta}$ as a matrix:
\begin{equation}\label{2023_7_10_20_51(1)}
B_{\alpha\beta} = 
\begin{pmatrix}
1 & b^{12}_{\alpha\beta} \\
0 & b^{22}_{\alpha\beta}
\end{pmatrix}.
\end{equation}
Remark that $\{ b^{22}_{\alpha\beta}\}_{\alpha\beta}$ is a 
multiplicative cocycle which defines the fixed line bundle $L_0$.
We take 
a meromorphic coboundary 
\begin{equation}\label{2023_7_10_12_37}
\{ b^{22}_{\alpha}   \}_{\alpha}  \quad 
\left(
\text{where  } b^{22}_{\alpha\beta} = \frac{b^{22}_{\alpha}}{b^{22}_{\beta}} \right)
\end{equation}
of 
the multiplicative cocycle 
$\{ b^{22}_{\alpha\beta}\}_{\alpha\beta}$.
By using the coboundary $\{ b^{22}_{\alpha}   \}_{\alpha} $,
we set 
\begin{equation}\label{2023_7_10_13_12(1)}
b^{\text{Bun}}_{\alpha\beta} :=  b^{12}_{\alpha\beta}b^{22}_{\beta}.
\end{equation} 
By comparing the $(1,2)$-entries of the both hand sides of
$B_{\alpha\beta} B_{\beta\gamma} = B_{\alpha \gamma}$,
we have 
$$
b^{12}_{\beta\gamma}
+b^{22}_{\beta\gamma} b^{12}_{\alpha\beta}
=b^{12}_{\alpha\gamma}.
$$
So we have the cocycle condition
$b^{\text{Bun}}_{\beta\gamma}
-b^{\text{Bun}}_{\alpha\gamma}+b^{\text{Bun}}_{\alpha\beta}=0$.
We have a class
$[\{ b^{\text{Bun}}_{\alpha\beta}\} ] \in H^1(C, L_0^{-1} (-D))$.
Then we have a map \eqref{2023_7_10_12_44}:
$$
(E,\{ l^{(i)}\} )
\longmapsto
\overline{[\{ b^{\text{Bun}}_{\alpha\beta}\} ]}.
$$

\subsection{Moduli space of quasi-parabolic connections 
with fixed trace connection}\label{SubSect:ModuliSL2}

Now we recall the moduli space of quasi-parabolic connections.
We fix an irregular curve with residues 
$X=(C,D,\{ z_i \} , \{\boldsymbol{\theta}_{i} \}, \boldsymbol{\theta}_{\mathrm{res}})$
defined in Definition \ref{2023_7_14_23_01}.
Moreover we assume that 
\begin{equation}\label{2023_8_24_8_47}
\sum_{i \in I}\theta_{i,-1}^{-} \neq 0.
\end{equation}
Let $({L_0},\nabla_{L_0}
\colon {L_0} \rightarrow {L_0}\otimes \Omega^1_C(D))$ 
be a rank $1$ connection on $C$ with degree $2g-1$
such that the polar part of 
$\nabla_{L_0}$
at $t_i$ is
$\mathrm{tr} (\omega_i(X ))$.

\begin{definition}\label{2023_7_10_14_08}
We say $(E,\nabla , \lambda, \{l^{(i)}\} )$ 
a {\rm rank $2$ quasi-parabolic $\lambda$-connection over $X$ with fixed
trace connection $({L_0},\nabla_{L_0})$} if 
\begin{itemize}
\item[(i)] $E$ is a rank $2$ vector bundle on $C$ with $\det(E) \cong L_0$,

\item[(ii)] 
$\lambda \in \mathbb{C}$ and
$\nabla \colon E \rightarrow E \otimes \Omega^1_{C}(D)$ is a $\lambda$-connection
that is, 
$\nabla (fs ) = \lambda s \otimes df + f \nabla (s)$
for any $f \in \mathcal{O}_C$ and $s \in E$, and

\item[(iii)] $\nabla (s_1) \wedge s_2 + s_1 \wedge \nabla(s_2) = \lambda \nabla_L(s_1\wedge s_2) $
for $s_1,s_2 \in E$,

\item[(iv)] $E|_{m_i[t_i]}  \supset l^{(i)} \supset 0$
is a filtration by free
$\mathcal{O}_{m_i[t_i]}$-modules such that, for any $i \in I$, 

\begin{itemize} 

\item 
$E|_{m_i[t_i]}/ l^{(i)}\cong \mathcal{O}_{m_i[t_i]}$ and
$l^{(i)}\cong \mathcal{O}_{m_i[t_i]}$,

\item
$\nabla|_{m_i[t_i]} (l^{(i)}) \subset l^{(i)} \otimes \Omega^1_{C}(D)$, and 

\item the image of
$(E|_{m_i[t_i]}/ l^{(i)}) \oplus  l^{(i)} $
under
$\mathrm{Gr}_i (\nabla) - \lambda \cdot \omega_i(X ) $
is contained in
$$ 
\left( (E|_{m_i[t_i]}/ l^{(i)}) \oplus  l^{(i)}   \right)
\otimes \Omega^1_C.$$

\end{itemize}

\end{itemize}

Here $\mathrm{Gr}_i (\nabla)$ is the induced morphism 
$$
\mathrm{Gr}_i (\nabla)
\colon (E|_{m_i[t_i]}/ l^{(i)}) \oplus  l^{(i)} 
\longrightarrow \left( (E|_{m_i[t_i]}/ l^{(i)}) \oplus  l^{(i)}   \right)
\otimes \Omega^1_C(D).
$$
\end{definition}

Notice that, if $\lambda=0$, then $\nabla$ is an $\mathcal{O}_C$-morphism,
which is called a Higgs field.
So $(E,\nabla , \lambda, \{l^{(i)}\} )$ is called a
(trace free) quasi-parabolic Higgs bundle when $\lambda=0$.
We consider only rank $2$ quasi-parabolic $\lambda$-connections
$(E,\nabla , \lambda, \{l^{(i)}\} )$ over $X$ with
$({L_0},\nabla_{L_0})$
such that the underlying quasi-parabolic bundle $(E,\{l^{(i)}\} )$
is in the moduli space $P_{(C,D)}(L_0)_0$.

We define
the moduli spaces
$\widetilde{M}_X({L_0},\nabla_{L_0})_0$
and 
$M_X(L_0,  \nabla_{L_0})_0$
as follows:
$$
\widetilde{M}_X({L_0},\nabla_{L_0})_0
= \left. \left\{ 
\begin{array}{l}
(E,\nabla, \lambda, \{ l^{(i)}\} ) \\
\text{quasi-parabolic $\lambda$-connection} \\
\text{over $X$ with trace $({L_0},\nabla_{L_0})$} 
\end{array}
\ \middle| \ 
\begin{array}{l}
(E, \{ l^{(i)}\} )  \in P_{(C,D)}(L_0)_0
\end{array}
\right\} \right/ \cong
$$
and
\begin{equation}\label{2026_8_14_16_44}
M_X(L_0,  \nabla_{L_0})_0
= \left. \left\{ 
\begin{array}{l}
(E,\nabla, \lambda, \{ l^{(i)}\} ) \\
\text{quasi-parabolic $\lambda$-connection} \\
\text{over $X$ with trace $({L_0},\nabla_{L_0})$} 
\end{array}
\ \middle| \ 
\begin{array}{l}
(E, \{ l^{(i)}\} )  \in P_{(C,D)}(L_0)_0  \\
\text{and }\lambda \neq 0 
\end{array}
\right\} \right/ \cong.
\end{equation}

\subsection{Maps from the moduli space}\label{SubSect:MapsSL2}

Now we describe two maps: 
the forgetful map $\pi_{\mathrm{Bun}}$ forgetting connections and 
the apparent map $\pi_{\mathrm{App}}$.
First we consider the composition 
$$
\widetilde{M}_X({L_0},\nabla_{L_0})_0
\longrightarrow P_{(C,D)}(L_0)_0 
\longrightarrow \mathbb{P} H^1( C, L_0^{-1}(-D)).
$$
Here the first map is the forgetful map, 
and the second map is \eqref{2023_7_10_12_44}.
We denote by 
$$
\pi_{\mathrm{Bun}} \colon
\widetilde{M}_X({L_0},\nabla_{L_0})_0
\longrightarrow \mathbb{P} H^1( C, L_0^{-1}(-D)).
$$
the composition.

Second we define a map
\begin{equation}\label{2023_7_10_13_02}
\pi_{\mathrm{App}} \colon 
\widetilde{M}_X({L_0},\nabla_{L_0})_0
\longrightarrow
 \mathbb{P} H^0(C, {L_0}\otimes \Omega^1_C(D)) 
\end{equation}
as follows.
Let $(E,\nabla, \lambda, \{ l^{(i)}\} )$ 
be a point on
$\widetilde{M}_X({L_0},\nabla_{L_0})_0$.
Let $\{\varphi_{\alpha}^{\mathrm{Ext}}\}_{\alpha} $ be
the trivializations in Definition \ref{2023_7_10_12_29}.
Let $A_{\alpha}$ be the connection matrix of the $\lambda$-connection 
$\nabla$ with respect to $\varphi_{\alpha}^{\mathrm{Ext}}$, that is, 
$$
\lambda \operatorname{d} + A_{\alpha} := (\varphi_{\alpha}^{\mathrm{Ext}} )^{-1} \circ
\nabla \circ
\varphi_{\alpha}^{\mathrm{Ext}} 
\colon \mathcal{O}_{U_{\alpha}}^{\oplus 2} 
\longrightarrow (\Omega^1_{U_{\alpha}}(D))^{\oplus 2} .
$$
We denote the matrix $A_{\alpha}$ as follows:
\begin{equation}\label{2023_7_10_20_51(2)}
A_{\alpha} = 
\begin{pmatrix}
a_{\alpha}^{11} & a^{12}_{\alpha} \\
a^{21}_{\alpha} & a^{22}_{\alpha}
\end{pmatrix}.
\end{equation}
By the condition (ii) in Definition \ref{2023_7_10_12_29}
and the condition (iv) in Definition \ref{2023_7_10_14_08},
the polar part of the connection matrix $A_{\alpha}$ at $t_i$ is 
a lower triangular matrix, that is, the Laurent expansion of $A_{\alpha}$
at $t_i$ is as follows:
\begin{equation}\label{2023_7_10_20_59}
A_{\alpha} =
\begin{pmatrix}
\lambda\nu_{i}^{-} & 0 \\
* &\lambda \nu_{i}^{+}
\end{pmatrix} \frac{1}{z_i^{m_i}} +[\text{ holo. part }].
\end{equation}
Here $\nu_{i}^{-}, \nu_{i}^{+} \in \Omega^1_C(D)|_{m_i[t_i]}$
are defined so that
$$
\lambda \cdot \omega_i(X )
=\begin{pmatrix}
\lambda\nu_{i}^{-} & 0 \\
0 &\lambda \nu_{i}^{+}
\end{pmatrix} .
$$
By using the coboundary $\{ b^{22}_{\alpha}   \}_{\alpha} $ in \eqref{2023_7_10_12_37},
we set
\begin{equation}\label{2023_7_10_13_12(2)}
a^{\text{App}}_{\alpha} :=  a^{21}_{\alpha}(b^{22}_{\alpha} )^{-1}.
\end{equation}
We consider the equation
\begin{equation}\label{2025_10_18_13_22}
\lambda \cdot \operatorname{d}\! B_{\alpha\beta} + A_{\alpha} B_{\alpha\beta} = 
B_{\alpha\beta}A_{\beta}.
\end{equation}
By comparing the $(2,1)$-entries of the both hand side,
we have 
$a^{21}_{\alpha}  = b^{22}_{\alpha\beta} a^{21}_{\beta} $. 
By \eqref{2023_7_10_12_37}, 
we have a cocycle condition $a^{\text{App}}_{\alpha}=a^{\text{App}}_{\beta}$.
So we have a class
$[\{ a^{\text{App}}_{\alpha}\} ] \in H^0(C, {L_0} \otimes \Omega^1_C(D))$.
Then we have a map \eqref{2023_7_10_13_02}:
$$
(E,\nabla, \lambda, \{ l^{(i)}\} )
\longmapsto
\overline{[\{ a^{\text{App}}_{\alpha}\} ]}.
$$
Finally, we have a map
\begin{equation}\label{2023_9_8_10_17}
(\pi_{\mathrm{App}}, \pi_{\mathrm{Bun}}) \colon 
\widetilde{M}_X({L_0},\nabla_{L_0})_0
\longrightarrow
 \mathbb{P} H^0(C, {L_0}\otimes \Omega^1_C(D)) \times 
\mathbb{P} H^1( C, L_0^{-1}(-D)).
\end{equation}

We consider the natural pairing 
\begin{equation}\label{2023_7_10_20_48}
H^0(C, {L_0}\otimes \Omega^1_C(D)) \times H^1(C, L_0^{-1} (-D)) \longrightarrow 
H^1(C,\Omega^1_{C}) \cong \mathbb{C} .
\end{equation}

\begin{lem}\label{2023_7_10_21_01}
Let $(E,\nabla, \lambda, \{ l^{(i)}\} )
\in \widetilde{M}_X({L_0},\nabla_{L_0})_0$.
Let $a^{\text{App}}_{\alpha}$ and $b^{\text{Bun}}_{\alpha\beta}$
be the cocycles in \eqref{2023_7_10_13_12(1)} and
in \eqref{2023_7_10_13_12(2)},
respectively. 
Then we have
$$
[\{  b^{\mathrm{Bun}}_{\alpha\beta} \cdot a^{\mathrm{App}}_{\beta}  \} ] 
= \lambda \cdot \sum_{i \in I}\theta_{i,-1}^-.
$$
Here the left hand side is the pairing \eqref{2023_7_10_20_48}.
\end{lem}

\begin{proof}
Let $B_{\alpha\beta}$ be the transition function in \eqref{2023_7_10_20_51(1)}.
By comparing the $(1,1)$-entries of the both hand sides of \eqref{2025_10_18_13_22},
we have 
$$
a_{\alpha}^{11} - a_{\beta}^{11} = 
b^{\mathrm{Bun}}_{\alpha\beta} \cdot a^{\mathrm{App}}_{\beta}.
$$
By \eqref{2023_7_10_20_59} and the isomorphism $H^1(C,\Omega^1_{C}) \cong \mathbb{C}$,
we have $[\{  b^{\mathrm{Bun}}_{\alpha\beta} \cdot a^{\mathrm{App}}_{\beta}\}  \} ] 
= \lambda \cdot \sum_{i }\theta_{i,-1}^-$.
\end{proof}

Set
$$
N_0 :=  \dim_{\mathbb{C}} \mathbb{P} H^0(C, {L_0} \otimes \Omega^1_C(D)) =  3g+n-3.
$$
Let us introduce the
homogeneous coordinates
$\boldsymbol{a} = (a_0 :\cdots : a_{N_0})$
on 
$\mathbb{P} H^0(C, {L_0} \otimes \Omega^1_C(D))\cong \mathbb{P}_{\boldsymbol{a}}^{N_0}$
and the dual coordinates 
$\boldsymbol{b} = (b_0 :\cdots : b_{N_0})$
on 
$$
\mathbb{P} H^1(C, L_0^{-1}(-D))\cong 
\mathbb{P} H^0(C, {L_0}\otimes \Omega^1_C(D))^{\vee} \cong \mathbb{P}_{\boldsymbol{b}}^{N_0}.
$$
Let $\Sigma\subset \mathbb{P}_{\boldsymbol{a}}^{N_0} \times \mathbb{P}_{\boldsymbol{b}}^{N_0}$
be the incidence variety 
whose defining equation is given by $\sum_ja_jb_j =0$.
By Lemma \ref{2023_7_10_21_01},
we have that 
$$
\widetilde{M}_X({L_0},\nabla_{L_0})_0
\setminus 
M_X({L_0},\nabla_{L_0})_0
\xrightarrow{ \ (\pi_{\mathrm{App}}, \pi_{\mathrm{Bun}}) \  }
\Sigma.
$$

\begin{remark}\label{2023_7_14_9_20}
Loray--Saito (for $g=0$) and Matsumoto (for $g\geq 1$) discussed the birationality of
the map \eqref{2023_9_8_10_17}.
They showed the birationality of 
the map \eqref{2023_9_8_10_17} when $D$ is a reduced effective divisor
(\cite[Theorem 4.3]{LS} for $g=0$ 
and 
\cite[Theorem 5.6]{Matsu}
for $g\ge 1$).
In these cases, quasi-parabolic connections have only simple poles.
But we may apply the arguments in \cite[Theorem 4.3]{LS}
and in 
\cite[Theorem 5.6]{Matsu}
to our cases where 
quasi-parabolic connections admit generic unramified irregular singular points.
So we can reconstruct $(E,\nabla, \lambda, \{ l^{(i)}\} )
\in M_X({L_0},\nabla_{L_0})_0$ 
from an element of 
$$
\mathbb{P} H^1( C, L_0^{-1}(-D))_0 \times 
\mathbb{P} H^0( C, L_0\otimes \Omega^1_C(D)) \setminus \Sigma.
$$
Here we set 
$$
\mathbb{P} H^1( C, L_0^{-1}(-D))_0 := \left\{ b\in  \mathbb{P} H^1( C, L_0^{-1}(-D))
\ \middle| \  
\begin{array}{l}
\text{{\rm The extension $E$ corresponding to $b$}} \\
\text{{\rm satisfies $\dim_{\mathbb{C}} H^1(C,E)=0$}}
\end{array}
\right\}.
$$
Then we have an isomorphism
\begin{equation}\label{2026_8_14_16_43}
M_X({L_0},\nabla_{L_0})_0 \cong
\mathbb{P} H^1( C, L_0^{-1}(-D))_0 \times 
\mathbb{P} H^0( C, L_0\otimes \Omega^1_C(D)) \setminus \Sigma.
\end{equation}
\end{remark}

\begin{remark}
By using the stability condition for quasi-parabolic connections 
due to Inaba--Saito \cite[Definition 2.2]{IS},
we have a moduli space $M_X({L_0},\nabla_{L_0})$ 
of stable quasi-parabolic connections over $X$ 
with fixed trace connection $(L_0,\nabla_{L_0})$.
Here remark that the generic condition in Definition \ref{2023_7_14_23_01}, 
our quasi-parabolic connections over $X$ are irreducible. 
So they are automatically stable.

The moduli space $M_X({L_0},\nabla_{L_0})_0$ 
in \eqref{2026_8_14_16_44}
is a Zariski open subset of 
the full moduli space $M_X({L_0},\nabla_{L_0})$.
This shows that one can give a geometric description of an open subset of the moduli space by restricting to the locus where the underlying bundle is stable.
However, there do exist stable connections on unstable bundles.
For example, Beilinson--Drinfeld's opers \cite{BD} are such connections.
Parabolic version of $\mathrm{SL}_r$-opers is in \cite{BDP}.
So, connections on stable bundle cannot be expected to give a geometric description of the full moduli space.
To describe the full moduli space,
Loray--Saito \cite{LS} and 
Fassarella--Loray \cite{FL} tried the idea 
to vary the stability condition of the underlying parabolic bundles 
and to patch the geometric descriptions of the locus where the underlying bundle is stable
(see \cite[Proposition 3.4 and Section 6]{LS} and \cite[Section 3.6.3]{FL}).
\end{remark}

\subsection{Symplectic structure and explicit description}\label{SubSect:SyplecticSL2}

Now we recall the natural symplectic structure on 
$M_X({L_0},\nabla_{L_0})_0$.
We define a complex $\cF_0^{\bullet}$ for $(E, \frac{1}{\lambda} \nabla ,  \{l^{(i)}\} )$ by 
\begin{equation*}
\begin{aligned}
&\cF_0^0 := \left\{  s \in \cE nd (E)  
\ \middle| \  
\mathrm{tr}(s)=0,\,
 s |_{m_i t_i} (l_j^{(i)}) 
\subset l_j^{(i)} \text{ for $j=0,1$ and any $i$} \right\} \\
&\cF^1_0 :=  \left\{  s \in  \cE nd (E)\otimes \Omega^1_{C}(D)  
\ \middle| \  \mathrm{tr}(s)=0,\,  
s|_{m_i t_i} (l_j^{(i)}) 
\subset l_{j+1}^{(i)} \otimes \Omega^1_C 
\text{ for $j=0,1$ and any $i$}  \right\} \\
&\nabla_{\cF^{\bullet}} \colon \cF_0^0 \lra \cF_0^1; 
\quad \nabla_{\cF_0^{\bullet}} (s) = 
(\frac{1}{\lambda}\nabla)\circ s - s \circ (\frac{1}{\lambda}\nabla).
\end{aligned}
\end{equation*}
Here we put $l_0^{(i)} := E_{m_i[t_i]} $, 
$l_1^{(i)} : =l^{(i)}$, and 
$l_2^{(i)} : =0$.
We define the following morphism 
\begin{equation}\label{2023_3_12_20_42}
\begin{aligned}
\bH^1( \cF_0^{\bullet}) \otimes \bH^1(\cF_0^{\bullet}) 
\lra \bH^2(\mathcal{O}_C \xrightarrow{d}\Omega_{C}^{1}) \cong \mathbb{C}
\end{aligned}
\end{equation}
as in \eqref{2020.11.7.15.48}.
This pairing gives the symplectic form
on $M_X({L_0},\nabla_{L_0})_0$.
We denote by $\omega_0$ the symplectic form.

The maps $\pi_{\mathrm{App}}$ and $\pi_{\mathrm{Bun}}$ give 
coordinates on $M_X({L_0},\nabla_{L_0})_0$ (see Remark \ref{2023_7_14_9_20}). 
Now we describe the symplectic structure \eqref{2023_3_12_20_42}
by using the coordinates on $M_X({L_0},\nabla_{L_0})_0$. 
We define a 1-form $\eta$ on 
$\mathbb{P}_{\boldsymbol{a}}^{N_0} \times \mathbb{P}_{\boldsymbol{b}}^{N_0}$
as follows:
$$
\eta := \left( - \sum_i \theta_{i,-1}^- \right) \cdot
\frac{a_0 \, d b_0 + a_1 \, d  b_1 
+ \cdots + a_{N_0} \, d b_{N_0}
}{a_0b_0 + a_1b_1 + \cdots + a_{N_0}b_{N_0}}.
$$

\begin{thm}\label{2023_8_22_16_22}
Assume that $\sum_{i \in I}\theta_{i,-1}^{-} \neq 0$.
Let $\omega_{\boldsymbol{a}, \boldsymbol{b}}$ be
the $2$-form on $\mathbb{P}^{N_0}_{\boldsymbol{a}} \times \mathbb{P}^{N_0}_{\boldsymbol{b}}$
defined by $\omega_{\boldsymbol{a}, \boldsymbol{b}}= d \eta$.
The pull-back of $\omega_{\boldsymbol{a}, \boldsymbol{b}}$ under 
the map
$$
M_X({L_0},\nabla_{L_0})_0
\xrightarrow{\  (\pi_{\mathrm{App}}, \pi_{\mathrm{Bun}}) \  }
\mathbb{P}_{\boldsymbol{a}}^{N_0} \times \mathbb{P}_{\boldsymbol{b}}^{N_0}
$$
coincides with the symplectic form $\omega_0$
on $M_X({L_0},\nabla_{L_0})_0$.
\end{thm}

\begin{proof}
Let $v , v' \in 
T_{(E, \frac{1}{\lambda} \nabla ,  \{l^{(i)}\} )}
M_X({L_0},\nabla_{L_0})_0$.
We have the isomorphism 
$$
T_{(E, \frac{1}{\lambda} \nabla ,  \{l^{(i)}\} )}
M_X({L_0},\nabla_{L_0})_0
\xrightarrow{ \ \cong \ } \bH^1( \cF_0^{\bullet}).
$$
Let
$u_{\alpha\beta} (v)$ and $v_{\alpha} (v)$
be cocycles 
such that the class
$[\{u_{\alpha\beta} (v)\}_{\alpha\beta},\{v_{\alpha} (v)\}_{\alpha}]$
is the image of $v$ under the isomorphism.
We calculate $u_{\alpha\beta} (v)$ and $v_{\alpha} (v)$
by using the trivialization $\{ \varphi_{\alpha}^{\mathrm{Ext}} \}_{\alpha}$
as follows:
\begin{equation}\label{2023_3_14_11_49}
\begin{aligned}
u_{\alpha\beta} (v)
&= 
\varphi_{\beta}^{\mathrm{Ext}}|_{U_{\alpha\beta}} \circ
\left( B_{\alpha\beta}^{-1} v(B_{\alpha\beta}) \right) \circ
(\varphi_{\beta}^{\mathrm{Ext}} |_{U_{\alpha\beta}})^{-1} \\
&=
\varphi_{\beta}^{\mathrm{Ext}}  |_{U_{\alpha\beta}} \circ 
\begin{pmatrix}
0 & v (b^{12}_{\alpha\beta}) \\
0 & 0
\end{pmatrix}
\circ
(\varphi_{\beta}^{\mathrm{Ext}}|_{U_{\alpha\beta}} )^{-1} \\
&=
\varphi_{\beta}^{\mathrm{Ext}} |_{U_{\alpha\beta}} \circ 
\begin{pmatrix}
0 & \frac{v (b^{\text{Bun}}_{\alpha\beta}  )}{
b^{22}_{\beta}} \\
0 & 0
\end{pmatrix}
\circ
(\varphi_{\beta}^{\mathrm{Ext}}|_{U_{\alpha\beta}} )^{-1}
\end{aligned}
\end{equation}
and
\begin{equation}\label{2023_3_14_11_50}
\begin{aligned}
v_{\alpha} (v)
&= 
\varphi_{\alpha}^{\mathrm{Ext}} \circ
 v \left( \frac{1}{\lambda } A_{\alpha} \right) \circ
(\varphi_{\alpha}^{\mathrm{Ext}} )^{-1} \\
&=
\varphi_{\alpha}^{\mathrm{Ext}} \circ  
\begin{pmatrix}
v(a_{\alpha}^{11}/\lambda) & v(a^{12}_{\alpha} /\lambda) \\
v(a^{21}_{\alpha}/\lambda )&  v(a^{22}_{\alpha}/\lambda)
\end{pmatrix}
\circ
(\varphi_{\alpha}^{\mathrm{Ext}} )^{-1}\\
&=
\varphi_{\alpha}^{\mathrm{Ext}} \circ  
\begin{pmatrix}
v(a_{\alpha}^{11}/\lambda) & v(a^{12}_{\alpha} /\lambda) \\
v(a^{\text{App}}_{\alpha}/\lambda ) b^{22}_{\alpha} &  v(a^{22}_{\alpha}/\lambda)
\end{pmatrix}
\circ
(\varphi_{\alpha}^{\mathrm{Ext}} )^{-1}.
\end{aligned}
\end{equation}
Here $\{ b^{22}_{\alpha}   \}_{\alpha} $ is the coboundary in \eqref{2023_7_10_12_37}.
Since we fix the determinant bundle $L_0$, 
we may assume that
the coboundary $\{ b^{22}_{\alpha}   \}_{\alpha} $
is independent of the moduli space 
$M_X({L_0},\nabla_{L_0})_0$.

Now we calculate the class
\begin{equation}\label{2023.3.18.22.55}
[ (\{  \mathrm{tr}( u_{\alpha\beta} (v)  u_{\beta\gamma} (v') ) \}, -\{
\mathrm{tr} \left(u_{\alpha\beta} (v) 
v_{\beta} (v') \right)
- \mathrm{tr} \left(
v_{\alpha} (v)
u_{\alpha\beta} (v') 
 \right) \} )]
\end{equation}
in $\bH^2(\mathcal{O}_C \xrightarrow{d}\Omega_{C}^{1}) \cong \mathbb{C}$.
First we calculate $u_{\alpha\beta} (v)  u_{\beta\gamma} (v')$ as follows:
$$
\begin{aligned}
&u_{\alpha\beta} (v)  u_{\beta\gamma} (v')  \\
&=
\varphi_{\beta}^{\mathrm{Ext}}  |_{U_{\alpha\beta}} \circ 
\begin{pmatrix}
0 & \frac{v (b^{\text{Bun}}_{\alpha\beta}  )}{b^{22}_{\beta}} \\
0 & 0
\end{pmatrix}
\circ
(\varphi_{\beta}^{\mathrm{Ext}}|_{U_{\alpha\beta}} )^{-1}\circ 
\varphi_{\gamma}^{\mathrm{Ext}}|_{U_{\alpha\beta}} \circ 
\begin{pmatrix}
0 & \frac{v (b^{\text{Bun}}_{\beta\gamma}  )}{b^{22}_{\gamma}} \\
0 & 0
\end{pmatrix}
\circ
(\varphi_{\gamma}^{\mathrm{Ext}} |_{U_{\alpha\beta}})^{-1} \\
&=
\varphi_{\beta}^{\mathrm{Ext}}  |_{U_{\alpha\beta}} \circ 
\begin{pmatrix}
0 & \frac{v (b^{\text{Bun}}_{\alpha\beta}  )}{b^{22}_{\beta}} \\
0 & 0
\end{pmatrix}
B_{\beta\gamma} 
\begin{pmatrix}
0 & \frac{v (b^{\text{Bun}}_{\beta\gamma}  )}{b^{22}_{\gamma}} \\
0 & 0
\end{pmatrix}
\circ
(\varphi_{\gamma}^{\mathrm{Ext}} |_{U_{\alpha\beta}})^{-1} \\
&=
\varphi_{\beta}^{\mathrm{Ext}} |_{U_{\alpha\beta}} \circ 
\begin{pmatrix}
0 & 0 \\
0 & 0
\end{pmatrix}
\circ
(\varphi_{\gamma}^{\mathrm{Ext}}|_{U_{\alpha\beta}} )^{-1} =0.
\end{aligned}
$$
So we may take a representative of the class \eqref{2023.3.18.22.55} 
so that 
$$
[  -\{
\mathrm{tr} \left(u_{\alpha\beta} (v) 
v_{\beta} (v') \right)
- \mathrm{tr} \left(
v_{\alpha} (v)
u_{\alpha\beta} (v') 
 \right) \} ]
$$
is in $H^1(C,\Omega^1_C)$.
By using equalities 
\eqref{2023_3_14_11_49}, \eqref{2023_3_14_11_50},
and $b^{22}_{\alpha} = b^{22}_{\alpha\beta}b^{22}_{\beta}$,
we have the following equality
\begin{equation}\label{2023_3_14_11_51}
\begin{aligned}
&\mathrm{tr} \left(u_{\alpha\beta} (v) 
v_{\beta} (v') \right)-
\mathrm{tr} \left(
v_{\alpha} (v) u_{\alpha\beta} (v') \right) \\
&= \mathrm{tr} \left(
\begin{pmatrix}
0 & \frac{v (b^{\text{Bun}}_{\alpha\beta}  )}{b^{22}_{\beta}} \\
0 & 0
\end{pmatrix}
\begin{pmatrix}
v'(a_{\beta}^{11}/\lambda) & v'(a^{12}_{\beta} /\lambda) \\
v'(a^{\text{App}}_{\beta}/\lambda ) b^{22}_{\beta} &  v'(a^{22}_{\beta}/\lambda)
\end{pmatrix} \right) \\
&\qquad - \mathrm{tr} \left(
\begin{pmatrix}
v(a_{\alpha}^{11}/\lambda) & v(a^{12}_{\alpha} /\lambda) \\
v(a^{\text{App}}_{\alpha}/\lambda ) b^{22}_{\alpha} &  v(a^{22}_{\alpha}/\lambda)
\end{pmatrix}
B_{\alpha\beta}
\begin{pmatrix}
0 & \frac{v' (b^{\text{Bun}}_{\alpha\beta}  )}{b^{22}_{\beta}} \\
0 & 0
\end{pmatrix}B_{\alpha\beta}^{-1} \right) \\
&=
v (b^{\text{Bun}}_{\alpha\beta}  ) v'\left( \frac{a^{\text{App}}_{\beta} }{\lambda} \right)
-v\left( \frac{a^{\text{App}}_{\alpha} }{\lambda} \right)  v' (b^{\text{Bun}}_{\alpha\beta}  ) .
\end{aligned}
\end{equation}
Here remark that $a^{\text{App}}_{\alpha}=a^{\text{App}}_{\beta}$.
We take bases 
$$
a^{\text{App}(0)}, a^{\text{App}(1)}, \ldots ,a^{\text{App}(N_0)} \in 
H^0(C,  L_0\otimes \Omega^1_C(D))
$$
of $H^0(C,  L_0\otimes \Omega^1_C(D))$ and
$$
[\{ b_{\alpha\beta}^{\text{App}(0)}\}],
[\{ b_{\alpha\beta}^{\text{App}(1)}\}], \ldots ,
[\{ b_{\alpha\beta}^{\text{App}(N_0)}\}]  
$$
of $H^1(C, L_0^{-1}(-D))$
so that these bases give 
the homogeneous coordinates
$(a_0 :\cdots : a_{N_0})$
on 
$ \mathbb{P}_{\boldsymbol{a}}^{N_0}$
and
$(b_0 :\cdots : b_{N_0})$
on $\mathbb{P}_{\boldsymbol{b}}^{N_0}$.
We may assume that 
these bases
are independent of the moduli space 
$M_X({L_0},\nabla_{L_0})_0$.
We set 
$$
a^{\text{App}}_{\alpha} = a_0 a^{\text{App}(0)}|_{U_{\alpha}}
+ a_1 a^{\text{App}(1)}|_{U_{\alpha}}+\cdots 
+a_{N_0} a^{\text{App}(N_0)}|_{U_{\alpha}}
$$
and
$$
b^{\text{App}}_{\alpha\beta} = b_0 b^{\text{App}(0)}_{\alpha\beta}
+ b_1 b^{\text{App}(1)}_{\alpha\beta}+\cdots 
+b_{N_0} b^{\text{App}(N_0)}_{\alpha\beta}.
$$
By \eqref{2023_3_14_11_51}, we have that 
$$
\begin{aligned}
\mathrm{tr} \left(u_{\alpha\beta} (v) 
v_{\beta} (v') \right)-
\mathrm{tr} \left(
v_{\alpha} (v) u_{\alpha\beta} (v') \right) 
&= 
v \left(\sum_{k=0}^{N_0}b_k b^{\text{App}(k)}_{\alpha\beta}  \right)
 v'\left( \frac{\sum_{k=0}^{N_0} a_k 
 a^{\text{App}(k)}|_{U_{\alpha \beta}} }{\lambda} \right) \\
&\qquad -v\left( \frac{\sum_{k=0}^{N_0} a_k 
a^{\text{App}(k)}|_{U_{\alpha \beta}} }{\lambda} \right)  
v'\left(\sum_{k=0}^{N_0}b_k b^{\text{App}(k)}_{\alpha\beta}  \right) \\
&=
\left( \sum_{k=0}^{N_0} v \left(b_k   \right) b^{\text{App}(k)}_{\alpha\beta}\right)
\left( \sum_{k=0}^{N_0}  v'\left(  \frac{a_k}{\lambda}   \right)
a^{\text{App}(k)}|_{U_{\alpha\beta}} \right) \\
&\qquad 
-\left( \sum_{k=0}^{N_0}  v\left(  \frac{a_k}{\lambda}   \right)
a^{\text{App}(k)}|_{U_{\alpha \beta}} \right) 
\left( \sum_{k=0}^{N_0} v' \left(b_k   \right) b^{\text{App}(k)}_{\alpha\beta}\right).
\end{aligned}
$$
Since $(b_0 :\cdots : b_{N_0})$
is dual of $(a_0 :\cdots : a_{N_0})$
with respect to the natural pairing 
$$
H^0(C, L_0\otimes \Omega^1_C(D)) \times H^1(C, L_0^{-1} (-D)) \longrightarrow 
H^1(C,\Omega^1_{C}) \cong \mathbb{C},
$$
we have that
$$
\begin{aligned}
H^1(C,\Omega^1_C) &\xrightarrow{\ \cong \ }
\mathbb{C} \\
[  -\{
\mathrm{tr} \left(u_{\alpha\beta} (v) 
v_{\beta} (v') \right)
- \mathrm{tr} \left(
v_{\alpha} (v)
u_{\alpha\beta} (v') 
 \right) \} ]  
&\longmapsto
\sum_{k=0}^{N_0}   v \left(b_k   \right) v'\left(  \frac{a_k}{\lambda}   \right) 
-\sum_{k=0}^{N_0}   v' \left(b_k   \right) v\left(  \frac{a_k}{\lambda}   \right) .
 \end{aligned}
$$
On the other hand, we have that
$$
\lambda =
\frac{ \langle [\{ a^{\text{App}}_{\alpha} \} ,[ \{ b^{\text{Bun}}_{\alpha\beta}\}]   \rangle}{
-\sum_i \theta_{i,-1}^-}
=
\frac{ a_0b_0 + a_1b_1 + \cdots + a_{N_0}b_{N_0}}{
-\sum_i \theta_{i,-1}^-}.
$$
Then we have
$$
\begin{aligned}
 d\eta(v,v')
 =\sum_{k=0}^{N_0}   v \left(b_k   \right) v'\left(  \frac{a_k}{\lambda}   \right) 
-\sum_{k=0}^{N_0}   v' \left(b_k   \right) v\left(  \frac{a_k}{\lambda}   \right).
\end{aligned}
$$
This means the statement.
\end{proof}

\section{Companion normal forms for an elliptic curve with two poles}\label{Sect:CompForElliptic}

In Section \ref{sect:CompanionNF}, we introduced the companion normal form 
of a rank 2 meromorphic connection with some assumption. 
The purpose of the present section is to detail the case of an elliptic curve with two simple poles, or with an unramified irregular singularity of order $2$. 
The latter case arises by confluence from the first one, up to some modification in the arguments. 
We will give explicit description of 
the companion normal form for an elliptic curve in these cases.
Moreover, we will calculate the canonical coordinates introduced 
in Section \ref{subsect:Canonical_Coor}. 
First we start from the construction of the companion normal form
$(\mathcal{O}_C\oplus (\Omega_C^1(D))^{-1} , \nabla_0)$.
Next we will construct a rank 2 meromorphic connection $(E,\nabla)$
by transforming the companion normal form.

Let $C$ be the elliptic curve constructed by 
gluing affine cubic curves
$$
U_{0} := (y_1^2-x_1(x_1-1)(x_1-\lambda) =0)
\quad\text{and} \quad 
U_{\infty} := ( y_2^2-x_2(1-x_2)(1-\lambda x_2)  =0)
$$
with the relations
$x_1 = x_2^{-1}$ and
$y_1 = y_2 x_2^{-2}$.
We fix some $t\in\mathbb{C}$ and set $D=t_1+t_2$
where $t_1=(t,s)$ and $t_2=(t,-s)$, so that $D$ is the positive part of $\mathrm{div}(x-t)$.
Let $q_1,q_2,q_3$ be points on $C$:
$$
q_j\colon (x_1,y_1) = (u_j,v_j)
$$
for each $j=1,2,3$.
Now we assume that 
$u_j \not\in \{ 0,1,\lambda, \infty, t \} $ for any $j$.

We take a trivialization of the line bundle $(\Omega^1_{C}(D))^{-1}$
over $C$ as follows:
\begin{equation}\label{2023_6_30_21_56(1)}
\mathcal{O}_{U_0} \xrightarrow{\ \sim \ } (\Omega^1_{C}(D))^{-1} |_{U_0} ; \quad 
1 \longmapsto  \left( \frac{\operatorname{d}\!x_1}{(x_1-t)y_1} \right)^{-1}
\end{equation}
and
\begin{equation}\label{2023_6_30_21_56(2)}
\mathcal{O}_{U_\infty} \xrightarrow{\ \sim \ } (\Omega^1_{C}(D))^{-1} |_{U_\infty} ;\quad 
1 \longmapsto  \left( \frac{\operatorname{d}\!x_2}{(1-tx_2)y_2} \right)^{-1}.
\end{equation}
Then the corresponding transition function $f_{0\infty}$ is as follows:
\begin{equation}\label{2023_7_1_12_35}
\begin{aligned}
f_{\infty0} \colon \mathcal{O}_{U_0} |_{U_0 \cap U_\infty} &\xrightarrow{\ \sim \ }
 \mathcal{O}_{U_\infty} |_{U_0 \cap U_\infty} \\
1 &\longmapsto  -\frac{1}{x_2}.
\end{aligned} 
\end{equation}

\subsection{Definition of a connection $\nabla_0$ on 
$\mathcal{O}_C\oplus (\Omega^1_{C}(D))^{-1}$}
For $\zeta_1,\zeta_2,\zeta_3 \in \mathbb{C}$, 
we define $1$-forms $\omega_{12}$, $\omega_{21}$, and $\omega_{22}$ as follows:
\begin{equation}\label{2023_7_1_10_5}
\begin{aligned}
&\omega_{12}=\sum_{j=1}^3\frac{\zeta_j}{2} \cdot
\frac{y_1+v_{j}}{x_1-u_{j}}
\cdot
\frac{\operatorname{d}\!x_1}{y_1} 
+\left( \frac{A_1 + A_2 y_1   }{x_1-t} 
+ A_3 +A_4 x_1    \right)
\frac{\operatorname{d}\!x_1}{y_1}\\
&\omega_{21}:= \frac{1}{x_1-t}\frac{\operatorname{d}\!x_1}{y_1} \\
&\omega_{22}:=\sum_{j=1}^3\frac{1}{2} \cdot
\frac{y_1+v_{j}}{x_1-u_{j}} \cdot \frac{\operatorname{d}\!x_1}{y_1}
+ \left(\frac{B_1 + B_2 y_1 }{x_1-t} +B_3 \right) \frac{\operatorname{d}\!x_1}{y_1} .
\end{aligned}
\end{equation}
Here $A_1,\ldots,A_4 \in \mathbb{C}$ and $B_1,\ldots,B_3\in \mathbb{C}$ 
are parameters.
Notice that 
$\omega_{12} \otimes \omega_{21}$
is a global section of $(\Omega_{C}^1)^{\otimes 2}(2D+B)$
and 
$\omega_{22}$ is a global section of $\Omega_{C}^1(D+B+\infty)$.

\subsubsection{Fixing the polar parts in the logarithmic case}
We start by analyzing the case where $t\notin \{ 0,1,\lambda, \infty \}$. 
In this case, we have $s\neq 0$, so $t_1\neq t_2$. 
We fix complex numbers $\theta_1^{\pm}, \theta_2^{\pm}$
such that $\sum_{i=1}^2 (\theta^+_i+\theta^-_i) = -1$,
which is called Fuchs' relation.
Now we assume that the eigenvalues of the matrix 
$$
  \mathrm{res}_{t_1} \begin{pmatrix}
0 & \omega_{12}\\
\omega_{21} & \omega_{22}
\end{pmatrix}
$$
are given by $\theta_1^+, \theta_1^-$ and the eigenvalues of the matrix 
$$
  \mathrm{res}_{t_2} \begin{pmatrix}
0 & \omega_{12}\\
\omega_{21} & \omega_{22}
\end{pmatrix}
$$
are given by $\theta_2^+, \theta_2^-$. 
(To be coherent with Definition~\ref{2023_7_14_23_01}, we should write $\theta_{1,-1}$ and $\theta_{2,-1}$ for elements of the Cartan subalgebra, and $\theta_{1,-1}^{\pm}$ and $\theta_{2,-1}^{\pm}$ for their eigenvalues; however, we drop the subscript $-1$ for ease of notation, because there are only poles of order $1$, so no confusion is possible.) 
Specifically, these conditions read as  
\begin{equation}\label{2023_6_30_22_02(1)}
\mathrm{res}_{(t,s)}\omega_{12} \cdot
\mathrm{res}_{(t,s)}\omega_{21} =  \theta^+_1 \cdot \theta^-_1, \qquad 
\mathrm{res}_{(t,-s)}\omega_{12} \cdot 
\mathrm{res}_{(t,-s)}\omega_{21} =  \theta^+_2 \cdot \theta^-_2,
\end{equation}
and
\begin{equation}\label{2023_6_30_22_02(2)}
\mathrm{res}_{(t,s)}\omega_{22} =\theta^+_1 + \theta^-_1, \qquad   
\mathrm{res}_{(t,-s)}\omega_{22} = \theta^+_2 + \theta^-_2.
\end{equation}
Notice that $\mathrm{res}_{(u_j,v_j)}\omega_{22} = 1$ for each $j$.
By the residue theorem, 
$\mathrm{res}_{\infty}\omega_{22} = -2$.
By the assumption \eqref{2023_6_30_22_02(1)}
and \eqref{2023_6_30_22_02(2)},
we may determine the parameters $A_1,A_2, B_1$, and $B_2$.

\begin{lem}\label{2023_7_21_17_07}
Let complex numbers $\theta_1^{\pm}, \theta_2^{\pm}$ satisfying Fuchs' relation be given. 
Then, there exist unique values of the parameters $A_1,A_2, B_1$, and $B_2$ such that \eqref{2023_6_30_22_02(1)} and \eqref{2023_6_30_22_02(2)} are fulfilled. 
Moreover, these parameter values are independent of $u_1,u_2,u_3$, $\zeta_1,\zeta_2$, and $\zeta_3$.
So the polar parts of $\omega_{12},\omega_{21}$, and $\omega_{22}$
at $t_i$ are
independent of $u_1,u_2,u_3$, $\zeta_1,\zeta_2$, and $\zeta_3$.
\end{lem}

\begin{proof}
By the equalities \eqref{2023_6_30_22_02(1)}, we have 
$$
\frac{A_1 + A_2s}{s} \cdot \frac{1}{s} = \theta^+_1 \cdot \theta^-_1
\quad \text{and} \quad 
\frac{A_1 - A_2s}{-s} \cdot \frac{1}{-s} = \theta^+_2 \cdot \theta^-_2.
$$
By the equalities in \eqref{2023_6_30_22_02(2)}, we have
$$
\frac{B_1 + B_2s}{s}  = \theta^+_1 + \theta^-_1
\quad \text{and} \quad 
\frac{B_1 - B_2s}{-s}  = \theta^+_2 + \theta^-_2.
$$
By these equalities, 
$A_1,A_2, B_1$, and $B_2$ are determined,
and $A_1,A_2, B_1$, and $B_2$
are
independent of $u_1,u_2,u_3$, $\zeta_1,\zeta_2$, and $\zeta_3$.
It is clear that 
the polar parts of $\omega_{12},\omega_{21}$, and $\omega_{22}$
at $t_i$ are
independent of $u_1,u_2,u_3$, $\zeta_1,\zeta_2$, and $\zeta_3$.
\end{proof}

\subsubsection{Fixing the polar part in the irregular case}
We now study the situation $t\in \{ 0, 1 , \lambda , \infty \}$. 
For sake of concreteness, we let $t=0$, the other cases being similar. 
Then, $s=0$ and $t_1 = t_2$, so the divisor $D$ is reduced of length $2$. 
A local holomorphic coordinate of the elliptic curve $C$ in a neighbourhood of $t_1$ is given by $y_1$. 

We fix $\theta_{-2}^{\pm}, \theta_{-1}^+\in\mathbb{C}$ so that $\theta_{-2}^+ \neq \theta_{-2}^-$ and set $\theta_{-1}^- = - 1 - \theta_{-1}^+$. 
(To be coherent with Definition~\ref{2023_7_14_23_01}, we should write $\theta_{1,-2}$ and $\theta_{1,-1}$ for elements of the Cartan subalgebra, and $\theta_{1,-2}^{\pm}$ and $\theta_{1,-1}^{\pm}$ for their eigenvalues; however, we omit the subscript $1$ for ease of notation, because there is only one singular point, so no confusion is possible.) 
\begin{lem}
 Fix $\theta_{-2}^{\pm}, \theta_{-1}^{\pm}$ as above. 
 Then, there exist unique values $A_1, A_2, B_1, B_2\in\mathbb{C}$ such that the eigenvalues of 
$$
  \mathrm{res} \begin{pmatrix}
0 & \omega_{12}\\
\omega_{21} & \omega_{22}
\end{pmatrix}
$$
admit Laurent expansions of the form 
$$
  \left( \theta_{-2}^{\pm} \frac{1}{y_1^2} + \theta_{-1}^{\pm} \frac{1}{y_1} + O(1) \right)  \otimes \operatorname{d}\! y_1. 
$$
Moreover, the values of the solutions are independent of $u_i, \zeta_i$. 
\end{lem}

\begin{proof}
By the inverse function theorem, there exists an analytic open subset $U\subset \mathbb{C}$ and a holomorphic function $h\colon U \to \mathbb{C}$ satisfying $h(0)=0$ such that $C$ is given by the explicit equation $x_1 = h (y_1^2 )$. 
It is obvious that this function $h$ is independent of the choice of $u_i, \zeta_i$, and it is easy to see that $h'(0) = \frac{1}{\lambda}\neq 0$. 
From the defining equation of $C$ we get 
$$
  \frac{\operatorname{d}\!x_1}{y_1} = \frac{2 \operatorname{d}\!y_1}{3 x_1^2 -2(1+\lambda ) x_1 + \lambda}, 
$$
so $\frac{\operatorname{d}\!x_1}{y_1}$ is a holomorphic $1$-form around $t_1$. 
Moreover, 
$$
  \frac{\operatorname{d}\!x_1}{x_1 y_1} =  \frac{\operatorname{d}\!y_1}{y_1^2} g(y_1^2)
$$
for some holomorphic function $g\colon U \to \mathbb{C}$ satisfying $g(0)=2$. 
The polar parts of the coefficients can be separated as 
\begin{align*}
  \omega_{12} & = (A_1 + A_2 y_1 ) \frac{\operatorname{d}\!x_1}{x_1y_1} + O(1) = 2 (A_1 + A_2 y_1 ) \frac{\operatorname{d}\!y_1}{y_1^2} + O(1) \\
  \omega_{21} & = 2 \frac{\operatorname{d}\!y_1}{y_1^2} + O(1) \\
  \omega_{22} & = (B_1 + B_2 y_1 ) \frac{\operatorname{d}\!x_1}{x_1y_1} + O(1) = 2 (B_1 + B_2 y_1 ) \frac{\operatorname{d}\!y_1}{y_1^2} + O(1).
\end{align*}
Now, the sum of the eigenvalues must be  
$$
  ( \theta_{-2}^+ + \theta_{-2}^- ) \frac{1}{y_1^2} + ( \theta_{-1}^+ + \theta_{-1}^-) \frac{1}{y_1}. 
$$
These conditions determine 
$$
  B_1 = \frac 12 ( \theta_{-2}^+ + \theta_{-2}^- ), \qquad B_2 = \frac 12 (\theta_{-1}^+ + \theta_{-1}^-) = - \frac12 . 
$$
Moreover, we have 
$$
  - \omega_{12} \omega_{21} = -4 (A_1 + A_2 y_1 ) \frac{\left( \operatorname{d}\!y_1 \right)^{\otimes 2}}{y_1^4} + O\left( \frac{1}{y_1^2} \right). 
$$
On the other hand, the product of the eigenvalues must have the expansion (up to a global factor $\left( \operatorname{d}\!y_1 \right)^{\otimes 2}$)
$$
  \theta_{-2}^+ \theta_{-2}^- \frac{1}{y_1^4} +  ( \theta_{-2}^+ \theta_{-1}^- + \theta_{-2}^- \theta_{-1}^+) \frac{1}{y_1^3}. 
$$
These condition then determine 
$$
  A_1 = - \frac 14 \theta_{-2}^+ \theta_{-2}^-, \qquad A_3 = - \frac 14 (\theta_{-2}^+ \theta_{-1}^- + \theta_{-2}^- \theta_{-1}^+).
$$
This finishes the proof. 
\end{proof}

\subsubsection{Construction of the connection}
We define
$$
\begin{aligned}
&\beta \colon (\Omega^1_{C}(D))^{-1} \longrightarrow 
\Omega^1_{C}(D+B) 
&& \text{($\mathcal{O}_{C}$-morphism)} \\
&\delta \colon (\Omega^1_{C}(D))^{-1} \longrightarrow 
(\Omega^1_{C}(D))^{-1} 
\otimes \Omega^1_{C}(D+B) 
&& \text{(connection)} \\
&\gamma \colon \mathcal{O}_C
\longrightarrow 
(\Omega^1_{C}(D))^{-1} 
\otimes \Omega^1_{C}(D)
&&\text{($\mathcal{O}_{C}$-morphism)}
\end{aligned}
$$
by using 
the trivializations \eqref{2023_6_30_21_56(1)}
and \eqref{2023_6_30_21_56(2)} of $(\Omega^1_{C}(D))^{-1}$
as follows:
$$
\beta=
\begin{cases}
\omega_{12} 
\colon \mathcal{O}_{U_0} \rightarrow \mathcal{O}_{U_0} \otimes \Omega^1_{C}(D+B)|_{U_0} \\
\mathrm{id} \circ \omega_{12}  \circ f^{-1}_{\infty0}
\colon \mathcal{O}_{U_\infty} \rightarrow
 \mathcal{O}_{U_\infty} \otimes \Omega^1_{C}(D+B)|_{U_\infty},
\end{cases}
$$
$$
\delta=
\begin{cases}
\operatorname{d}+\omega_{22} 
\colon \mathcal{O}_{U_0} \rightarrow 
 \mathcal{O}_{U_0} \otimes \Omega^1_{C}(D+B)|_{U_0} \\
\operatorname{d} + f_{\infty0} \circ \omega_{22} \circ f^{-1}_{\infty0}
+ f_{\infty0} \circ \operatorname{d}\! f^{-1}_{\infty0} 
\colon \mathcal{O}_{U_\infty} \rightarrow 
 \mathcal{O}_{U_\infty} \otimes\Omega^1_{C}(D+B)|_{U_\infty},
\end{cases}
$$
$$
\gamma :=
\begin{cases}
\omega_{21}
\colon \mathcal{O}_{U_0} \rightarrow 
 \mathcal{O}_{U_0} \otimes \Omega^1_{C}(D+B)|_{U_0} \\ 
f_{\infty 0}\circ \omega_{21} \circ \mathrm{id}
\colon \mathcal{O}_{U_\infty} \rightarrow
 \mathcal{O}_{U_\infty} \otimes \Omega^1_{C}(D+B)|_{U_\infty}.
\end{cases}
$$
Here $f_{\infty0}$ is the transition function of $(\Omega^1_{C}(D))^{-1}$
described in \eqref{2023_7_1_12_35}.
Notice that 
$$f_{\infty0} \circ \omega_{22} \circ f^{-1}_{\infty0}
+ f_{\infty0} \circ \operatorname{d}\! f^{-1}_{\infty0}  = 
\omega_{22} + \frac{\operatorname{d}\!x_2}{x_2},$$
which is holomorphic at $\infty \in C$,
since we have $\mathrm{res}_{\infty}\omega_{22} = -2$.
We define a connection as follows:
\begin{equation}\label{2023_7_14_17_29}
\nabla_0 :=
\operatorname{d} +\begin{pmatrix}0&\beta\\ \gamma& \delta \end{pmatrix}
 \colon \mathcal{O}_C \oplus (\Omega^1_{C}(D))^{-1}
\longrightarrow \left(  \mathcal{O}_C \oplus (\Omega^1_{C}(D))^{-1} \right) 
\otimes \Omega^1_{C}(D+B),
\end{equation}
which is the companion normal form.
Remark that 
$$
\mathrm{res}_{q_j} (\nabla_0)
= \begin{pmatrix}
0 & \zeta_j \\
0 & 1
\end{pmatrix}
$$
for $j=1,2,3$.

\begin{lem}\label{2023_7_21_17_07(2)}
The fact that $\nabla_0$ has apparent singular points at $q_1,q_2,q_3$ 
imposes 3 linear conditions
on $A_3,A_4,B_3$ in terms of spectral data, and $((u_j,v_j),\zeta_j)$'s; 
we can uniquely determine $A_3,A_4,B_3$
from these conditions if, and only if, we have
\begin{equation}\label{2023_7_3_22_49}
\det\begin{pmatrix}1&u_1&\zeta_1\\ 1&u_2&\zeta_2 \\ 1&u_3&\zeta_3\end{pmatrix}\not=0.
\end{equation}
\end{lem}

\begin{proof}
It is just Lemma \ref{lem:independence} specified to the present elliptic case with 2 poles.
We set 
\begin{equation}\label{2023_7_4_13_18}
C_{j}=
\sum_{j'\in\{1,2,3\}\setminus \{j\}}\frac{\zeta_{j'}-\zeta_j}{2} \cdot
\frac{v_j+v_{j'}}{u_j-u_{j'}}
+ \frac{A_1 + A_2 v_j-\zeta_j (B_1 + B_2 v_j) -\zeta_j^2   }{u_j-t} .
\end{equation}
We denote by $((a_j)_j,(b_j)_j,(c_j)_j)$ the $3\times 3$-matrix
$$
((a_j)_j,(b_j)_j,(c_j)_j)=
\begin{pmatrix}
a_1 & b_1 & c_1 \\
a_2 & b_2 & c_2\\
a_3 & b_3 & c_3
\end{pmatrix}.
$$
The condition where $q_1	,q_2,q_3$ are apparent singularities means that 
\begin{equation}\label{2023_7_1_12_13}
((1)_j,(u_j)_j,(-\zeta_j)_j)
\begin{pmatrix}
A_3 \\
A_4 \\
B_3
\end{pmatrix}
=-
\begin{pmatrix}
C_1 \\
C_2 \\
C_3
\end{pmatrix}.
\end{equation}
By Cramer's rule, 
the parameters
$A_3,A_4,B_3$ of the family of connections $\nabla_0$ are uniquely determined 
$$
\begin{aligned}
&A_3= - \frac{\det((C_j)_j,(u_j)_j,(\zeta_j)_j )}{
\det(((1)_j,(u_j)_j,(\zeta_j)_j))} &&
&A_4= - \frac{\det((1)_j,(C_j)_j,(\zeta_j)_j)}{\det((1)_j,(u_j)_j,(\zeta_j)_j)}\\
&B_3=  \frac{\det((1)_j,(u_j)_j,(C_j)_j)}{\det((1)_j,(u_j)_j,(\zeta_j)_j)},
&&
\end{aligned}
$$
if and only if \eqref{2023_7_3_22_49}.
\end{proof}

\begin{lem}\label{lem:unstabledet} We have:
$$\det\begin{pmatrix}1&u_1&\zeta_1\\ 1&u_2&\zeta_2 \\ 
1&u_3&\zeta_3\end{pmatrix}=0$$
if, and only if, $E$ is not stable.
\end{lem}

\begin{proof} The vanishing of the determinant gives that 
$\zeta_j=\sigma(q_j)$ for a global section 
$\sigma\in H^0(C,\Omega_C^1(D))$. In other words, 
the quasi-parabolic structure on $E_0$ given over each $q_j$
by the eigenvectors corresponding to eigenvalue $1$ lie on a subbundle 
$(\Omega^1_C(D))^{-1}\subset E_0$.
After elementary transformations at each $q_j$, 
we get $L\subset E$ with $\deg(L)=1$ (in fact $L=\det(E)$).
\end{proof}

\subsection{Definition of a rank 2 vector bundle $E$}
We set
$$
\tilde{U}_0 := U_0 \setminus \{ q_1,q_2,q_3\}
\quad 
\text{and}
\quad
\tilde{U}_\infty := U_\infty \setminus \{ q_1,q_2,q_3\}.
$$
We take an analytic open subsets $\tilde U_{q_j}$ ($j=1,2,3$) of $C$ such that 
$q_j \in \tilde U_{q_j}$ and $\tilde U_{q_j}$ are small enough. 
In particular, $(u_j,-v_j) \not\in \tilde U_{q_j}$.
On $\tilde U_{q_j}$, the apparent singular point $q_j$
is defined by $x_1-u_j=0$.
We have an open covering $(\tilde U_k)_{k \in \{ 0,1,q_1,q_2,q_3\} }$ of $C$.
We define transition functions $B_{k_1k_2}$ ($k_1,k_2 \in \{ 0,1,q_1,q_2,q_3\}$)
as follows:
$$
B_{0 q_j} := 
\begin{pmatrix}
1 & \frac{\zeta_j}{x_1 - u_j} \\
0 & \frac{1}{x_1 - u_j}
\end{pmatrix}\colon 
\mathcal{O}^{\oplus 2}_{\tilde U_{q_j}}|_{\tilde U_{0} \cap \tilde U_{q_j}}
\xrightarrow{ \ \sim \ }
\mathcal{O}^{\oplus 2}_{\tilde U_{0}}|_{\tilde U_{0} \cap \tilde U_{q_j}};
$$
$$
B_{0 \infty} := 
\begin{pmatrix}
1 & 0 \\
0 & -x_2
\end{pmatrix}\colon 
\mathcal{O}^{\oplus 2}_{\tilde U_{\infty}}|_{\tilde U_{0} \cap \tilde U_{\infty}}
\xrightarrow{ \ \sim \ }
\mathcal{O}^{\oplus 2}_{\tilde U_{0}}|_{\tilde U_{0} \cap \tilde U_{\infty}}.
$$
Then we have a vector bundle
$$
E = \left((\tilde U_k)_{k \in \{ 0,1,q_1,q_2,q_3\} } , 
\ (B_{k_1k_2} )_{k_1,k_2 \in \{ 0,1,q_1,q_2,q_3\}} \right),
$$
where $E$ is trivial on each $\tilde U_k$
and the transition function from $\tilde U_{k_2}$ to $\tilde U_{k_1}$
is $B_{k_1k_2}$.

\subsection{Definition of a connection $\nabla$ on $E$}

We define matrices $A_0,A_{q_j}, A_\infty$ as follows:
$$
\begin{aligned}
&A_0 := \begin{pmatrix}
0 & \omega_{12}\\
\omega_{21} & \omega_{22}
\end{pmatrix},
&&A_{\infty}:=
\begin{pmatrix}
0 & -x_2 \omega_{12} \\
-\frac{\omega_{21}}{x_2} & \omega_{22} + \frac{\operatorname{d}\! x_2}{x_2}
\end{pmatrix}, \\
&A_{q_j}:=
\begin{pmatrix}
\omega^{(j)}_{11} & \frac{\omega^{(j)}_{12}}{x_1-u_j}\\
(x_1-u_j)\omega_{21} & \omega^{(j)}_{22}
\end{pmatrix}.&&
\end{aligned}
$$
The 1-form $\omega_{12}$, $\omega_{21}$, and $\omega_{22}$ 
are defined in \eqref{2023_7_1_10_5}.
The 1-form $\omega_{12}^{(j)}$, $\omega_{21}^{(j)}$, and $\omega_{22}^{(j)}$ 
are defined as follows:
$$
\begin{aligned}
\omega^{(j)}_{11}&= -\frac{\zeta_j}{x_1-t}   \cdot \frac{\operatorname{d}\!x_1}{y_1} ,  \\
\omega^{(j)}_{12}&=
\sum_{j'\in\{1,2,3\}\setminus \{j\}}\frac{\zeta_{j'}-\zeta_j}{2} \cdot
\frac{y_1+v_{j'}}{x_1-u_{j'}} \cdot \frac{\operatorname{d}\!x_1}{y_1} \\
& \qquad + \left( \frac{A_1 + A_2 y_1-\zeta_j (B_1 + B_2 y_1)  -\zeta_j^2 }{x_1-t} 
+A_3 + A_4 x_1  - \zeta_j B_3 
\right) \frac{\operatorname{d}\!x_1}{y_1} , \\
\omega^{(j)}_{22}&=\frac{1}{2} \cdot \frac{-y_1+v_{j}}{x_1-u_{j}} \cdot \frac{\operatorname{d}\!x_1}{y_1} +
\sum_{j'\in\{1,2,3\}\setminus \{j\}}\frac{1}{2} \cdot
\frac{y_1+v_{j'}}{x_1-u_{j'}} \cdot \frac{\operatorname{d}\!x_1}{y_1}  \\
&\qquad + \left( \frac{B_1 + B_2 y_1 }{x_1-t}  + B_3
+ \frac{\zeta_j}{x_1-t}   \right) \frac{\operatorname{d}\!x_1}{y_1}  .
\end{aligned}
$$

\begin{prop}
\begin{itemize}
\item The $(1,2)$-entry of $A_{q_j}$ is a section of 
$\Omega_C^1(D)|_{\tilde U_{q_j}}$ for each $j=1,2,3$.

\item We define a local connection on each $\tilde U_k$ $(k \in  \{ 0,1,q_1,q_2,q_3\})$ by
$$
\begin{cases}
\operatorname{d}+A_0 \colon \mathcal{O}_{\tilde U_{0}}^{\oplus 2}
\longrightarrow \mathcal{O}_{\tilde U_{0}}^{\oplus 2} \otimes \Omega_C^1(D)|_{\tilde U_{0}} 
& \text{on $\tilde U_{0}$} \\
\operatorname{d}+A_{q_j} \colon \mathcal{O}_{\tilde U_{q_j}}^{\oplus 2}
\longrightarrow \mathcal{O}_{\tilde U_{q_j}}^{\oplus 2} \otimes \Omega_C^1(D)|_{\tilde U_{q_j}} 
& \text{on $\tilde U_{q_j}$}\\
\operatorname{d}+A_\infty \colon \mathcal{O}_{\tilde U_{\infty}}^{\oplus 2}
\longrightarrow \mathcal{O}_{\tilde U_{\infty}}^{\oplus 2} \otimes \Omega_C^1(D)|_{\tilde U_{\infty}} 
& \text{on $\tilde U_{\infty}$}.
\end{cases}
$$
Then we can glue these local connections.
So we have a global connection 
$\nabla \colon E \rightarrow E \otimes \Omega^1_C(D)$ on $E$.
\end{itemize}
\end{prop}

\begin{proof}
Since $A_3,A_4,B_3$ are determined so that these parameters satisfy the condition 
\eqref{2023_7_1_12_13},
we have
$$
\omega_{12}^{(j)} |_{q_j} = 
\left(C_j + A_3 + A_4 u_j  - \zeta_j B_3 \right) \frac{\operatorname{d}\!x_1|_{q_j}}{v_j} =0.
$$
Here, $C_j$ is in \eqref{2023_7_4_13_18}.
So $\frac{\omega^{(j)}_{12}}{x_1-u_j}$ has no pole at $q_j$ for each $j=1,2,3$.
Since we have 
$$
B_{k_1k_2}^{-1}A_{k_1}B_{k_1k_2} 
+B_{k_1k_2}^{-1} \operatorname{d}\! B_{k_1k_2} = A_{k_2} 
$$
for each $k_1,k_2 \in \{ 0,\infty ,q_1,q_2,q_3\}$, 
the connection $\nabla$ acting on $E$ is defined globally. 
\end{proof}

\begin{remark}
By Definition \ref{2023_7_2_11_52}
in Section \ref{2023_7_4_13_59},
we have trivializations of $E$.
On $C \setminus \{ t_1,t_2 \}$, 
the trivializations in
Definition \ref{2023_7_2_11_52} 
coincide with 
the trivializations described in the present section.
We have defined the trivialization in
Definition \ref{2023_7_2_11_52} at $t_i$ $(i=1,2)$
so that the residue matrix (respectively, the polar part in the reduced case) is a diagonal matrix.
On the other hand, 
by the trivializations described in the present section, 
the residue matrix at $t_i$ $(i=1,2)$ (respectively, the polar part) is not a diagonal matrix.
The reason why the residue matrix at $t_i$ $(i=1,2)$ is a diagonal matrix
is that 
the corresponding description of the variation \eqref{2023_7_4_14_16}
satisfies the compatibility conditions of the quasi-parabolic structure in $\mathcal{F}^0$
and $\mathcal{F}^1$
of \eqref{2023_7_4_14_18}.
On the other hand, now we are interested in behavior of the connection $\nabla$
around $q_j$ $(j=1,2,3)$.
So now we do not consider the diagonalization of the residue matrices at $t_i$ $(i=1,2)$ (respectively, of the polar part when $D$ is reduced).
\end{remark}

\subsection{Canonical coordinates}

We calculate the canonical coordinates introduced 
in Section \ref{subsect:Canonical_Coor}.
Since $\beta|_{U_0}= \omega_{12} \otimes \omega_{21}$,
we have that 
$$
\mathrm{res}_{q_j} (\beta) =  \frac{\zeta_j  }{(u_j-t)v_j} \operatorname{d}\!x_1.
$$
We have $\delta_{\alpha_{q_j}} = \omega_{22}$.
The connection matrix $\delta_{\alpha_{q_j}}$ has the following expansion at $q_j$:
$$
\begin{aligned}
\delta_{\alpha_{q_j}} 
&= \Bigg(  \frac{1}{x_1 - u_j} 
- \frac{K'(u_j)}{4v_j^2}  +
\sum_{j'\in\{1,2,3\}\setminus \{j\}}\frac{1}{2} \cdot
\frac{v_j+v_{j'}}{u_j-u_{j'}} \cdot \frac{1}{v_j}  \\
&\qquad + \left( \frac{B_1 + B_2 v_j }{u_j-t}  + B_3 \right) \frac{1}{v_j}
+O(x_1 -u_j)
\Bigg) \operatorname{d}\!x_1.
\end{aligned}
$$
Here we set
$K(x_1) := x_1(x_1-1)(x_1 - \lambda)$.
So the accessory parameter $\tilde p_j $ (for $j=1,2,3$) in Definition \ref{2023_7_12_23_06}
is as follows:
$$
\begin{aligned}
\tilde p_j &= \mathrm{res}_{q_j} (\beta)
+ \left( \delta_{\alpha_{q_j}}  - \frac{\operatorname{d}\!x_1}{x_1 - u_j}  \right)|_{q_j}  
= p_j \operatorname{d}\!x_1,
\end{aligned}
$$
where
$$
\begin{aligned}
p_j&= \frac{\zeta_j}{(u_j-t)v_j} - \frac{K'(u_j)}{4v_j^2}  +
\sum_{j'\in\{1,2,3\}\setminus \{j\}}\frac{1}{2} \cdot
\frac{v_j+v_{j'}}{u_j-u_{j'}} \cdot \frac{1}{v_j}  \\
&\qquad + \left( \frac{B_1 + B_2 v_j }{u_j-t}  + B_3 \right) \frac{1}{v_j}  .
\end{aligned}
$$
Notice that $B_1$ and $B_2$ are determined by Lemma \ref{2023_7_21_17_07} 
and $B_3$ is determined by Lemma \ref{2023_7_21_17_07(2)}.
Notice that $B_3$ depends on $\zeta_1,\zeta_2$ and $\zeta_3$.
So we have a map
$$
(E,\nabla) \longmapsto 
(u_1,u_2,u_3,\zeta_1,\zeta_2,\zeta_3) 
\longmapsto 
(u_1,u_2,u_3, p_1,p_2, p_3).
$$
The symplectic structure is $\sum_{j=1}^3 \operatorname{d}\! p_j \wedge \operatorname{d}\! u_j$ 
by Theorem \ref{2023_8_22_12_09}.

\section{Canonical coordinates revised and another proof for birationality}\label{Sec:Higgs}

In this section, we will give another proof of Proposition \ref{prop:birational}.
For simplicity, we will consider the cases where $D$ is a reduced effective divisor. 
Let $(E, \nabla) \in M_X^0 $ be a connection on a fixed irregular curve 
$X =(C,   D,  \{ z_i \}_{i \in I}, \{ \theta_i \}_{i \in I},  \theta_{res})$ 
with genericity conditions as before.  

 We set $D = t_1 + \cdots + t_n$  and the connection is given  by
$$
\nabla \colon  E \lra E \otimes \Omega^1_C(D).   
$$
In this section, we assume that $g= g(C) \geq 1$ and $n \geq 1$ as in the previous sections.
Moreover if $g=g(C) =1$, we assume that $n \geq 2$.  

Note that we have the unique extension 
\begin{equation}\label{eq:ext_6}
0 \lra  \cO_C \lra E \lra L_0 \lra 0
\end{equation}
with $L_0 = \det (E)$.  
Moreover for $(E, \nabla) \in M^0_X$ we have
$\deg L_0= 2g-1$ and $\dim_{\mathbb{C}} H^0(C, E)=1$.   
Then we can define  apparent singularities $q_1, \ldots, q_N \in C$ where 
$N  = 4g-3 +n$.
Since $ \deg \Omega^1_C(D) =2g-2+n \geq 2$ and $\deg L_0 =2g-1  \geq 1$,   
we see that $\dim_{\mathbb{C}} H^0(C, \Omega^1_C(D))  = g-1+n  \geq  2$ and $\dim_{\mathbb{C}}H^0(C, L_0) \geq 2g-1+1-g=g \geq 1$.  
Therefore, we can  choose $\gamma\in H^0(C,  \Omega^1_C(D)) $
and $ s \in H^0(C, L_0)$ whose zeros are given by 
$$
\{ \gamma=0 \} =\{  c_1, \ldots,  c_{2g-2+n}  \}\quad \text{and} \quad 
\{ s =0 \} = \{u_1, \ldots, u_{2g-1}\}.
$$
We assume the following genericity conditions: 
\begin{enumerate}
\item $u_{i_1} \not= u_{i_2}$ (for $ i_1\not=i_2$), 
and $c_{k_1} \not= c_{k_2}$ (for $k_1 \neq k_2$); 
\item 
$  \{ u_1, \ldots, u_{2g-1} \}  \cap \{ c_1, \ldots, c_{2g-2 + n } \}  = \emptyset$;
\item 
$\{ q_1, \ldots, q_N \} \cap \{u_1, \ldots, u_{2g-1},  c_1, \ldots, 
c_{2g-2 +n } \} = \emptyset.$
\end{enumerate}

Set 
$$
U_0 = C \setminus \{u_1, \ldots, u_{2g-1}, c_1, \ldots, c_{2g-2 + n } \}.
$$
Moreover we take a small analytic neighborhood $ U_{i} $ of $u_i$ for $1 \leq i \leq 2g-1$
and  $U_{2g-1+k} $ of $c_k$ for $1 \leq k \leq 2g-2+n$.  For $i = 1,\ldots, 4g-3+n$, 
we can identify $U_i$ with a unit disc $\Delta = \{  z \in \C \mid |z| <  1 \}$   
with the origin corresponding to $u_i$ ($1 \leq i \leq 2g-1$) and
$c_{i-2g+1}$ ($2g \leq i \leq 4g-3+n$).  
We can assume that $U_{i_1} \cap U _{i_2} = \emptyset $ for $ i_1 \not=i_2$, $i_1, i_2 \geq 1$.  
Note that since $U_0$ is an affine variety and $U_0 \cap U_i \cong  \Delta  \setminus \{ 0 \} $
for $ i=1, \ldots, 4g-3+n $,
the covering
$C =   U_0 \cup U_1 \cup \cdots \cup U_{4g-3+n}  $ gives a Stein covering of $C$.  
 For $0 \leq i \leq 4g-3+n $, 
we have nonzero sections $\be_1^{(i)} \in \mathcal{O}_{U_i}, \be_2^{(i)}  \in (L_0)_{|U_i}$ giving trivializations of $E$ on $U_i$ respectively:  
 $$
 E_{|U_i} \simeq \cO_{|U_i}\be_1^{(i)} \oplus \cO_{|U_i} \be_2^{(i)}.
 $$
 Moreover we have a transition matrix $H_{0i}$ on $U_{0} \cap U_{i}$ of the form 
 \begin{equation}
 H_{0i}= \begin{pmatrix}
1 & h_{0i} \\
0 & g_{0i} 
\end{pmatrix}
\end{equation}
satisfying 
\begin{equation}\label{eq:trans_6}
(\be_1^{(i)}, \be_2^{(i)} ) =  (\be_1^{(0)}, \be_2^{(0)}) H_{0i} =(  \be_1^{(0)}, h_{0i} \be_1^{(0)}+   g_{0i}  \be_2^{(0)} ) .  
\end{equation}
Here $\{h_{0i}\}_i  \in \operatorname{Ext}^1(L_0, \cO_C) \cong H^1(C, L_0^{-1})$ corresponds to the 
extension class of (\ref{eq:ext_6})   and 
$\{  g_{0i}  \}_i  \in H^1(C, \cO_C^{*})$  gives the transition function of $L_0 = \det(E)$.
With these trivializations we have connection matrices $A^{(i)}$:
\begin{equation}\label{eq:conn_6}
\nabla (\be_1^{(i)}, \be_2^{(i)})  =(\be_1^{(i)}, \be_2^{(i)})  A^{(i)}
\end{equation}
of the form 
\begin{equation}
A^{(i)} = \begin{pmatrix} a_{11}^{(i)} \gamma_i   &  a_{12}^{(i)} \gamma_i \\
a_{21}^{(i)} \gamma_i   &  a_{22}^{(i)}  \gamma_i  
\end{pmatrix}.
\end{equation}
Here $a_{kl}^{(i)} \in \Gamma(U_i, \cO_{U_i})$
and 
$\gamma_i  \in  \Gamma(U_i, \Omega^1_{U_i}(D))$.  
We set $\gamma_0 =\gamma_{|U_0}$  as above.

From  (\ref{eq:trans_6}) and (\ref{eq:conn_6}), we can verify the following
\begin{lem}
For $1 \leq i \leq 4g-3 + n$,  on $U_{0} \cap U_{i}$, we have
\begin{equation}
A^{(i)}  = H_{0i}^{-1} A^{(0)} H_{0i}  +  H_{0i}^{-1}  \operatorname{d}\! H_{0i}.
\end{equation}
Specifically, we have the following identities:
\begin{equation} \label{eq:app_6}
a_{21}^{(i)} \gamma_i =  a_{21}^{(0)}  \gamma_0  g_{0i}^{-1}; \quad  \text{and}
\end{equation}
\begin{equation} \label{eq:dual_6}
a_{22}^{(i)} \gamma_i =   a_{22}^{(0)}  \gamma_0 +  a_{21}^{(0)} \gamma_0  h_{0i} g_{0i}^{-1} 
+ \frac{\operatorname{d}\! g_{0i}}{g_{0i}} .
\end{equation}
\end{lem}

The identity (\ref{eq:app_6}) shows that  $\{ a_{21}^{(i)}\gamma_i \}_{i} $  defines a section of $H^0(C, \Omega_1(D) \otimes L_0)$ and the zeros of this  section are nothing but the 
apparent singularities, that is, 
$$
\{ a_{21}^{(0)}\gamma_0 =0 \} = \{ q_1, \ldots,  q_N \}  \subset U_0.
$$
Evaluating the identity (\ref{eq:dual_6}) at $q_j$
($j=1,\ldots,N$), we then have 
\begin{equation}\label{eq:twist_6}
  ( a_{22}^{(i)}  \gamma_i)_{q_j}  = 
   (a_{22}^{(0)}  \gamma_0)_{q_j} + 
  \left( \frac{\operatorname{d}\! g_{0i}}{g_{0i}}\right)_{q_j} 
\end{equation}
Noting that the cohomology class of the cocycle 
$\left\{ \frac{\operatorname{d}\!g_{0i}}{g_{0i}} \right\}_i$ 
corresponds to $ c_1(L_0)$, from (\ref{eq:twist_6}),  we have the following 
\begin{prop}
For each $ 0 \leq j \leq N$,  
the data $(E, \nabla) \in M^0_X$ defines $N$ points  $(q_j, \tilde{p}_j)$ on 
the total space of  $ \boldsymbol{\Omega}(L_{\delta}, D) $  by the formula 
\begin{equation}
\tilde{p}_j  =   (a_{22}^{(0)}  \gamma_0)_{q_j} \in \Omega^1_C(L_{\delta}, D)_{|q_j}
\end{equation}
\end{prop}
The above definition of $\tilde{p}_j$ does not depend on the choice of the sections $s
\in H^0(C, L_0) $ and 
$\gamma\in H^0(C, \Omega^1_C(D)) $ and defines the same map as in Definition~\ref{2023_7_12_23_06}:
\begin{equation}
 f_{\mathrm{App}}\colon M^0_X  \lra   
  \operatorname{Sym}^N (\boldsymbol{\Omega}(L_{\delta}, D)) . 
\end{equation}

Now we consider $q_j$ as a local coordinate  near $q_j$  and we write 
$\gamma =  c(q_j)  \operatorname{d}\! q_j$ for some local holomorphic function $c(q_j)$. Then we have 
$$
\tilde{p}_j  =  p_j    \operatorname{d}\! q_j   
$$ 
with 
$$
p_j = a_{22}^{(0)}(q_j) c(q_j)  .
$$
As we have proved in Theorem~\ref{2023_8_22_12_09}, the map  $f_{\mathrm{App}}$ is symplectic.

\subsection{From a connection to a Higgs field}

Keeping the notation, let us consider the section $s   \in H^0(C, L_0)$ as before, and set $s^{(0)} = s$. 
Take  trivialization of $L_{0|U_i}$ over $U_i$  we have a holomorphic function $s^{(i)}  \in 
\Gamma(U_i, \cO_{U_i})$ such that 
$$
s^{(0)}=  g_{0i} s^{(i)}.
$$
Note that $s^{(i)}$ has zeros at $ u_i \in U_i$ for $1 \leq i \leq 2g-1$.  
Set $D(s)= u_1+\cdots + u_{2g-1}$. We can show the following
\begin{lem}
There exists  a connection 
$$
\nabla_{1}\colon  E \lra E \otimes \Omega_C^1(D(s))
$$
such that for each $0\leq i \leq N = 4g -3 + n$, on $U_i$ it has the form
$$
\nabla_1^{(i)}  =  \operatorname{d} + \,  S^{(i)} =\operatorname{d} +  
\begin{pmatrix}   0  &  -\frac{\beta_i}{s^{(i)} } \\
0  &  -\frac{\operatorname{d}\!s^{(i)}}{s^{(i)}}  
\end{pmatrix}
$$
with respect to the trivialization $(\be_1^{(i)}, \be_2^{(i)})$. 
Here $\beta_i \in \Gamma(U_i,\Omega^1_{U_i})$.  
\end{lem}

\begin{proof}
Since $s^{(0)}= g_{0i}  s^{(i)}$, one has 
$$
\frac{\operatorname{d}\! s^{(0)}}{s^{(0)}} = \frac{\operatorname{d}\!g_{0i}}{g_{0i}}
+  \frac{\operatorname{d}\! s^{(i)}}{s^{(i)}} 
$$
on $U_{0i} = U_0 \cap U_i$.  

Since $\{ h_{0i}  \}_i$ is a class in $H^1(C, L_0^{-1})$ and $s \in H^0(C, L_0)$, 
the class  $\{ s^{(i)} h_{0i} \}_i $ defines a class in $H^1(C, \cO_{C})$.  
Then, by the Hodge theory, the derivative   $\{  \operatorname{d} (s^{(i)} h_{0i} )  \}_i \in H^1(C, \Omega^1_C)$ vanishes, so 
there exist $\beta_i \in \Gamma(U_i, \Omega^1_{U_i})$  such that 
$$
\operatorname{d} (s^{(i)}  h_{0i} ) =  \beta_0 - \beta_i.  
$$
Choose such $ \beta_i $'s for the formula.  Then we have 
\begin{equation} \label{eq:sp_6_0}
 \operatorname{d}\! h_{0i} 
 = - h_{0i} \frac{\operatorname{d}\!s^{(i)}}{s^{(i)} } + g_{0i} \frac{\beta_0}{s^{(0)}}-
 \frac{ \beta_i}{s^{(i)}}.
\end{equation}

The compatibility condition for connection matrices $S^{(i)}$ is  
\begin{equation} \label{eq:sp_6}
S^{(i)} = H_{0i}^{-1} S^{(0)} H_{0i} +   H_{0i}^{-1} \operatorname{d}\! H_{0i}.
\end{equation}
The right hand side of (\ref{eq:sp_6}) is  
\begin{equation} \label{eq:iden_6}
\begin{pmatrix}
0  &  -g_{0i} \frac{\beta_0}{s^{(0)}}  +  h_{0i} \left(\frac{\operatorname{d}\! s^{(0)}}{s^{(0)}}-  
\frac{\operatorname{d}\! g_{0i}}{g_{0i}} \right) + \operatorname{d}\! h_{0i}  \\
0  &  -\frac{\operatorname{d}\! s^{(0)}}{s^{(0)}}  +  \frac{\operatorname{d}\! g_{0i}}{g_{0i}}
\end{pmatrix} =
\begin{pmatrix}
0  &  -g_{0i} \frac{\beta_0}{s^{(0)}}  +  h_{0i}   \frac{\operatorname{d}\! s^{(i)}}{s^{(i)}}  + \operatorname{d}\! h_{0i}  \\
0  &  -\frac{\operatorname{d}\! s^{(i)}}{s^{(i)}} 
\end{pmatrix}
\end{equation}
By using (\ref{eq:sp_6_0}), the right hand side of (\ref{eq:iden_6}) becomes 
$$
\begin{pmatrix}
0  &  - \frac{\beta_i}{s^{(i)}} \\
0 & -\frac{\operatorname{d}\! s^{(i)}}{s^{(i)} }
\end{pmatrix}
$$
as desired.
\end{proof}

For any $(E, \nabla) \in M^0_X$,  the difference 
$$
\nabla- \nabla_1\colon E \lra  E  \otimes \Omega^1_C(D+D(s))
$$
defines an $\cO_C$-homomorphism, that is a rational Higgs fields on $E$.
We reprove Proposition \ref{prop:birational}.
\begin{thm}\label{thm:birational}
For a generic $(E, \nabla) \in M_X^0$,  
the point $(q_j,  \tilde{p}_j)_{j=1,\ldots,N}  \in  
 \mathrm{Sym}^N(\boldsymbol{\Omega}(L_{\delta},D)) $  determines $ (E, \nabla) $.  So the map $f_{\mathrm{App}}$ is birational.  
\end{thm}

\begin{proof}
Consider the Higgs field 
$$
\Phi=\Phi_{\nabla}  = \nabla -  \nabla_1
\colon E  \lra  E \otimes \Omega^1_C(D + D(s))
$$
where $D = t_1 + \cdots  + t_n $  and $D(s) =  u_1 +  \cdots + u_{2g-1}$  as in the notation 
above.   For generic $(E, \nabla)$, we may assume that  apparent singularities  $q_1, \ldots, q_N$ of $(E, \nabla)$ are distinct each other and are disjoint from  $D$  and $D(s)$.  
Following \cite[\S 3]{BNR}, we will consider the spectral curve of $\Phi$.  
On $U_i$, we have 
\begin{equation}\label{eq:higgs}
\Phi_{i}  =  A^{(i)} -  S^{(i)}  =  \begin{pmatrix}
a^{(i)}_{11} \gamma_i  &  a^{(i)}_{12} \gamma_i + \frac{\beta_i}{s^{(i)}} \\
a^{(i)}_{21} \gamma_i   & a^{(i)}_{22} \gamma_i + \frac{\operatorname{d}\!s^{(i)}}{s^{(i)}}
\end{pmatrix}.
\end{equation}
The spectral curve   $C_s$ can be defined in the total space of $\boldsymbol{\Omega}(D+D(s))$ of the line bundle  $\Omega^1_C(D+D(s))$ by 
\begin{equation}\label{eq:spc}
C_s=\{  x^2  -  b_1  x  -  b_2  = 0   \}  \subset  \boldsymbol{\Omega}(D+D(s)).
\end{equation}
Here     $x$ denotes the tautological section of $\boldsymbol{\Omega}(D+D(s))$ and  from (\ref{eq:higgs}) we see that
\begin{eqnarray}
b_1   &=& \tr (\Phi_i)     \in H^0(C, \Omega^1_C(D+D(s)),   \\    
b_2  &=&  \det (\Phi_i)   \in H^0(C,  (\Omega^1(D))^2(D(s))) \subset H^0(C,  (\Omega^1(D+D(s)))^2)) .
\end{eqnarray}

Therefore  the  dimension of the family of spectral curves in (\ref{eq:spc})  is  given by 
\begin{eqnarray*}
\dim H^0( C, \Omega^1_C(D+D(s))) +   \dim H^0( C, (\Omega^1_C(D))^2(D(s)))  &&   \\
=  2g-2+n+2g-1 + 1-g  +  2(2g-2+n)+2g-1+ 1-g &=& 8g-6+3n .   \\
\end{eqnarray*}
On the other hand, by definition, $\Phi$  (and thus the spectral curve  $C_s$)  is constrained by the following conditions.
\begin{enumerate}
\item At $t_i, i=1, \ldots, n$, $\Phi$ has eigenvalues fixed by data $X$. 
These impose $2n -1$ conditions on $C_s$  because of the Fuchs relation.
\item At $u_k$,  $k=1, \ldots,2g-1$,  take a local coordinate $z_k  $ such that 
$z_k(u_k)=0$.  Then 
$\Phi$  has the following form near $z_k=0$
$$
\Phi  =  \begin{pmatrix}
0  &   \frac{\beta_i(0)}{z_k}  \\
0  &   \frac{\operatorname{d}\!z_k}{z_k} 
\end{pmatrix} +  \mbox{holomorphic}.
$$
Then eigenvalues of the residue matrix are $0,  1$, so 
gives  conditions on $C_s$.  Then totally we have $2 \times (2g-1)$
conditions on $C_s$.  
\item  At $q_j ,  j=1, \ldots, N$,  
the points $(a_{22}^{(0)}  \gamma_0)_{q_j}+ (\frac{ds^{(0)}}{s^{(0)}})_{q_j} 
=  \tilde{p}_j  +  c_j     \in \boldsymbol{\Omega}(D+D(s))_{|q_j}$  lie on 
the spectral curve $C_s$.  Here we set $ c_j=(\frac{ds^{(0)}}{s^{(0)}})_{q_j}$.   (Note that we assume that $q_j \in U_0$ for all $j$, and $s^{(0)}$ is a fixed nonzero section of $L_0$ on $U_0$.)
These give $N=4g-3+n$ conditions. 
\end{enumerate}

For generic choice of $q_1, \cdots, q_N$  and $s  \in H^0(C, L_0)$, we can see using the method of Lemma~\ref{lem:independence} and  Proposition~\ref{prop:birational}
that these conditions are independent, so we obtain a 
total of  
$$
2n-1 + 2(2g-1) + (4g-3 + n) = 8g+3n-6
$$
conditions, so these determine the spectral curve $C_s$ uniquely.  

Now the divisor  $ \mu=  \sum_{j=1}^N ( \tilde{p}_j+c_j)  +  
\sum_{k=1}^{2g-1} (1_k)$ determines the rank 1 sheaf 
 $\cO_{C_s}(\mu)$  
where $ (1_k) \in C_s $ denotes the point over $u_k$ corresponding to the eigenvalue $1$ of the residue of $\Phi$ at $u_k$.   
Consider the following exact sequence
$$
0 \lra \cO_{C_s} \lra \cO_{C_s}(\mu) \lra  \cO_{\mu} \lra  0
$$
By the morphism $\pi\colon C_s \lra C$, we can take the push forward of this sequence to obtain
$$
0 \lra \pi_{\ast} \cO_{C_s} \lra \pi_{\ast}  \cO_{C_s}(\mu) \lra \pi_{\ast}\cO_{\mu} \lra 0, 
$$
and from \cite[Proposition 3.6]{BNR}, we see that this turns into the following  
$$
0 \lra \cO_{C} \oplus (\Omega^1_{C}(D+D(s)))^{-1} \lra  E  \lra \cO_{q_1+\cdots+q_N+u_1+\cdots u_{2g-1}}
 \lra 0.
$$
Moreover we have a natural morphism 
$ x :\cO_{C_s} \lra \cO_{C_s} \otimes \pi^{\ast}(\Omega^1_{C}(D+D(s)))$ which determines $\Phi:E \lra E \otimes \Omega^1_{C}(D+D(s)) $ by \cite[Proposition 3.6]{BNR}. Thus 
$(\pi\colon C_s \lra C,  \cO_{C_s}(\mu) )$ determines $(E,  \Phi)$ 
uniquely by \cite[Proposition 3.6]{BNR}.  
Therefore  $E$ and 
$\nabla = \Phi + \nabla_1$ is determined uniquely.
\end{proof}

\end{document}